\documentclass[11pt]{amsart}

\usepackage{subcaption}
\usepackage{subfiles}
\usepackage{comment}
\usepackage{float}
\usepackage{marginnote}
\usepackage{tabu}
\usepackage[autostyle]{csquotes}
\usepackage{euscript}
\usepackage{mathdots}
\usepackage[dvipsnames]{xcolor}
\usepackage{graphicx}			
\usepackage{amssymb}
\usepackage{mathrsfs}
\usepackage{amsthm}
\usepackage{amsmath}
\usepackage{stmaryrd}
\usepackage{tikz}
\usepackage{tikz-cd}
\usepackage{accents}
\usepackage{upgreek}
\usepackage{enumerate}
\usepackage{bm}
\usepackage{mathtools}
\usepackage{thmtools}
\usetikzlibrary{patterns}
\usepackage[all]{xy}
\usepackage{caption}
\usepackage{url}
\usepackage{float}
\usepackage{todonotes} 
\usepackage{colonequals}
\usepackage{bbm}
\usepackage{longtable}
\usepackage[full]{textcomp}
\usepackage[cal=cm]{mathalfa}
\usepackage{xparse}
\usepackage{comment}
\usepackage{cite}
\usepackage{enumitem}
\usepackage{caption}
\usetikzlibrary{fit,backgrounds}

\usetikzlibrary{calc}
\usetikzlibrary{fadings}
\usetikzlibrary{decorations.pathmorphing}
\usetikzlibrary{decorations.pathreplacing}
\usepackage{tikz,tikz-cd,tikz-3dplot}
\usepackage{pgfplots}
\usetikzlibrary{arrows,shadows,positioning, calc, decorations.markings, 
	hobby,quotes,angles,decorations.pathreplacing,intersections,shapes}
\usepgflibrary{shapes.geometric}
\usetikzlibrary{fillbetween,backgrounds}

\usepackage[margin=0.9in]{geometry}

\tikzset{
	commutative diagrams/.cd, 
	arrow style=tikz, 
	diagrams={>=stealth}
}
\tikzset{
	arrow/.pic={\path[tips,every arrow/.try,->,>=#1] (0,0) -- +(0,4pt);},
	pics/arrow/.default={triangle 90}
}
\tikzset{->-/.style={decoration={
			markings,
			mark=at position .6 with {\arrow{latex}}},postaction={decorate}}
}
\tikzset{
	c/.style={every coordinate/.try}
}

\theoremstyle{theorem}

\theoremstyle{theorem}

\theoremstyle{theorem}

\makeatletter
\def\@tocline#1#2#3#4#5#6#7{\relax
	\ifnum #1>\c@tocdepth 
	\else
	\par \addpenalty\@secpenalty\addvspace{#2}%
	\begingroup \hyphenpenalty\@M
	\@ifempty{#4}{%
		\@tempdima\csname r@tocindent\number#1\endcsname\relax
	}{%
		\@tempdima#4\relax
	}%
	\parindent\z@ \leftskip#3\relax \advance\leftskip\@tempdima\relax
	\rightskip\@pnumwidth plus4em \parfillskip-\@pnumwidth
	#5\leavevmode\hskip-\@tempdima
	\ifcase #1
	\or\or \hskip 1em \or \hskip 2em \else \hskip 3em \fi%
	#6\nobreak\relax
	\dotfill\hbox to\@pnumwidth{\@tocpagenum{#7}}\par
	\nobreak
	\endgroup
	\fi}
\makeatother

\newcounter{marginnote}
\DeclareMathAlphabet{\mathpzc}{OT1}{pzc}{m}{it}

\usepackage[backref=page]{hyperref}
\hypersetup{
  colorlinks   = true,          
  urlcolor     = olive,          
  linkcolor    = olive,          
  citecolor   = olive             
}
\tikzset{
  vertex/.style={circle, fill=black, inner sep=2pt}
}

\usepackage{thmtools}
\usepackage{cleveref}

\theoremstyle{theorem}
\newtheorem{theorem}{Theorem}[section]
\newtheorem{claim}[theorem]{Claim}
\newtheorem{conjecture}[theorem]{Conjecture}
\newtheorem{corollary}[theorem]{Corollary}
\newtheorem{lemma}[theorem]{Lemma}
\newtheorem{proposition}[theorem]{Proposition}

\newtheorem*{theoremA}{Theorem A}
\newtheorem*{theoremB}{Theorem B}

\newtheorem*{theoremD}{Theorem D}
\newtheorem*{theoremE}{Theorem E}

\newtheorem*{corollaryC}{Corollary C}

\newtheorem*{corollaryF}{Corollary F}
\newtheorem*{corollaryG}{Corollary G}
\newtheorem*{corollaryH}{Corollary H}

\newtheorem*{conjectureI}{Conjecture I}

\theoremstyle{definition}
\newtheorem{observation}[theorem]{Observation}
\newtheorem{remark}[theorem]{Remark}

\newtheorem*{runningexample*}{Running example}

\newtheorem*{aside*}{Aside}

\newtheorem{definition}[theorem]{Definition}
\newtheorem{example}[theorem]{Example}

\newtheorem{proposition-definition}[theorem]{Proposition-Definition}
\newtheorem{theorem-definition}[theorem]{Theorem/Definition}
\newtheorem{question}[theorem]{Question}

\newtheorem{thmdef}{Theorem/Definition}[section]

\newcommand{\xdashleftrightarrow}[2][]{\ext@arrow 3359\leftrightarrowfill@@{#1}{#2}}

\newcommand{\bcd}{\begin{center}\begin{tikzcd}}
		\newcommand{\ecd}{\end{tikzcd}\end{center}}

\newcommand{\mgn}{\mathcal{M}_{g,n}}
\newcommand{\mgbar}{\overline{\mathcal{M}}_g}
\newcommand{\mgnbar}{\overline{\mathcal{M}}_{g,n}}

\newcommand{\Pone}{\mathbb{P}^1}
\newcommand{\lk}{\mathrm{Lk}}
\newcommand{\inde}{\mathrm{Ind}}
\newcommand{\wpsi}{\widehat{\psi}}

\newcommand{\mgraphbar}{\overline{\mathcal{M}}_{g,G}}
\newcommand{\mzerog}{M_{0,G}}
\newcommand{\lef}{\mathbb{L}}
\newcommand{\grogrp}{K_0(\mathrm{Var}_{k})}
\newcommand{\tw}{\chi^{\mathrm{top-wt}}}
\newcommand{\bound}{\Delta_{0,G}}
\newcommand{\pvir}{P_{M_{0,G}}^{\mathrm{vir}}}
\newcommand{\nkappa}{\Tilde{\kappa}}

\crefname{equation}{eq.}{eqs.}
\crefname{eqnarray}{eq.}{eqs.}
\crefname{conjecture}{conjecture}{conjectures}
\crefname{lemma}{lemma}{lemmas}
\crefname{theorem}{theorem}{theorems}
\crefname{claim}{claim}{claims}
\crefname{remark}{remark}{remarks}
\crefname{proposition}{proposition}{propositions}
\crefname{section}{section}{sections}
\crefname{appendix}{appendix}{appendices}
\crefname{corollary}{corollary}{corollaries}
\crefname{figure}{figure}{figures}
\crefname{table}{table}{tables}
\crefname{example}{example}{examples}
\crefname{assumption}{assumption}{assumptions}
\crefname{definition}{definition}{definitions}
\crefname{innercustomthm}{theorem}{theorems}
\crefname{innercustomconj}{conjecture}{conjectures}

\setlist[enumerate,1]{label=(\roman*),itemsep=0.9ex}
\setlist[itemize]{itemsep=0.9ex}

\begin{document}
    \title{Deletion--contraction properties of graphically stable spaces}
	\author{Andy Xiaoan Yang}
	
\begin{abstract}
Graphically stable spaces $\mgraphbar$ parametrize marked nodal curves whose permitted collisions of markings are determined by a graph. We study intersection numbers of $\psi$-classes on $\mgraphbar$, as well as the classes $[M_{g,G}]$ in the Grothendieck ring of varieties. In both settings, we show that the geometry is governed by an underlying graphical structure, expressed through deletion--contraction relations. As consequences, we derive string and dilaton equations and express several families of $\psi$-class integrals in terms of the chromatic polynomial. We also express the Grothendieck class of $M_{0,G}$ over an arbitrary field in terms of the chromatic polynomial and identify various Euler characteristics with combinatorial quantities. Along the way, we obtain a new formula for Crapo's $\beta$-invariant of graphs. Finally, we extend these relations to genus one and, under a chromatic condition, to higher genus.
\end{abstract}

	\maketitle
    \setcounter{tocdepth}{1}
	\tableofcontents
\section*{Introduction}
 For a finite graph $G$ with vertex set $\{1,...,n\}$ and a non-negative integer $g$, the graphically stable space 
\[\mgraphbar\]
provides an alternative modular compactification of $\mathcal{M}_{g,n}$. In contrast to the Deligne--Mumford compactification $\mgnbar$, graphically stable spaces allow collections of markings to coincide whenever the corresponding vertices in $G$ are pairwise non-adjacent. These spaces are closely related to Hassett's moduli spaces of weighted stable curves \cite{Hassett}; in particular, when $G$ is the complete graph $K_n$, we recover $\mgnbar$. In genus $0$, graphically stable spaces have been studied in \cite{Fry_trop, Fry, SRI}.

The central insight of this paper is that deletion--contraction relations should serve as a natural organizing principle for the study of graphically stable spaces in all genera, just as they do in graph theory. This principle appears in two settings: $\psi$-class intersection numbers and the Grothendieck classes of the open loci $M_{g,G}$.
\vspace{-0.5em}
\subsection{Deletion--contraction for $\psi$-class integrals}Among the most studied tautological classes on moduli spaces of curves are the cotangent classes, or $\psi$-classes, assigned to the markings. We are interested in their intersection numbers. 

In genus $0$ and $1$, we impose a mild stability condition on the graph $G$; graphs satisfying this condition will be called $g$-stable (Definition \ref{def:graph_stab}) and throughout this section we will assume $G$ is $g$-stable unless otherwise specified.
\begin{theoremA}[Theorem \ref{thm_del_con_psi}]\phantomsection\label{thm:A}
Let $e \in E(G)$ be an edge with endpoints $a$ and $b$. Suppose we are also given non-negative integers $k_1,...,k_{|V(G)|=n}$ such that $\sum k_i = 3g-3+n$, and $G-e$, $G/e$ are both $g$-stable graphs.
    \begin{enumerate}
        \item If $(k_a,k_b) \neq (0,0)$, we have
        \begin{align*}
        \int_{\overline{\mathcal{M}}_{g,G}} \psi_1^{k_1} \cdots \psi_n^{k_n} = \int_{\overline{\mathcal{M}}_{g,G - e}} \psi_1^{k_1} \cdots \psi_n^{k_n} + \int_{\overline{\mathcal{M}}_{g,G/e}} \psi_1^{k_1} \cdots \psi_{ab}^{k_a+k_b-1} \cdots \psi_n^{k_n}. 
    \end{align*}
    \item If $(k_a,k_b) = (0,0)$, we have
    \begin{align*} 
        \int_{\overline{\mathcal{M}}_{g,G}} \psi_1^{k_1} \cdots \psi_n^{k_n} = \int_{\overline{\mathcal{M}}_{g,G - e}} \psi_1^{k_1} \cdots \psi_n^{k_n} + \int_{\overline{\mathcal{M}}_{g,G \odot e}} \psi_1^{k_1} \cdots \psi^0_{ab}\psi^0_{\widetilde{ab}}\cdots \psi_n^{k_n}
    \end{align*}
    where $G \odot e$ is the graph obtained from attaching a leaf to the vertex $ab$ of $G/e$. The vertex of this leaf is labelled as $\widetilde{ab}$.
    \end{enumerate} 
\end{theoremA}
This structural property provides a unifying mechanism for the study of $\psi$-class integrals on $\mgraphbar$ and yields new results while recovering several known results from the literature as special cases. For instance, we derive graphical analogues of the classical string and dilaton equations for $\mgnbar$ \cite{Witten1991}. While these equations have been generalized in various directions, our framework appears to be the first to extend them uniformly to a large class of alternative compactifications of $\mathcal{M}_{g,n}$.
\begin{theoremB}[Theorem \ref{dilaton} \& \ref{string_complete}]\label{thm:B}Label the neighbours of the vertex $n$ as $u_1,...,u_d$, and the corresponding incident edges to the vertex $n$ as $e_1,...,e_d$.
\begin{enumerate}
    \item ($G$-dilaton equation)
    \begin{align*}
        \int_{\overline{\mathcal{M}}_{g,G}} \psi_1^{k_1} \cdots \psi_{n-1}^{k_{n-1}}\psi_n = (2g-2)\int_{\overline{\mathcal{M}}_{g,G - n}} \psi_1^{k_1} \cdots \psi_{n-1}^{k_{n-1}} + \sum_{i =1}^{d} \int_{\overline{\mathcal{M}}_{g,(G-e_1-\cdots - e_{i-1})/e_{i}}} \psi_1^{k_1} \cdots \psi_{nu_i}^{k_{u_i}} \cdots \psi_{n-1}^{k_{n-1}}.
    \end{align*}
    Here $G-n$ denotes the graph obtained from $G$ after deleting the vertex $n$ and we set $e_{0} = \varnothing$.
    \item ($G$-string equation) If either $k_{u_i} \neq 0$ for all $i=1,...,d$, or $G$ is a complete multipartite graph, then we have 
            \begin{align*}
            \int_{\overline{\mathcal{M}}_{g,G}} \psi_1^{k_1} \cdots \psi_{n-1}^{k_{n-1}} = \sum_{i =1}^{d} \int_{\overline{\mathcal{M}}_{g,(G-e_1-\cdots - e_{i-1})/e_{i}}} \psi_{nu_i}^{k_{u_i}-1} \prod_{j \neq u_i} \psi_j^{k_j}.
        \end{align*}

\end{enumerate}
These formulas hold provided that all graphs appearing above are $g$-stable.
\end{theoremB}
A fundamental graph invariant satisfying the deletion--contraction relation is the \textit{chromatic polynomial} $\chi_G(x)$. By Theorem~\hyperref[thm:A]{A}, this allows us to express many intersection numbers on graphically stable spaces in terms of $\chi_G(x)$. Conversely, the intersection theory of $\mgraphbar$ determines the chromatic polynomial of $G$ completely (Proposition \ref{lemma_acyc_psi}).
\begin{corollaryC}\phantomsection \label{cor:C}
Given any graph $G$, not necessarily $g$-stable, we relate the following $\psi$-class integrals to the chromatic polynomial (see Section \ref{subsec_unweight_graph} for details and more examples):
\begin{align*}
    24 \cdot \int_{\overline{\mathcal{M}}_{1,G}} \psi_1 \cdots \psi_n &= (-1)^{|V(G)|-1}\left. \frac{d}{dx} \chi_G(x) \right|_{x = 0}. \\
    \int_{\overline{\mathcal{M}}_{0,G  \sqcup K_3 }} \psi_1 \cdots \psi_n &= (-1)^{|V(G)|} \chi_G(2). \\
    \int_{\overline{\mathcal{M}}_{0,G  * K_3 }} \psi_1 \cdots \psi_n &= (-1)^{|V(G)|} \chi_G(-1).
\end{align*}
\end{corollaryC}
We also study $\kappa$-classes on $\mgraphbar$ in Section \ref{sec:kappa}. In particular, arbitrary $\psi$-class integrals on $\mgraphbar$ can be reduced, via the deletion--contraction relations of Theorem~\hyperref[thm:A]{A}, to intersection numbers of certain $\nkappa$-classes on graphically stable spaces (Observation \ref{obs:kappa}). This perspective also naturally recovers relations on Deligne--Mumford spaces; for example, Proposition \ref{prop:kappa_mg2} gives:
\begin{align*}
        \int_{\overline{\mathcal{M}}_{g,2}} \psi_1^{k_1}\psi_2^{k_2} = \int_{\mgbar} \kappa_{k_1-1}\kappa_{k_2-1} + \int_{\mgbar} \kappa_{k_1+k_2-2}.
    \end{align*}

\subsection{Deletion--contraction for Grothendieck classes} We also study the class $[M_{g,G}]$ in the Grothendieck ring of varieties $K_0(\mathrm{Var}_k)$, a universal setting for cut-and-paste invariants of algebraic varieties.
\begin{theoremD}[Theorem~\ref{Gro_Class}]

\label{thm:D}The following relation holds within $K_0(\mathrm{Var}_{k})$:
    \begin{align*}
        [M_{g,G}] = [M_{g, G-e}] - [M_{g, G/e}],
    \end{align*}
    providing all three graphs are $g$-stable.
\end{theoremD}
This property also has several applications. Among them is an explicit formula for the Grothendieck class of $M_{0,G}$ in terms of the chromatic polynomial.
\begin{theoremE}[Theorem~{\ref{Gro_thm_zero}}]
\phantomsection
\label{thm:E}
    Let $G$ be a $0$-stable graph, and set
    $\lef=[\mathbb{A}^1]$. Then, we have
    \begin{align*}
        [\mzerog] = \frac{\chi_G(\lef + 1)}{\lef^3 - \lef} \in \mathbb{Z}[\lef] \subseteq K_0(\mathrm{Var}_k).
    \end{align*}
    Indeed, the divisibility $(\lef^3 - \lef)\! \mid \! \chi_G(\lef + 1)$ follows from the $0$-stability of $G$, see \eqref{eq:cancel}.
\end{theoremE}
    By specializing the identity of Theorem~\hyperref[thm:E]{E} to different fields, we obtain the following results.
\begin{corollaryF}[Corollary \ref{lhop}]\phantomsection
\label{cor:F} The following topological invariants are determined by the chromatic polynomial:
\hfill
\begin{align}
    \chi(\mzerog(\mathbb{C}))
    &= \frac{1}{2}\left. \frac{d}{dx} \chi_G(x) \right|_{x = 2}
    \label{intro_eul} \\  
    \tw(\mzerog(\mathbb{C}))
    &= -\left.\frac{d}{dx} \chi_G(x) \right|_{x = 1}
    \label{intro_tw} \\
    \chi_c(\mzerog(\mathbb{R}))
    &= \frac{1}{2} \left. \frac{d}{dx} \chi_G(x) \right|_{x = 0}
    \label{intro_real} 
\end{align}
where $\tw$ denotes the top-weight Euler characteristic associated with Deligne's weight filtration.

\end{corollaryF}
We also study the higher-genus case using the deletion--contraction relation from Theorem~\hyperref[thm:D]{D}; for instance, in genus $1$ we determine both the complex Euler characteristics and the top-weight Euler characteristics.
\begin{corollaryG}[Proposition \ref{lem_chi_one}]\phantomsection
\label{cor:G}
    Suppose $G$ has at least one vertex, we have
    \begin{align*}
        \chi(M_{1,G}(\mathbb{C})) &= - \frac{1}{12}\left. \frac{d}{dx} \chi_G(x) \right|_{x = 0} + \frac{\chi_G(1)}{3} +  \frac{\chi_G(2)}{4}+\frac{\chi_G(3)}{9}-\frac{\chi_G(4)}{48}. \\
        \tw (M_{1,G}(\mathbb{C}))& = -\frac{1}{2} \left. \frac{d}{dx} \chi_G(x) \right|_{x = 0} + \chi_G(1) -  \frac{\chi_G(2)}{4}.
    \end{align*}
\end{corollaryG}
\subsection{Deletion--contraction perspectives on existing results}

The deletion--contraction framework developed in this paper provides a unified perspective on several constructions and results in the literature, allowing us to recover, extend, or formulate precise analogues of them.

\subsubsection{Wall-crossing for $\psi$-class integrals}

The proof of Theorem~\hyperref[thm:A]{A} relies on a wall-crossing formula relating $\psi$-class integrals on $\mgraphbar$ to those on $\mgnbar$ (Theorem \ref{Theorem:wallcross}), via the natural proper birational morphism 
\[ \rho_G \colon \mgnbar \longrightarrow \ \mgraphbar.\]
This formula is adapted from the work of Alexeev--Guy \cite{AG08} and Blankers--Cavalieri \cite{BlankersRenzo}, where wall-crossing from Hassett spaces to the Deligne--Mumford compactification is studied.

At the intersection of Hassett spaces and graphically stable spaces lie the \textit{heavy/light} spaces $\overline{\mathcal{M}}_{g,K_m * E_d}$, with $m$ heavy markings and $d$ light markings. These spaces have rich geometric structures. For example, the space $\overline{\mathcal{M}}_{g,K_m * E_d}/S_d$ coincides with Marian--Oprea--Pandharipande's moduli space of stable quotients $\overline{Q}_{g,m}(\mathrm{Gr}(1,1),d)$ \cite{StableQuotient}. Moreover, $\psi$-class integrals on heavy/light spaces were studied in \cite[§4]{StableQuotient}, also using wall-crossing techniques. 

When $g=0$ and $m=2$, the heavy/light space $\overline{\mathcal{M}}_{0,K_2 * E_d}$ recovers the Losev--Manin space, first introduced in \cite{LosevManin}. Its $\psi$-class intersection theory was studied in detail by Dastidar--Ross \cite{Matriod_psi}, where it serves as a prototype for the theory of \textit{matroid $\psi$-classes}. We will review the intersection theory of $\psi$-classes over $\overline{\mathcal{M}}_{g,K_m * E_d}$ in detail in Section \ref{sec: heavy_light}.

\subsubsection{Reinke--Silversmith type integrals}\label{section:intro_rob}

Corollary~\hyperref[cor:C]{C} has intriguing connections to the results of Reinke--Silversmith \cite{RobReinke}, who study certain graph-associated integrals over $\mgnbar$. More precisely, for each vertex $i$ of $G$, let
\[
    \pi_{G,i}\colon \mgnbar \longrightarrow \overline{\mathcal{M}}_{g,|N_G[i]|}
\]
be the morphism between Deligne--Mumford spaces forgetting all markings outside the closed neighbourhood of $i$. Reinke--Silversmith consider integrals $\omega_{G,g,m}$ (Definition \ref{def_rob}) built from the classes $\pi_{G,i}^*(\psi_i)$, defined formally similar but conceptually different to the integrals appearing in Corollary~\hyperref[cor:C]{C}. It was shown in \cite{RobReinke} that the integrals $\omega_{G,g,m}$ satisfy the same deletion--contraction relation as Theorem~\hyperref[thm:A]{A} and take the same values as those in Corollary~\hyperref[cor:C]{C}.

We explore the comparison between the two settings in detail in Section~\ref{subsec_comparison_psi}, where we conjecture (Conjecture \ref{conjecture_psi} \& ~ \ref{conj:psi_2}) a precise relation that would imply that the two theories are completely analogous. 

\subsubsection{Graph-theoretic interpretations} Many of the graph invariants arising from the geometry of graphically stable spaces have been studied in graph theory and admit interpretations in terms of counting certain graph-related quantities. Assuming that $G$ is connected, we summarize these results in Table~ \ref{table:1}.
\begin{table}[h]
\centering
\renewcommand{\arraystretch}{1.8}
\begin{tabular}{|
>{\centering\arraybackslash}m{5cm}|
>{\centering\arraybackslash}m{5cm}|
>{\centering\arraybackslash}m{5cm}|}
\hline
Evaluation of $\chi_G$ & As a count of... & Connection to graphically stable spaces \\
\hline
\hline
$\chi_G(q)$, \; $q \geq 0$ & Proper $q$-colourings of $G$
& Propositions \ref{lem_two_colour}, \ref{lem_chi_one} \\
\hline
$(-1)^{|V(G)|}\chi_G(-m)$, \; $m
\geq 1$ &Pairs $(\sigma,\mathcal O)$, where
$\sigma:V(G)\to \{1,\ldots,m\}$ and $\mathcal O$ is an acyclic orientation such that, whenever $u\to v$ in $\mathcal O$, one has $\sigma(u)\geq \sigma(v)$\cite{STANLEY_Acyc}
& Propositions \ref{lemma_acyc_psi}, \ref{lem_tw_g} \\
\hline
$(-1)^{|V(G)|-1}\left. \frac{d}{dx} \chi_G(x) \right|_{x = 0}$   & Acyclic orientations of $G$ with a unique prescribed source 
\cite{CurtisZaslavsky1983} & Propositions \ref{balanced_elliptic}, \ref{lem:double_cone}, \ref{lem_chi_one} and Section~\ref{real_locus}\\
\hline
$(-1)^{|V(G)|} \left.\frac{d}{dx} \chi_G(x) \right|_{x = 1}$      & Acyclic orientations of $G$ with a unique prescribed source and a unique prescribed sink, where the source and sink are adjacent \cite{CurtisZaslavsky1983} & Section \ref{dual_complex} \\
\hline
$\left.\frac{d}{dx} \chi_G(x) \right|_{x = 2}$ & ----- & Section \ref{subsec_euler_char}\\
\hline
\end{tabular}
\caption{Graph invariants and their interpretations.}
\label{table:1}
\end{table}

\subsubsection{Sign-reversing involutions}

In \cite{SRI}, Blankers, Gillespie, and Levinson construct sign-reversing involutions (SRIs) to study the geometry of graphically stable spaces. One of their main results \cite[Theorem~3.1]{SRI} relates, for graphs $G$ with two dominant vertices, an alternating sum over the set of genus-$0$ $G$-stable graphs $\Gamma_{0,G}$, to acyclic orientations of $G$ with the two dominant vertices removed.

We generalize their result to arbitrary $0$-stable graphs as a consequence of our computation of the top-weight Euler characteristic of $\mzerog$ in Corollary~\hyperref[cor:F]{F}, \eqref{intro_tw}.

\begin{corollaryH}[Corollary~\ref{beta}]\phantomsection\label{cor:H}
For any $0$-stable graph $G$ and any edge $e$ with endpoints labelled $P$ and $Q$, we have
\[
    \sum_{T \in \Gamma_{0,G}} (-1)^{|E(T)|}
    =
    (-1)^{n-k(G)}\beta(G)
    =
    (-1)^{n-k(G)}|\mathrm{ACO}_{P,Q}(G)|.
\]
Here $\beta(G)$ is Crapo's $\beta$-invariant of $G$ \cite{CRAPO_beta}, $k(G)$ is the number of connected components of $G$, and $\mathrm{ACO}_{P,Q}(G)$ denotes the set of acyclic orientations of $G$ with prescribed unique source at $P$ and prescribed unique sink at $Q$.
\end{corollaryH}

Similar SRI techniques were used by Clader and Valverde \cite{Clader_SRI} to evaluate the analogous sum in genus one when $G=K_n$. The authors of \cite{SRI} also uses SRI methods to study $\psi$-class integrals. In particular, they obtain a combinatorial interpretation of arbitrary $\psi$-class integrals on $\overline{\mathcal{M}}_{0,G}$ when $G$ is complete multipartite \cite[Theorem~4.43]{SRI}.

\subsubsection{CHY scattering potentials}

In \cite{GraphScattering}, the authors associate to each graph $G$ a scattering potential on $M_{0,n}$, drawing on the CHY scattering formalism from particle physics \cite{CHY_scattering}. They conjecture an explicit formula for the number of critical points of this potential; see \cite[Conjecture~7.1]{GraphScattering}. 

We will show that their conjectural formula agrees with our computation of the complex Euler characteristic of $M_{0,G}$ in Corollary~\hyperref[cor:F]{F}, see Proposition \ref{prop:conj_scattering}. Using the theory of maximum likelihood degrees \cite{Huh_likelihood}, together with the work of Fry \cite{Fry}, we give an alternative proof of their conjecture when $G$ has a dominant vertex. Moreover, to prove the full conjecture, it is enough to understand when $\mzerog$ is very affine. See Section \ref{subsec_euler_char} for a detailed discussion.

Finally, we remark that there is also a connection between $\psi$-class integrals over $\overline{\mathcal{M}}_{0,n}$ considered in \cite{RobReinke} and scattering potentials/maximum likelihood degrees, see \cite[§1.6]{RobReinke}. Such integrals are, in turn, related to $\psi$-class integrals on $\overline{\mathcal{M}}_{0,G}$; see Section~\ref{section:intro_rob} \& \ref{subsec_comparison_psi}.
 
\subsubsection{Wall-crossing for Grothendieck classes}

In \cite{Sidd_TopHass}, the authors study the topology of Hassett spaces parametrizing smooth $w$-weighted curves $M_{g,w}$, with particular emphasis on the heavy/light spaces. One of their methods is to relate the Grothendieck classes of $M_{g,w}$ to those of $M_{g,n}$ \cite[Proposition~4.1]{Sidd_TopHass}. This result can be viewed as a wall-crossing formula for Grothendieck classes, analogous to the wall-crossing formula for $\psi$-class integrals. We generalize this perspective in Section \ref{sec:alternative_GroClass} and use it to give alternative proofs of {Theorem \hyperref[thm:D]{D}} and {\hyperref[thm:E]{E}}. We also use this perspective to study Euler characteristics in genus $g\geq 2$ for graphs satisfying a chromatic condition; see Section \ref{section:arb_genus}.

\subsubsection{Relation to hyperplane arrangements}\label{section_hyperplane}

When $G$ has a dominant vertex $v$, the space $\mzerog$ is closely related to complements of hyperplane arrangements in affine space. Indeed, using the $\mathrm{PGL}_2$-action, we may send the marking labelled by $v$ to $\infty \in \mathbb{P}^1$. The remaining automorphisms form the affine group $\mathrm{Aff}_1$, and we obtain an isomorphism
\[
    \mzerog \simeq \mathrm{Conf}_{G-v}(\mathbb{A}^1)/\mathrm{Aff}_1.
\]
Here $\mathrm{Conf}_{G-v}(\mathbb{A}^1)$ is the complement of the graphical hyperplane arrangement associated to $G-v$ inside $\mathbb{A}_k^{n-1}$. Thus, Theorem~\hyperref[thm:E]{E} may be viewed as an analogue of a result of Aluffi~\cite{Aluffi_Hyperplane}, who showed that the Grothendieck class of a hyperplane arrangement complement is determined by its characteristic polynomial. Similarly, the topological invariants computed in Corollary~\hyperref[cor:F]{F} are analogous to known results for hyperplane arrangements; see, for example, \cite[Theorem~4.1]{ArdilaTutteHyperplane}.

In particular, when $G$ has a dominant vertex, the theory of hyperplane arrangements allows us to compute the Poincaré polynomial of $M_{0,G}$ in Section \ref{sec:poincare_poly}. It also gives a combinatorial explanation of the real Euler characteristic in Corollary~\hyperref[cor:F]{F}, \eqref{intro_real}, via a classical result of Zaslavsky \cite{Zaslavsky1975FacingUT}; see Section~\ref{real_locus}.

Moreover, when $G$ has a dominant vertex, Khoroshkin and Lyskov recently studied the wonderful compactification \cite{Wonderful} of $\mzerog$ from an operadic point of view \cite{Khoroshkin_Lyskov, GraphicalConf}.

\subsection{Future directions} We expect the deletion--contraction philosophy to provide a fruitful framework for further study of graphically stable spaces.

\subsubsection{Weighted graph invariants}\label{section:future_weighted_graph}

One possible direction is to view $\psi$-class integrals as invariants of weighted graphs. A vertex-weighted graph is a pair $(G,\mathbf w)$, where $V(G)=[n]$ and $\mathbf w=(w_1,\ldots,w_n)\in\mathbb Z^n$. For fixed $g$, define
\[
    \Psi_g(G,\mathbf w)
    :=
    \int_{\mgraphbar}
    \psi_1^{w_1+1}\cdots \psi_n^{w_n+1}.
\]
By dimension considerations, $\Psi_g(G,\mathbf w)=0$ unless $\sum_i w_i=3g-3$ and all $w_i\geq -1$.

The deletion--contraction relation for $\psi$-class integrals can then be reformulated as a weighted deletion--contraction relation. If $e=ab$ is an edge, deletion leaves the weights unchanged, while contraction assigns the new vertex $ab$ the weight $ w_{ab}=w_a+w_b$.
Thus by Theorem~\hyperref[thm:A]{A}, when $(w_a,w_b)\neq(-1,-1)$, one obtains
\[
    \Psi_g(G,\mathbf w)
    =
    \Psi_g(G-e,\mathbf w)
    +
    \Psi_g(G/e,\mathbf w/e).
\]
In the exceptional case $(w_a,w_b)=(-1,-1)$, the relation is replaced by a \textit{deletion--near--contraction} relation involving $G\odot e$.

Weighted graph invariants satisfying deletion--contraction relations have been studied in \cite{Noble,V_poly,Crew_del_con}; see 
\cite{UVW} for a survey of results. The deletion--near--contraction relation appears, for instance, in \cite{M_poly}. 

\begin{question}
How is the invariant $\Psi_g(G,\mathbf w)$ related to previously studied weighted graph invariants?
\end{question}
The case when $\mathbf w = \mathbf 0$ is answered in Corollary~\hyperref[cor:C]{C}, see also Section \ref{subsec_unweight_graph}.

\subsubsection{Comparison with Reinke--Silversmith type integrals}

Recall from Section \ref{section:intro_rob} that Reinke--Silversmith \cite{RobReinke} study graph-associated integrals over $\mgnbar$ built from the pullback classes $\pi_{G,i}^*(\psi_i)$. In Section~\ref{subsec_comparison_psi}, we conjecture a precise relation between their construction and the corresponding integrals on graphically stable spaces.

\begin{conjectureI}[Conjecture~\ref{conjecture_psi}]\label{conj:I}
Let $g \geq 1$, and let $k_1,\ldots,k_n$ be positive integers, with at most one $k_i$ allowed to be $0$, satisfying  $ \sum_{i=1}^n k_i = 3g-3+n$. Then we have
\[
    \int_{\mgraphbar} \psi_1^{k_1}\cdots \psi_n^{k_n}
    =
    \int_{\mgnbar}
    \pi^*_{G,1}(\psi_1^{k_1})\cdots
    \pi^*_{G,n}(\psi_n^{k_n}).
\]
\end{conjectureI}

This conjecture would identify two \textit{a priori} different graph-theoretic constructions of tautological integrals, revealing a close connection between the tautological ring of $\mgnbar$ and $\mgraphbar$. We have proved this conjecture for several infinite classes of examples (Lemma \ref{lem:complement_matching}), and verified it in several other cases using \textsf{admcycles} \cite{admcycles}. In fact, we show that Conjecture~{\hyperref[conj:I]{I}} is equivalent to the assertion that Reinke--Silversmith type integrals satisfy the same deletion--contraction relation as the corresponding integrals on graphically stable spaces; see Conjecture~\ref{conj:psi_2}.

\subsubsection{Generating series of $\psi$-class integrals}

Motivated by $2$-dimensional quantum gravity, Witten \cite{Witten1991} conjectured that the generating series of $\psi$-class intersection numbers on $\mgnbar$ satisfies the KdV hierarchy; this was later proved by Kontsevich \cite{Kontsevich1992} and others \cite{Witten_conj_rahul,Witten_conj_Mir,Witten_conj_lando}. A natural question, currently under investigation, is whether the analogous generating series for graphically stable spaces satisfies an integrable hierarchy.

\subsubsection{Topology of $\overline{M}_{0,G}$}

Let $p_G(\lef)$ denote the Grothendieck class of the compactification $\overline{M}_{0,G}$, viewed as a polynomial in $\lef$. When $G=K_n$, questions concerning the log-concavity of $p_{K_n}$ have recently been studied in \cite{Aluffi_Log_M0n,eur2026buildingsetschowrings,kiem_realrooted}. A natural question is whether analogous results hold for $p_G$ for more general graphs $G$.

\subsection{Conventions} Throughout this paper, all \textit{graphs} are assumed to be simple (containing no multiple edges or self-loops) and possibly disconnected, with vertex set $[n]:=\{1,...,n\}$.

\subsection{Acknowledgements}

The author is deeply grateful to their supervisor, Navid Nabijou, for countless inspiring discussions and constant encouragement. The author has also benefited from helpful correspondence with Vance Blankers, Maria Gillespie, David Klompenhouwer, Jake Levinson, Evgeny Shinder, and Rob Silversmith. We also thank Sebastian Bozlee for comments on an earlier draft of this work.

\section{Background}
\subsection{Graph-theoretic background} 
The \textit{complete graph} $K_n$ is the graph on $n$ vertices containing every possible edge, whereas its complement, the \textit{edgeless graph} $E_n$, is the graph on $n$ vertices with no edges. We say that a graph is empty if it has no vertices. A vertex is called \textit{dominant} if it is adjacent to every other vertex.

\subsubsection{Graph operations}
\begin{definition}[Graph join $G*H$]
    The join of two disjoint graphs $G$ and $H$, denoted $G * H$, is the graph obtained by taking their disjoint union and adding all possible edges between a vertex in $G$ and a vertex in $H$.
\end{definition}
\begin{definition}[Edge contraction $G/e$]
    The contraction of an edge $e=ab$ in a graph $G$, denoted $G/e$, is the operation that removes $e$ and merges its endpoints $a$ and $b$ into a single new vertex labelled as $ab$, which inherits all edges formerly incident to either $a$ or $b$. Moreover, any resulting duplicate edges or self-loops are automatically deleted so $G/e$ is also simple.
\end{definition}

\subsubsection{Chromatic polynomial} As the title suggests, we are interested in the deletion--contraction relations arising from graph theory. One of the most famous examples is the \textit{chromatic polynomial} $\chi_G(x)$. For a positive integer $x$, the evaluation $\chi_G(x)$ counts the number of proper $x$-colourings of $G$---that is, assignments of an integer from $\{1, \dots, x\}$ to each vertex such that adjacent vertices receive distinct colours. Moreover, $\chi_G(x)$ can be defined recursively by deletion--contraction.

\begin{definition}\label{def_chrom}
    The \textit{chromatic polynomial} $\chi_G(x) \in \mathbb{Z}[x]$ of a graph $G$ is the graph invariant characterized by the deletion--contraction relation, that is for every edge $e$ we have:
    \begin{align*}
        \chi_G(x) = \chi_{G-e}(x) - \chi_{G/e}(x);
    \end{align*}
    with initial condition in the case of edgeless graph $\chi_{E_n}(x)=x^n$.
\end{definition}

\subsection{Graphically stable spaces} \label{sec: background_mgraph} 
 In this section we recall some basics about graphically stable spaces. 
\begin{definition}\label{def:graph_stab}
    Let $G$ be a graph with vertex set $[n]$ and $g$ be a non-negative integer. We say $G$ is $g$-$stable$ if 
    \begin{itemize}
        \item $2g-2+n > 0$ and 
        \item if $g=0$, then $G$ is not the edgeless graph $E_n$ nor bipartite. 
    \end{itemize}
\end{definition}
\begin{remark}
    In genus zero, it is useful to recall that a graph is non-bipartite if and only if it contains an odd cycle.
\end{remark}
\begin{definition}\label{def_G}Fix a $(g,n)$-stable graph $G$ and let $(C,p_1,...,p_n)$ be an at-worst nodal $n$-marked curve of genus $g$. We say $(C,p_1,...,p_n)$ is $G$-$stable$ if and only if 
\begin{enumerate}
    \item If $Z \subseteq C$ is an irreducible rational component with two nodes, then there exists at least one marked point on $Z^{\mathrm{sm}}$.
    \item If $Z \subseteq C$ is an irreducible rational component with one node, then there exist at least two markings $p_i$ and $p_j$ on $Z^{\mathrm{sm}}$ such that $ij$ is an edge in $G$. 
    \item For each smooth point $x \in C^{\mathrm{sm}}$, the set $\{p_i \mid p_i = x\}$ forms an \textit{independent set} in $G$ i.e. no two vertices from this set are adjacent in $G$. This is equivalent to imposing that a set of markings are allowed to coincide if and only if they form an independent set. 
\end{enumerate}
\end{definition}
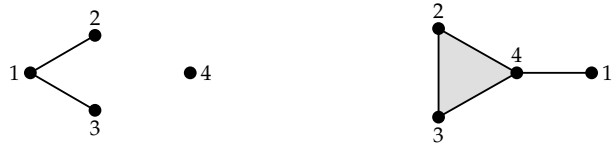
\begin{figure}[ht]
    \centering

\begin{tikzpicture}
[
    scale=0.9,
    every path/.style={line width=0.7pt, line cap=round},
    dot/.style={circle, fill=black, inner sep=1.7pt},
lab/.style={font=\normalfont\scriptsize}  
]

\begin{scope}[shift={(0,0)}]
    \coordinate (a1) at (0,0);
    \coordinate (a2) at (0.95,0.55);
    \coordinate (a3) at (0.95,-0.55);
    \coordinate (a4) at (2.35,0);

    \draw (a1) -- (a2);
    \draw (a1) -- (a3);

    \node[dot] at (a1) {};
    \node[dot] at (a2) {};
    \node[dot] at (a3) {};
    \node[dot] at (a4) {};

    \node[lab, left] at (a1) {1};
    \node[lab, above] at (a2) {2};
    \node[lab, below] at (a3) {3};
    \node[lab, right] at (a4) {4};
\end{scope}

\begin{scope}[shift={(6,0)}]
    \coordinate (b2) at (0,0.65);
    \coordinate (b3) at (0,-0.65);
    \coordinate (b4) at (1.15,0);
    \coordinate (b1) at (2.25,0);

    \filldraw[fill=gray!25, draw=black]
        (b2) -- (b3) -- (b4) -- cycle;

    \draw (b4) -- (b1);

    \node[dot] at (b2) {};
    \node[dot] at (b3) {};
    \node[dot] at (b4) {};
    \node[dot] at (b1) {};

    \node[lab, above] at (b2) {2};
    \node[lab, below] at (b3) {3};
    \node[lab, above] at (b4) {4};
    \node[lab, right] at (b1) {1};
\end{scope}
\end{tikzpicture}
\caption{The graph (left) and its corresponding independence complex (right).}
\end{figure}

We can associate an abstract simplicial complex called the \textit{independence complex} to $G$, whose $k$-faces are independent sets of size $k$. In \cite{BlankersBozlee}, Blankers and Bozlee introduced \textit{simplicial stable spaces}, which are compactifications of $\mgn$ associated to an abstract simplicial complex used to specify collision conditions. Roughly speaking, a set of markings are allowed to collide if they span a face of the simplicial complex. This further generalizes the work of Hassett \cite{Hassett} and Alexeev--Guy \cite{AG08}. By taking the independence complex of a graph, simplicial stable spaces specialise to the \textit{graphically stable spaces} that we are interested in.

\begin{thmdef}[{\!\cite[Theorem 4.18]{BlankersBozlee}}] Fix a non-negative integer $g$ and a $g$-stable graph $G$, there exists a smooth proper Deligne--Mumford stack $\mgraphbar$ over $\mathbb{Z}$ parametrizing flat families of $G$-stable curves of genus $g$. We will refer to $\mgraphbar$ as \textit{graphically stable spaces}. 
\end{thmdef}

\begin{remark}
    The locus $\mathcal{M}_{g,G} \subseteq \mgraphbar$ that parametrizes smooth genus $g$ $G$-stable curves is a partial compactification of $\mathcal{M}_{g,n}$. Their coarse space $M_{g,G}$ will be studied in Section \ref{Sec_Gro}.
\end{remark}
 
\begin{remark}
    We explain why we require a $0$-stable graph $G$ to be non-bipartite nor edgeless. Suppose $G$ is bipartite with parts $A$ and $B$, so $A$ and $B$ form two independent sets of $G$. Thus, markings belonging to $A$ are allowed to coincide with one another, as are markings belonging to $B$. We label the coincided points as $A$ and $B$, respectively. But we wish to exclude $(\mathbb{P}^1, A, B)$, a rational curve with two marked points, as it has infinite automorphisms. Similarly when $G$ is edgeless. These conditions correspond exactly to the \textit{triparted} condition on simplicial complexes in \cite{BlankersBozlee}.
\end{remark}

\begin{example}  \label{Ex_Gra} We present some examples of graphically stable spaces.
    \begin{enumerate}
        \item $G = K_n$, the complete graph on $n$ vertices. In this case no markings are allowed to collide with each other, and $\mgraphbar = \mgnbar$.
        \item  $G = K_m * E_n$, the join of an edgeless graph with a complete graph. The $n$ markings from $E_n$ can collide arbitrarily with each other but cannot collide with the other $m$ markings, whereas the other $m$ markings cannot collide with any markings. These spaces are also known as  \textit{heavy/light} spaces with Hassett weights \[\mathcal{A} = \left( \underbrace{1, \dots, 1}_{m}, \underbrace{\varepsilon, \dots, \varepsilon}_{n} \right), \quad \text{such that } n\varepsilon \leq 1.\] Such spaces have been studied in  \cite{CAVALIERI_HAMPE_MARKWIG_RANGANATHAN_2016,ChowHeavyLight,HodgeHeavyLight}. We will refer to the $m$ markings corresponding to the vertices of $K_m$ as \textit{heavy} markings, and the $n$ markings corresponds to $E_n$ as \textit{light} markings. When $g =0$ and $m = 2$, the corresponding graphically stable space coincides with Losev--Manin space \cite{LosevManin}.
        \item $G$ is complete multipartite. The heavy/light graphs $K_m * E_n$ are special cases of complete multipartite graphs: the heavy markings are in singleton parts and all the light markings are in one part. When $g=0$ the stability requires $G$ to be at least tripartite and has been studied in \cite{Fry,Khoroshkin_Lyskov,SRI}.

        \item $G = E_n \sqcup K_3 $, an edgeless graph disjoint union a complete graph on three vertices. When $g=0$, we can use $\mathrm{Aut}(\mathbb{P}^1)$ to fix the three markings of $K_3$, and the remaining $n$ markings can collide arbitrarily with each other and with these three fixed markings. Therefore, in this case the underlying $G$-stable curve must be smooth and we have $\overline{\mathcal{M}}_{0,G} \simeq \mathcal{M}_{0,G} \simeq (\mathbb{P}^1)^n$.
    \end{enumerate}
\end{example}

\subsection{Universal curve} \label{universal}
Let $G=H\sqcup\{*\}$ be a $g$-stable graph, where $*$ is an isolated vertex, and assume that $H$ is also $g$-stable.

\begin{proposition} \label{prop:forget}
    There is a natural forgetful morphism
    \begin{align} \label{eq:fgt}
        \mathrm{fgt}\colon \overline{\mathcal{M}}_{g,G = H \sqcup \{*\}} \longrightarrow \overline{\mathcal{M}}_{g,H}.
    \end{align}
\end{proposition}
\begin{proof}
We use the standard stabilization construction. Starting with a family of $G$-stable curves, we forget the marking $p_* $ and contract every irreducible component that becomes unstable, so we end up with a family of $H$-stable curves.

Since $*$ is isolated, the usual rational-tail instability cannot occur. Indeed, if a rational tail contains $p_*$, then $G$-stability requires the labels of the remaining markings on that tail to span at least one edge in $H$. Consequently, the only components that become unstable after forgetting $p_*$ are rational bridges whose only marking is $p_*$.

Label the vertices of $H$ by $1,\ldots,n-1$, and let $\pi\colon \mathcal{C}_{g,G}
    \rightarrow
    \overline{\mathcal{M}}_{g,G}$
be the universal curve. We define
\[
    \mathcal{L}
    :=
    \omega_\pi
    \left(
         \sum_{i=1}^{n-1}D_i
    \right),
\]
where $D_i$ is the image of the section corresponding to the marking $p_i$. 

Let $B$ be a rational bridge with two nodes whose only marking is $p_*$. Since none of the divisors $D_i$ meets $B$, we have
\[
    \deg\bigl(\mathcal{L}|_B\bigr)
    =
    \deg\bigl(\omega_{\mathbb{P}^1}(q_1+q_2)\bigr)
    =
    -2+2
    =
    0,
\]
where $q_1$ and $q_2$ are the two nodes of $B$. On can easily check on every other irreducible component, $\mathcal{L}$ has positive degree. In particular, a rational tail containing $p_*$ must also contain markings whose labels do not form an independent set in $H$, and hence contains at least two markings indexed by vertices of $H$. These markings contribute positively to the degree of $\mathcal{L}$. The same
degree statements remain true after replacing $\mathcal{L}$ by
$\mathcal{L}^{\otimes k}$ for any $k\geq 1$. 

For $k$ sufficiently large, one can deduce that $R^1 \pi_* \mathcal{L}^{\otimes ki}  = 0$ for $i \geq 0$. This can be seen fibrewise by analysing the long exact sequence in
cohomology induced by the normalization sequence $0 \rightarrow \mathcal{L}^{\otimes ki} \rightarrow \nu_*\nu^* \mathcal{L}^{\otimes ki}\rightarrow \bigoplus_{p \in C^{\mathrm{sing}}}\mathcal{L}^{\otimes ki}_p \rightarrow 0$, where $\nu$ is the normalization of a fibre $C$. Cohomology and base
change then implies that
$\pi_*\mathcal{L}^{\otimes ki}$ is locally free and commutes with
arbitrary base change. 

Therefore, we obtain the desired contraction and a flat family of $H$-stable curves over $\mgraphbar$ : 
\[ \begin{tikzcd} \mathcal{C}_{g,G} \arrow[rr] \arrow[dr,"\pi"] && \underline{\mathrm{Proj}} \bigoplus^{\infty}_{i=0} \pi_* \mathcal{L}^{\otimes ki} \arrow[dl, "\widetilde{\pi}"] \\ & \mgraphbar & \end{tikzcd} \]
In addition, the required morphism $\operatorname{fgt}$ is the morphism induced by this new family $\widetilde{\pi}$. 
\end{proof}

\vspace{-0.5em}
Denote $\inde(G)$ as the independence complex of $G$ which was discussed earlier. Notice that for all $S \in \inde(H)$, we have $S \sqcup \{*\} \in \inde(G)$. This implies that $\inde(G) = \mathrm{Cone}(\inde(H))$ with $*$ being the cone point. Geometrically this means the marking $p_*$ is allowed to collide with any other markings and can be placed anywhere on the smooth locus of a $H$-stable curve. If $p_*$ collides with a node, then we replace the node with a rational bridge with just $p_*$ on it. In terms of Hassett weights, the weight of $p_*$ should be treated as infinitesimal. The following statement is essentially a restatement of \cite[Proposition 5.4]{Hassett} in the case of graphically stable spaces: 
\begin{proposition}\label{prop:universal_curve}
     The universal curve $\mathcal{C} \rightarrow \overline{\mathcal{M}}_{g,H}$ is isomorphic to the forgetful morphism \eqref{eq:fgt}.  
\end{proposition} 
\begin{remark}
    This is also the same statement as \cite[Corollary 2.10]{AG08}. However, in \cite{AG08} the authors treated $p_*$ as a marking with weight $0$, hence it can even collide with nodes. Let us denote their space as $\overline{\mathscr{M}}_{g, H\sqcup \{*\}}$. If we assume $H$ has no isolated vertices, then both spaces are isomorphic to the universal curve of $\overline{\mathcal{M}}_{g,H} \simeq \overline{\mathscr{M}}_{g,H}$. The difference between the two conventions appears when there are two or more isolated vertices. When in the case of weight $0$, their space $\overline{\mathscr{M}}_{g, H \sqcup \{*,\star\}}$ is just the fibre product $\overline{\mathscr{M}}_{g, H \sqcup \{*\}} \times_{\overline{\mathscr{M}}_{g, H}} \overline{\mathscr{M}}_{g, H \sqcup \{\star\}}$. This space is singular when $p_*$ and $p_{\star}$ collide at a node \cite[§2.1.1]{Hassett}.
    Whereas our space $\overline{\mathcal{M}}_{g, H \sqcup \{*,\star\}}$ is smooth, and as we will see in Proposition \ref{small_res}, it is in fact a small resolution of the former space $\overline{\mathscr{M}}_{g, H \sqcup \{*,\star\}}$, i.e. the codimension of the exceptional loci is strictly greater than one.
\end{remark}

\subsection{$\psi$-classes} For each marking of the graphically stable space, there is a corresponding \textit{cotangent class}
\[\psi_i \in A^1(\mgraphbar; \mathbb{Q}).\] We are interested in the intersection numbers of such classes. When $G$ is complete multipartite, genus zero $\psi$-class integrals have been studied recently by Blankers, Gillespie and Levinson in \cite{SRI}, where they gave a combinatorial interpretation of such intersection numbers. \vspace{-1em}
\subsubsection{Isolated on and string equations} Suppose $n$ is an isolated vertex of $G$ i.e. $G = H \sqcup \{n\}$ and $H$ is $g$-stable. As discussed earlier in Section \ref{universal}, the geometric marking $p_n$ can collide with any other markings. The morphism forgetting $p_n$:
\begin{align*}
            \pi_{n} \colon \mgraphbar \longrightarrow \overline{\mathcal{M}}_{g,H}
        \end{align*}
coincides with the universal curve $\mathcal{C} \rightarrow \overline{\mathcal{M}}_{g,H}$ (Proposition \ref{prop:universal_curve}). We study how $\psi$-classes behave under this morphism. 

\begin{lemma}\label{cone_psi}
    Suppose $n$ is an isolated vertex of $G$, then the following holds in $A^1(\mgraphbar)$:
    \begin{enumerate}
        \item $\psi_n = c_1(\omega_{\pi_{n}})$, 
        \item $\psi_i = \pi_{n}^*\psi_i$, if $i=1,...,n-1$.
    \end{enumerate}  
\end{lemma}
These two assertions follow heuristically from the observation below, and we will also give a more formal justification. 
\begin{observation}
    There is no $G$-stable curve that supports a rational tail with $p_n$ and $p_i$ on it, such that this tail gets contracted after forgetting $p_n$. 
\end{observation}

Consider the following commutative diagram where we omit the genus $g$ for simplicity: 
\begin{equation} \label{diag:fiber_square_blowup}
\begin{tikzcd}[row sep=huge, column sep=large]
|[xshift=1.5em, yshift=-1.5em]| 
\overline{\mathcal{M}}_{H \sqcup \{n,m\}}
\arrow[dr, "\rho"]
\arrow[drr, "\tilde{\pi}_n", bend left=25]
\arrow[ddr, <-, "\sigma_{1,\dots,n}"', bend right=25] 
& & \\
& 
\overline{\mathcal{M}}_{H \sqcup \{m\}} 
\times_{\overline{\mathcal{M}}_{H}} 
\overline{\mathcal{M}}_{H \sqcup \{n\}}
\arrow[r, "h"]
\arrow[d, "f"']
\arrow[dr, phantom, "\scriptstyle\Box" description]
& 
\overline{\mathcal{M}}_{H \sqcup \{m\}} \simeq \mathcal{C}
\arrow[d, "\pi_m"'] \\
& 
\overline{\mathcal{M}}_{H \sqcup \{n\} = G}
\arrow[uul, <-, "\tilde{\pi}_m"]
\arrow[r, "\pi_n"'] 
& 
\overline{\mathcal{M}}_{H}
\arrow[u, "s_{i=1,\dots,n-1}"', bend right=35]
\end{tikzcd}
\end{equation}

\begin{proposition}\label{small_res}
    The morphism $\rho$ is a small resolution. In particular, this implies $\rho^* \omega_f = \omega_{\tilde{\pi}_m}$. 
\end{proposition}
\begin{proof}
    A point of the fibre product $\overline{\mathcal{M}}_{H \sqcup \{m\}} \times_{\overline{\mathcal{M}}_{H}} \overline{\mathcal{M}}_{H \sqcup \{n\}}$ can be thought as a curve with two free markings $p_m$ and $p_n$. We should allow $p_n$ and $p_m$ to collide at smooth points but they cannot collide at a node. If they collide at a node, we replace this node with a rational bridge with $p_n$ and $p_m$ on it, and this is exactly the exceptional loci of $\rho$. Since the loci of rational bridges has codimension $2$ inside $\overline{\mathcal{M}}_{H \sqcup \{n,m\}}$, this implies  $\rho$ is small.
\end{proof}

\begin{proof}[Proof of Lemma \ref{cone_psi}]
We first prove $(ii)$. By the definition of $\psi$-classes, we need to show $\pi_n^* s_i^* \omega_{\pi_m} = \sigma_i^* \omega_{\tilde{\pi}_m}$, for $i = 1,...,n-1$. Indeed, by the commutativity of the diagram, the fact that relative dualizing sheaf of a flat family of nodal curves is compatible with base change and Proposition \ref{small_res}, we have 
\begin{align*}
    \pi_n^* s_i^* \omega_{\pi_m} = \sigma_i^* \rho^* h^*  \omega_{\pi_m} = \sigma_i^* \rho^* \omega_f = \sigma_i^* \omega_{\tilde{\pi}_m}.
\end{align*}
To prove $(i)$, we need to show that $\sigma_n^*\omega_{\tilde{\pi}_m} = \omega_{\pi_n} $. First of all, the composition
\begin{align*}
    \tilde{\pi}_n \circ \sigma_n \colon \overline{\mathcal{M}}_{H \sqcup \{n\}} \longrightarrow \overline{\mathcal{M}}_{H \sqcup \{m\}}
\end{align*}
is an isomorphism --- simply relabels $p_n$ as $p_m$, the rest of the markings $p_{i \neq n}$ are unchanged. Therefore, we get 
\begin{align}\label{eq_lem211_a}
    \sigma_n^*\tilde{\pi}_n^* \omega_{\pi_m} = \omega_{\pi_n}.
\end{align}
Now by a similar argument as before, we also have
\begin{align}\label{eq_lem211_b}
    \tilde{\pi}_n^* \omega_{\pi_m} = \rho^* h^* \omega_{\pi_m} = \omega_{\tilde{\pi}_m}.
\end{align}
Substitute (\ref{eq_lem211_b}) into (\ref{eq_lem211_a}), we obtained the desired identity.
 \end{proof}

\begin{theorem}[{\! \cite[Theorem 8.1 \& 8.4]{AG08}}]\label{Thm_cone_dilaton}
    Suppose $n$ is an isolated vertex in $G$ and $G-n$ is $g$-stable. Then we have:
    \begin{enumerate}
        \item (Isolated dilaton equation)
            \begin{align}\label{Isolated dilaton}
                \int_{\overline{\mathcal{M}}_{g,G}} \psi_1^{k_1} \cdots \psi_{n-1}^{k_{n-1}}\psi_n = (2g-2)\int_{\overline{\mathcal{M}}_{g,G - n}}\psi_1^{k_1} \cdots \psi_{n-1}^{k_{n-1}}.
            \end{align}
        \item (Isolated string equation)
            \begin{align}\label{Isolated string}
            \int_{\overline{\mathcal{M}}_{g,G}} \psi_1^{k_1} \cdots \psi_{n-1}^{k_{n-1}}= 0.
            \end{align}
    \end{enumerate}
\end{theorem}
\begin{proof}
    Recall $\pi_n$ is the forgetful morphism from $\overline{\mathcal{M}}_{g,G}$ to $\overline{\mathcal{M}}_{g,G - n}$. Then $(i)$ follows from integrating $\psi_n$ along the fibre of $\pi_n$, Lemma \ref{cone_psi} and the projection formula. On the other hand, $(ii)$ follows from Lemma \ref{cone_psi} $(ii)$ and the projection formula. 
\end{proof}

\subsubsection{Wall-crossing} Instead of forgetting markings, another way to study intersection theory over graphically stable spaces is to \enquote{wall-cross} back to $\mgnbar$, where the intersection theory is better understood. Here \enquote{walls} should mimic the fine chamber decomposition within the space of stability conditions of Hassett space \cite{Hassett}. Specifically, one considers the natural proper birational morphism (\!\cite[Theorem 6.3]{BlankersBozlee}):
\begin{align}\label{reduction_morphism}
    \rho_G: \mgnbar \longrightarrow \ \mgraphbar 
\end{align}
and study how classes behave under pull-back/push-forward along this morphism. For example, Newman \cite{Newman} recently showed that the Chow ring of $\overline{\mathcal{M}}_{0,G}$, and more generally genus zero simplicial stable spaces, are generated by boundary divisors with relations coming from push-forward of WDVV relations on $\overline{\mathcal{M}}_{0,n}$.

The wall-crossing formula for $\psi$-class integrals was proved for Hassett spaces in \cite[Theorem~7.9]{AG08} \cite[Corollary 3.8]{BlankersRenzo}. A straightforward adaption (e.g. \cite[Theorem 2.13]{SRI}) of their work yields an analogous wall-crossing formula for graphically stable spaces as follow.

\begin{definition}[$G$-stable partition]\label{def:g_stable_partition}
    We say a partition $\mathcal{P} = \{P_1,...,P_{l(\mathcal{P})}\} \vdash V(G) = [n]$ is $G$-$stable$ if each part $P_i$ forms an independent set in $G$, where $l(\mathcal{P})$ is the length of partition. We denote $\mathfrak{P}_G $ as the set of $G$-stable partitions. 
\end{definition}
\begin{theorem}[{\!\cite{AG08,BlankersRenzo}}]\label{Theorem:wallcross} For $i=1,...,n$, let $k_i$ be non-negative integers such that $\sum k_i = 3g-3+n$. Then
    \begin{align}\label{wallcross}
        \int_{\overline{\mathcal{M}}_{g,G}} \psi_1^{k_1} \cdots \psi_n^{k_n} = \sum_{\mathcal{P} \in \mathfrak{P}_G}(-1)^{l(\mathcal{P})+n}\int_{\overline{\mathcal{M}}_{g,l(\mathcal{P})}} \psi_1^{\alpha_1-|P_1|+1} \cdots \psi_{l(P)}^{\alpha_{l(\mathcal{P})}-|P_{l(\mathcal{P})}|+1};
    \end{align}
    where $\alpha_m = \sum_{i \in P_m}k_i$ and $|P_m|$ is the size of the $m$-th part.
\end{theorem}

In principle, wall-crossing techniques solve the problem of computing intersection numbers of $\psi$-classes over graphically stable spaces---for example we know the $\psi$-class integrals over Deligne--Mumford spaces are governed by the KdV hierarchy \cite{Witten1991,Kontsevich1992}. In the case of genus $0$, the integrals over $\overline{\mathcal{M}}_{0,n}$ simplify to the multinomial and the wall-crossing formula (\ref{wallcross}) becomes: 
\begin{align*}
    \int_{\overline{\mathcal{M}}_{0,G}} \psi_1^{k_1} \cdots \psi_n^{k_n} = \sum_{\mathcal{P} \in \mathfrak{P}_G}(-1)^{l(\mathcal{P})+n} \binom{l(\mathcal{P}) - 3}{\alpha_1-|P_1|+1 ,\dots, \alpha_{l(\mathcal{P})}-|P_{l(\mathcal{P})}|+1}.
\end{align*}
Nevertheless, explicitly computing these sums over $\mathfrak{P}_G$ is challenging. One may wonder whether a nicer formula exists, or if these integrals satisfy notable structural properties. 

\subsubsection{Heavy/light spaces}\label{sec: heavy_light} In this subsection we review some results about $\psi$-class integrals over $\overline{\mathcal{M}}_{g,K_m * E_n}$. Denote $\psi_j$ as the $\psi$-class associated to a heavy point from $K_m$ and $\wpsi_i$ as the $\psi$-class associated to a light point from $E_n$. We are interested in the integrals of the following form:
\begin{align}\label{integral_heavy_light}
    \int_{\overline{\mathcal{M}}_{g,K_m * E_n}} \psi_1^{k_1}\cdots \psi_m^{k_m} \wpsi_1^{l_1} \cdots \wpsi_n^{l_n}. 
\end{align}
We first focus on the case when $\psi$-classes with positive exponents are all concentrated on the heavy markings i.e. $l_i = 0$ for all $i =1,...,n$. 

\begin{lemma}\label{lem:heavy}
For $i=1,...,m$, let $k_i$ be non-negative integers such that $\sum k_i = 3g-3+n+m$. Then
    \begin{align}\label{heavy}
        \int_{\overline{\mathcal{M}}_{g,K_m * E_n}} \psi_1^{k_1} \cdots \psi_m^{k_m} = \int_{\overline{\mathcal{M}}_{g,m+n}} \psi_1^{k_1} \cdots \psi_m^{k_m}.
    \end{align}
\end{lemma}
\begin{proof}
    We evaluate the integral on the left hand side using the wall-crossing formula (\ref{wallcross}). Let $\mathcal{P}$ be a $(K_m * E_n)$-stable partition and we wish to compute the term corresponds to $\mathcal{P}$ of the wall-crossing formula (\ref{wallcross}): 
    \begin{align}\label{term_light}
        (-1)^{l(\mathcal{P})+n}\int_{\overline{\mathcal{M}}_{g,l(\mathcal{P})}} \psi_1^{\alpha_1-|P_1|+1} \cdots \psi_t^{\alpha_t-|P_t|+1} \cdots \psi_{l(P)}^{\alpha_{l(\mathcal{P})}-|P_{l(\mathcal{P})}|+1}.
    \end{align}
    Since the $m$ heavy markings are connected to all the other vertices, they cannot form an independent set with any other vertices so they must be singletons in $\mathcal{P}$. Any arbitrary non-empty subset of the $n$ light markings forms an independent set, and suppose they made up the $t$-th part $P_t$. So the exponent of $\psi_t$ in (\ref{term_light}) is 
    \begin{align*}
        \alpha_t - |P_t|+1 = \sum_{i \in P_t}k_i - |P_t| +1 = 0 - |P_t| +1.  
    \end{align*}
    This number is non-negative if and only if $|P_t|=1$. Thus in order for (\ref{term_light}) not to vanish, all the $n$ light markings must be singletons in $\mathcal{P}$ too. Hence, in this case the wall-crossing formula (\ref{wallcross}) simplifies to a sum over one partition where all parts are singletons: 
    \begin{align*}
        \int_{\overline{\mathcal{M}}_{g,K_m * E_n}} \psi_1^{k_1} \cdots \psi_m^{k_m} &= (-1)^{n+m+n+m}\int_{\overline{\mathcal{M}}_{g,m+n}} \psi_1^{k_1-1+1} \cdots \psi_{m}^{k_m-1+1}\psi_{m+1}^{0-1+1}\cdots \psi_{m+n}^{0-1+1} \\ 
        & = \int_{\overline{\mathcal{M}}_{g,m+n}} \psi_1^{k_1} \cdots \psi_m^{k_m}
    \end{align*}
    as desired. 
\end{proof}
In the case of genus $0$, Lemma \ref{lem:heavy} simplifies to a closed formula: 
\begin{align}\label{heavy_zero}
    \int_{\overline{\mathcal{M}}_{0,K_m * E_n}} \psi_1^{k_1} \cdots \psi_m^{k_m} = \binom{m+n - 3}{{k_1, \dots, k_m}}. 
\end{align} 
When $m=2$, this formula specializes to a known formula: 
\begin{proposition}[\!\cite{StableQuotient}]\label{LM_psi}
    Let $k_1 + k_2 = n-1$, then
    \begin{align*}
        \int_{\overline{\mathcal{M}}_{0,K_2 * E_{n}}} \psi_1^{k_1}\psi_2^{k_2} = \binom{n-1}{k_1}.
    \end{align*}
\end{proposition}
\begin{remark}
    Recall that $\overline{\mathcal{M}}_{0,K_2 * E_n}$ is the Losev--Manin space. This formula was proved in \cite[§4.5]{StableQuotient}. Instead of using wall-crossing techniques, the authors derived the formula by studying the geometry of the forgetful morphism 
    $\overline{\mathcal{M}}_{0,K_2 * E_{n+1}} \rightarrow \overline{\mathcal{M}}_{0,K_2 * E_{n}}$. Their argument closely parallels our later proof in Example \ref{example_LM}, where we obtain an identical recursion via a string equation (Theorem~\ref{string_complete}). 
\end{remark}

We now study $\wpsi$-classes which are associated to the light markings. The following statement was proved in \cite{AG08}, and here we present a different proof. 
\begin{lemma}[{\!\cite[Lemma 5.5]{AG08}}]\label{LM}
    When $m=2$, the class $\wpsi_i = 0$ in $A^1(\overline{\mathcal{M}}_{0,K_2 * E_{n}})$ for all $i = 1,...,n$.
\end{lemma}
\begin{proof}
    Without loss of generality, we prove this statement for $i=1$. There is a proper birational morphism: 
    \begin{align*}
        \rho\colon \overline{\mathcal{M}}_{0,K_2 * E_{n}} \longrightarrow (\mathbb{P}^1)^{n-1}.
    \end{align*}
    Which corresponds to the reduction of graphs from $K_2 * E_{n}$ to $K_3 \sqcup E_{n-1}$ as shown below:

    \[
    \scalebox{0.75}
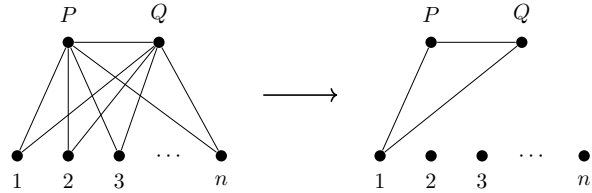
{%
\begin{tikzpicture}[
  every node/.style={circle, fill=black, inner sep=2pt},
  label distance=2pt
]

\begin{scope}[local bounding box=L]
  \node[label=above:$P$] (P) at (0,2) {};
  \node[label=above:$Q$] (Q) at (1.6,2) {};

  \node[label=below:$1$] (1) at (-0.9,0) {};
  \node[label=below:$2$] (2) at (0,0) {};
  \node[label=below:$3$] (3) at (0.9,0) {};
  \node[fill=none, inner sep=0pt] (dots) at (1.8,0) {$\cdots$};
  \node[label=below:$n$] (n) at (2.7,0) {};

  \draw (P)--(Q);

  \foreach \x in {1,2,3,n} {
    \draw (P)--(\x);
    \draw (Q)--(\x);
  }
\end{scope}

\begin{scope}[xshift=6.4cm, local bounding box=R]
  \node[label=above:$P$] (P2) at (0,2) {};
  \node[label=above:$Q$] (Q2) at (1.6,2) {};

  \node[label=below:$1$] (b1) at (-0.9,0) {};
  \node[label=below:$2$] (b2) at (0,0) {};
  \node[label=below:$3$] (b3) at (0.9,0) {};
  \node[fill=none, inner sep=0pt] (dots2) at (1.8,0) {$\cdots$};
  \node[label=below:$n$] (bn) at (2.7,0) {};

  \draw (P2)--(Q2);
  \draw (P2)--(b1);
  \draw (Q2)--(b1);
\end{scope}

\draw[->, thick] 
  ($(L.east)+(0.5,0)$) -- ($(R.west)+(-0.5,0)$);

\end{tikzpicture} 
}
\]

    We now claim that the universal curve of $\overline{\mathcal{M}}_{0, K_3 \sqcup E_{n-1}} \simeq (\mathbb{P}^1)^{n-1}$ is simply the trivial $\Pone$-bundle. Recall from Example \ref{Ex_Gra}, all $(K_3 \sqcup E_{n-1})$-stable rational curves are smooth, thus the universal curve is a $\Pone$-bundle. The three sections that correspond to the three vertices of $K_3$ are disjoint, thus trivialises this $\Pone$-bundle. Immediately we get $\psi_{1,P,Q} = 0$ as they correspond to constant sections $0,1,\infty$. Other $\psi$-classes on this space will be studied later in Proposition \ref{lem_two_colour}.

    The proposition now follows from observing that $\wpsi_1 = \rho^* \psi_1 = 0 $, which graphically corresponds to the fact that independent sets containing the vertex $1$ are unchanged after the graph reduction. 
\end{proof}
\begin{remark}
    In terms of Hassett spaces, the map $\rho$ corresponds to the reduction map (\!\cite[Theorem 4.1]{Hassett}) associated to the weight reduction: 
    \begin{align*}
        (1,1,\varepsilon,\varepsilon,...,\varepsilon) \geq (1-\frac{\varepsilon}{2}, 1-\frac{\varepsilon}{2}, \varepsilon, \frac{\varepsilon}{4(n-1)},...,\frac{\varepsilon}{4(n-1)}), \quad \text{such that } n\varepsilon \leq 1.
    \end{align*}
        One can check the Hassett space associated to the latter weight is exactly $(\mathbb{P}^1)^{n-1}$ (cf. \cite[p. 29]{Hassett}). 
\end{remark}
To recap, Lemma \ref{lem:heavy} resolves the case involving exclusively $\psi$-classes on heavy markings. Furthermore, Proposition~\ref{LM_psi} and Lemma~\ref{LM} effectively solve the problem of computing intersection numbers of $\psi$-classes over Losev--Manin spaces; indeed, the integrals studied in Proposition~\ref{LM_psi} are the only non-trivial ones. We now study the case when $m>2$ and $\psi$-classes with positive exponents are all concentrated on the light markings over arbitrary heavy/light spaces i.e. only involving $\wpsi$-classes. Such intersection numbers are closely related to the intersection numbers of Miller--Morita--Mumford $\kappa$-classes. 

\begin{lemma}[\!\cite{AG08,Pan12}]\label{psi_kappa} Suppose $\sum l_i=3g-3+n+m$ and $2g-2+m > 0$, then
\begin{align}
    \int_{\overline{\mathcal{M}}_{g,K_m * E_n}} \wpsi_1^{l_1} \cdots \wpsi_n^{l_n} = \int_{\overline{\mathcal{M}}_{g,m}} \kappa_{l_1-1} \cdots \kappa_{l_n-1}.
\end{align} 
\end{lemma}
\begin{proof}
    The case when $m=0$ was proved in \cite[Lemma 6.4]{AG08} and Pandharipande proved this statement for the loci containing curves of compact type \cite[Lemma
    3.1]{Pan12}. The ideas behind the two proofs are essentially the same, and can be easily adapted to prove our statement in full generality. For the sake of completeness, we lay out the proof here.

    We have a commutative diagram with a Cartesian square:
    \[
\begin{tikzcd}[row sep=huge, column sep=large]
\overline{\mathcal{M}}_{g,E_n * K_m}
\arrow[r, "f"]
\arrow[dr, "s\circ f"']
&
\underbrace{
\mathcal{C}_{g,m}
\times_{\overline{\mathcal{M}}_{g,m}}
\dots
\times_{\overline{\mathcal{M}}_{g,m}}
{\mathcal{C}}_{g,m}}
_{n\;\mathrm{times}}
\arrow[r, "{(\mathrm{pr}_1,\dots,\mathrm{pr}_n)}"]
\arrow[d, "s"']
\arrow[dr, phantom, "{\scriptstyle\Box}"{description, pos=0.3, xshift=0.52em}]
&
\mathcal{C}_{g,m} \times \cdots  \times \mathcal{C}_{g,m}
\arrow[d,"{(\pi,...,\pi)}"]
\\
&
\overline{\mathcal{M}}_{g,m}
\arrow[r, hook, "\Delta"']
&
\overline{\mathcal{M}}_{g,m} \times \cdots \times \overline{\mathcal{M}}_{g,m}
\end{tikzcd}
\]
        We now explain what the morphisms are:
        \begin{enumerate}
            \item For each light marking $p_{i=1,...,n}$, consider the forgetful morphism, $f_i$, which forgets all the other light markings except the $i$-th marking:
            \begin{align*}
                f_i \colon \overline{\mathcal{M}}_{g,E_n * K_m}  \longrightarrow \overline{\mathcal{M}}_{g,E_1 * K_m} = \overline{\mathcal{M}}_{g,K_{m+1}}.
            \end{align*}
            Recall $\overline{\mathcal{M}}_{g,K_{m+1}}$ is simply $\overline{\mathcal{M}}_{g,m+1}$, which in turn coincides with the universal curve $\pi: \mathcal{C}_{g,m} \rightarrow \overline{\mathcal{M}}_{g,m}$. So $f$ is simply the collection $(f_1,...,f_n)$. 
            \item $\mathrm{pr}_i$ is the projection to the $i$-th term of the fibre product and $s$ is the structure morphism. The composition $s \circ f $ is the forgetful morphism forgetting all the light markings and clearly we have $f_i = \mathrm{pr}_i \circ f$.
            \item $\Delta$ is the diagonal embedding. 
        \end{enumerate}
        The crucial observation is that $\wpsi_i = f_i^* \psi_i$, for all $i=1,..,n$. One way to see this is to consider the contracted loci of $f_i$. Alternatively, one can observe that $f$ is a proper birational morphism and, more specifically, a small resolution. Hence, by applying the projection formula to $f$ and chasing the Cartesian square, we get:
        \begin{align*}
            \int_{\overline{\mathcal{M}}_{g,E_n * K_m}} \wpsi_1^{l_1} \cdots \wpsi_n^{l_n} &= \int_{\mathcal{C}_{g,m} \times_{\overline{\mathcal{M}}_{g,m}} \dots \times_{\overline{\mathcal{M}}_{g,m}} \mathcal{C}_{g,m}}\mathrm{pr}_1^*\psi_1^{l_1} \cdots \mathrm{pr}_n^*\psi_n^{l_n} \\ 
            & = \int_{\overline{\mathcal{M}}_{g,m}} s_*\bigl( \mathrm{pr}_1^*\psi_1^{l_1} \cdots \mathrm{pr}_n^*\psi_n^{l_n} \bigl) \\ 
            & = \int_{\overline{\mathcal{M}}_{g,m}} \Delta^*(\pi_*\psi_1^{l_1}\boxtimes \cdots \boxtimes\pi_*\psi_n^{l_n}) \\
            & = \int_{\overline{\mathcal{M}}_{g,m}} \pi_*\psi_1^{l_1} \cdots \pi_*\psi_n^{l_n} \\
            & = \int_{\overline{\mathcal{M}}_{g,m}} \kappa_{l_1-1} \cdots \kappa_{l_n-1}.
        \end{align*}
        The last equality follows from the definition of $\kappa$-classes.  
\end{proof}
    \begin{remark}
            The morphism $f$ from the proof of Lemma \ref{psi_kappa} can also be viewed as a weight-reduction morphism from Hassett space with weights $(1^m, \varepsilon^n)$, to a 
            space with non-Hassett weights $(1^m, 0^n)$. See \cite[§5.2]{AG08} for a detailed discussion. So even though the wall-crossing formula wasn't used, Lemma \ref{psi_kappa} can still be seen as a wall-crossing phenomenon --- the \enquote{wall} here is the boundary of the space of Hassett stability conditions. 
    \end{remark}
Properties of intersection numbers of $\kappa$-classes were studied extensively in \cite{HighWP}, which in some sense concludes the study of the integrals only involving $\wpsi$-classes. 

In summary, we have now characterized the $\psi$-class integrals over heavy/light spaces for cases where the $\psi$-classes are supported exclusively on either the heavy markings or the light markings. Finally, we remark that arbitrary intersection numbers involving $\psi$-classes supported on both heavy and light markings still remain largely a mystery, except a combinatorial description is known in the case of genus zero \cite{SRI}.

\section{$\psi$-class integrals}

In this section, we study the graphical structure of $\psi$-class integrals on $\overline{\mathcal{M}}_{g,G}$. Our main result, Theorem~\ref{thm_del_con_psi}, establishes a deletion--contraction relation expressing integrals for $G$ in terms of integrals for $G-e$ and $G/e$. As applications, we derive $G$-dilaton and $G$-string equations; see Theorems~\ref{dilaton} and~\ref{string_complete}. We then discuss connections with $\kappa$-classes in Section \ref{sec:kappa}. Finally, in Section~\ref{subsec_unweight_graph}, we relate certain $\psi$-class integrals to classical graph invariants, namely evaluations of the chromatic polynomial.

\subsection{Main result: deletion--contraction for $\psi$-class integrals} Throughout this section, we fix a non-negative integer $g$ and a $g$-stable graph $G$.
\begin{theorem}[{Theorem \hyperref[thm:A]{A}}]\label{thm_del_con_psi}
    Let $e \in E(G)$ be an edge with endpoints $a$ and $b$. Suppose we are also given non-negative integers $k_1,...,k_{|V(G)|=n}$ such that $\sum k_i = 3g-3+n$, and $G-e$, $G/e$ are all $g$-stable graphs.
    \begin{enumerate}
        \item If $(k_a,k_b) \neq (0,0)$, we have
        \begin{align}\label{del_con_psi}
        \int_{\overline{\mathcal{M}}_{g,G}} \psi_1^{k_1} \cdots \psi_n^{k_n} = \int_{\overline{\mathcal{M}}_{g,G - e}} \psi_1^{k_1} \cdots \psi_n^{k_n} + \int_{\overline{\mathcal{M}}_{g,G/e}} \psi_1^{k_1} \cdots \psi_{ab}^{k_a+k_b-1} \cdots \psi_n^{k_n}. 
    \end{align}
    \item If $(k_a,k_b) = (0,0)$, we have
    \begin{align} \label{del_con_00psi}
        \int_{\overline{\mathcal{M}}_{g,G}} \psi_1^{k_1} \cdots \psi_n^{k_n} = \int_{\overline{\mathcal{M}}_{g,G - e}} \psi_1^{k_1} \cdots \psi_n^{k_n} + \int_{\overline{\mathcal{M}}_{g,G \odot e}} \psi_1^{k_1} \cdots \psi^0_{ab}\psi^0_{\widetilde{ab}}\cdots \psi_n^{k_n}. 
    \end{align}
    Where $G \odot e$ is the graph obtained from attaching a leaf to the vertex $ab$ of $G/e$. The vertex of this leaf is labelled as $\widetilde{ab}$.
    \end{enumerate}
\end{theorem}   
\begin{remark}
    We remark that \eqref{del_con_psi} can also be written in the form \eqref{del_con_00psi}, but that the formulation \eqref{del_con_psi} is easier to use and moreover the proof proceeds by first proving \eqref{del_con_psi} then deducing \eqref{del_con_00psi}. 
\end{remark}

\begin{example}\label{example:counter00} One might hope that \eqref{del_con_psi} continues to hold uniformly in both cases if one interprets $\psi^{-1}$ as zero. The following example shows that this naive convention fails. Consider the graphs below and let $e=13$. Then $(k_1,k_3)=(0,0)$, so we are in the exceptional case.

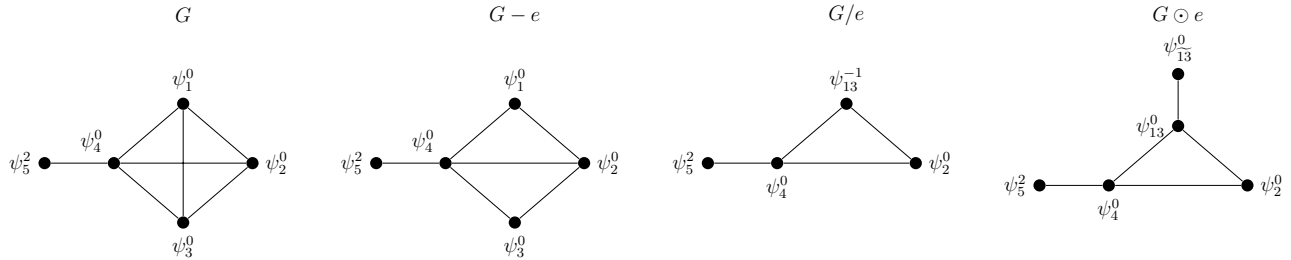
\begin{figure}[ht]
\centering
\resizebox{\textwidth}{!}{%
\begin{tikzpicture}[
  vertex/.style={circle,fill=black,inner sep=0pt,minimum size=7pt},
  edge/.style={line width=0.45pt},
  title/.style={font=\large}
]

\def\topY{1.2}
\def\midY{0}
\def\botY{-1.2}
\def\titleY{3.0}

\begin{scope}[shift={(0,0)}]
\node[title] at (0,\titleY) {$G$};

\node[vertex,label=above:{\large$\psi^0_1$}]      (g1) at (0,\topY) {};
\node[vertex,label=right:{\large$\psi^0_2$}]      (g2) at (1.4,\midY) {};
\node[vertex,label=below:{\large$\psi^0_3$}]      (g3) at (0,\botY) {};
\node[vertex,label=above left:{\large$\psi^0_4$}] (g4) at (-1.4,\midY) {};
\node[vertex,label=left:{\large$\psi^2_5$}]       (g5) at (-2.8,\midY) {};

\draw[edge]
(g5)--(g4)
(g4)--(g1)
(g1)--(g2)
(g2)--(g3)
(g3)--(g4)
(g4)--(g2)
(g1)--(g3);
\end{scope}

\begin{scope}[shift={(6.7,0)}]
\node[title] at (0,\titleY) {$G-e$};

\node[vertex,label=above:{\large$\psi^0_1$}]      (h1) at (0,\topY) {};
\node[vertex,label=right:{\large$\psi^0_2$}]      (h2) at (1.4,\midY) {};
\node[vertex,label=below:{\large$\psi^0_3$}]      (h3) at (0,\botY) {};
\node[vertex,label=above left:{\large$\psi^0_4$}] (h4) at (-1.4,\midY) {};
\node[vertex,label=left:{\large$\psi^2_5$}]       (h5) at (-2.8,\midY) {};

\draw[edge]
(h5)--(h4)
(h4)--(h1)
(h1)--(h2)
(h2)--(h3)
(h3)--(h4)
(h4)--(h2);
\end{scope}

\begin{scope}[shift={(13.4,0)}]
\node[title] at (0,\titleY) {$G/e$};

\node[vertex,label=above:{\large$\psi^{-1}_{13}$}] (k13) at (0,\topY) {};
\node[vertex,label=right:{\large$\psi^0_2$}]    (k2)  at (1.4,\midY) {};
\node[vertex,label=below:{\large$\psi^0_4$}]    (k4)  at (-1.4,\midY) {};
\node[vertex,label=left:{\large$\psi^2_5$}]     (k5)  at (-2.8,\midY) {};

\draw[edge]
(k5)--(k4)
(k4)--(k2)
(k4)--(k13)
(k13)--(k2);
\end{scope}

\begin{scope}[shift={(20.1,0)}]
\node[title] at (0,\titleY) {$G \odot e$};

\begin{scope}[shift={(0,-0.45)}]
\node[vertex,label=above:{\large$\psi^0_{\widetilde{13}}$}] (lleaf) at (0,2.25) {};
\node[vertex,label=left:{\large$\psi^0_{13}$}]              (l13)   at (0,\topY) {};
\node[vertex,label=right:{\large$\psi^0_2$}]                (l2)    at (1.4,\midY) {};
\node[vertex,label=below:{\large$\psi^0_4$}]                (l4)    at (-1.4,\midY) {};
\node[vertex,label=left:{\large$\psi^2_5$}]                 (l5)    at (-2.8,\midY) {};

\draw[edge]
(lleaf)--(l13)
(l5)--(l4)
(l4)--(l2)
(l4)--(l13)
(l13)--(l2);
\end{scope}

\end{scope}

\end{tikzpicture}%
}
\end{figure}

 We can compute the integrals associated to these graphs using the wall-crossing formula (\ref{wallcross}) as:
\begin{align*}
    \int_{\overline{\mathcal{M}}_{0,G}} \psi_5^2 = -2;\quad
     \int_{\overline{\mathcal{M}}_{0,G-e}} \psi_5^2 = -1;\quad
     \int_{\overline{\mathcal{M}}_{0,G/e}} \psi_5^2\psi^{-1}_{13} =0;\quad
     \int_{\overline{\mathcal{M}}_{0,G\odot e}} \psi_5^2 =-1 .
\end{align*}
Therefore, we see that (\ref{del_con_psi}) does not hold naively in this case. Instead, it is necessary to modify $G/e$ to $G \odot e$ and implement (\ref{del_con_00psi}). However, under certain circumstances, (\ref{del_con_psi}) does hold for $(0,0)$-edges as well, see Proposition \ref{00_edge_stronger}.

\end{example}   

\begin{remark}
        We discuss why it is necessary to assume $G$, $G-e$ and $G/e$ to all be $g$-stable, instead of just declaring the integral vanishes when the graph is unstable. Using the wall-crossing formula $(\ref{wallcross})$, one can show the following integral over the cyclic graph $C_5$ does not vanish:
        \begin{align*}
            \int_{\overline{\mathcal{M}}_{0,C_5}} \psi_1^2 = -1.
        \end{align*}
        However, if we delete an edge, say $e=12$, the graph becomes the path graph $P_5$, and if we contract an edge, it becomes $C_4$. In either case, the graph we obtain is bipartite, hence not $0$-stable. So if we set 
        \[\int_{\overline{\mathcal{M}}_{0,P_4}} \psi_{12} =\int_{\overline{\mathcal{M}}_{0,C_4}} \psi_{12} = 0,\]
        then \eqref{del_con_psi} fails as the original integral over $C_5$ is non-zero.
\end{remark}

\subsubsection{Proof of Theorem \ref{thm_del_con_psi}} Before proving the theorem, we introduce some notation and prove a technical lemma. Let $\inde(G)$ denote the independence complex of $G$, as discussed in Section \ref{sec: background_mgraph}. For a vertex $v$ of $\inde(G)$, we define the link of $v$ to be
\[
\lk_{\inde(G)}(v):=\{\sigma\in\inde(G) : v\notin \sigma \text{ and } \sigma\cup\{v\}\in \inde(G)\}.
\]
We also write $N_G(i)\subseteq [n]\setminus\{i\}$ for the set of neighbours of the vertex $i$ in $G$.
\begin{lemma}\label{tech}
    Let $e$ be an edge joining $a$ and $b$. As abstract simplicial complexes over ground set $[n] \setminus \{a,b\}$, there is a natural identification
    \[\lk_{\inde(G)}(a) \cap \lk_{\inde(G)}(b) = \lk_{\inde(G/e)}(ab).\] 
    In particular, there is a bijection between their faces. 
\end{lemma}

\begin{proof} We prove this lemma by showing           both sides have the same membership criteria for faces.
    Choose $L \subseteq [n] \setminus \{a,b\}$ such that $L \in \lk_{\inde(G)}(a)$ and $
    L\in \lk_{\inde(G)}(b)$. By definition this means $L$ is independent in $G - \{a,b\}$ with no vertices of $L$ adjacent to either $a$ or $b$. That is, 
    \begin{align}\label{eq:membership_L}
        L \in \lk_{\inde(G)}(a) \cap \lk_{\inde(G)}(b) &\iff L \text{ is independent in $G - \{a,b\}$ and } L \cap N_G(a) =  L \cap N_G(b) = \varnothing \\
        &\iff L \text{ is independent in $G - \{a,b\}$ and } L \cap (N_G(a) \cup  N_G(b)) = \varnothing. \nonumber
    \end{align}
    Now we take $K \subseteq [n] \setminus \{a,b\}$ such that $K \in \lk_{\inde(G/e)}(ab)$. Similarly, this means $K$ is independent in $G/e - \{ab\}$ and no vertices of $K$ are adjacent to $ab$. Notice that away from $ab$, the graph $G/e$ has exactly the same edges as $G$. So being independent in $G/e - \{ab\}$ is same as being independent in $G - \{a,b\}$:
    \begin{align*}
         K \in \lk_{\inde(G/e)}(ab) \iff K \text{ is independent in $G - \{a,b\}$ and } K \cap N_{G/e}(ab) = \varnothing.
    \end{align*}
    Now we claim there is a natural identification: 
    \begin{align*}
        N_{G/e}(ab) = (N_G(a) \cup N_G(b)) \setminus \{a,b\}.
    \end{align*}
    Indeed, $x$ is adjacent to $ab$ in $G/e$ if and only if $x$ is adjacent to $a$ or to $b$ in $G$. Since $K$ avoids $a,b$, we have 
    \begin{align*}
         K \in \lk_{\inde(G/e)}(ab) &\iff K \text{ is independent in $G - \{a,b\}$ and } K \cap N_{G/e}(ab) = \varnothing \\
         & \iff K \text{ is independent in $G - \{a,b\}$ and } K \cap (N_G(a) \cup N_G(b)) = \varnothing \\
         & \iff K \in \lk_{\inde(G)}(a) \cap \lk_{\inde(G)}(b);
    \end{align*}
    where last equivalence follows from the discussion about $L \in \lk_{\inde(G)}(a) \cap \lk_{\inde(G)}(b)$ \eqref{eq:membership_L}.
\end{proof}

We are now ready to prove Theorem \ref{thm_del_con_psi}. 
\begin{proof}[Proof of Theorem \ref{thm_del_con_psi}]
    We will first prove (\ref{del_con_psi}), which is when $(k_a,k_b) \neq (0,0)$. The idea is simply to compute the right hand side of (\ref{del_con_psi}) using the wall-crossing formula (\ref{wallcross}) and show it equals to the left hand side. We first proceed with the integral over $\overline{\mathcal{M}}_{g,G - e}$:
    \begin{align*}
        \int_{\overline{\mathcal{M}}_{g,G - e}} \psi_1^{k_1} \cdots \psi_n^{k_n} = \sum_{\mathcal{P} \in \mathfrak{P}_{G-e}}(-1)^{l(\mathcal{P})+n}\int_{\overline{\mathcal{M}}_{g,l(\mathcal{P})}} \psi_1^{\alpha_1-|P_1|+1} \cdots \psi_{l(P)}^{\alpha_{l(\mathcal{P})}-|P_{l(\mathcal{P})}|+1}.
        \end{align*}
         Let $\mathcal{P}$ be a partition from $\mathfrak{P}_{G-e}$, then $\mathcal{P}$ falls into two cases:
        \begin{enumerate}
    \item The vertices $a$ and $b$ lie in two different parts of $\mathcal P$.
    In this case, adding the edge $e$ does not violate the independence of any
    part. Hence $\mathcal P$ is a $G$-stable partition.

    \item The vertices $a$ and $b$ lie in the same part. After relabelling the
    parts, assume that $a,b\in P_1$. Then $P_i\in \mathrm{Ind}(G)$ for every
    $i\neq 1$, while $P_1\in \mathrm{Ind}(G-e)$ but $P_1\notin \mathrm{Ind}(G)$. The remaining vertices in $P_1$, together with $a$ and $b$, must form an independent set in $G-e$. Equivalently, both $P_1\setminus\{a\}$ and $P_1\setminus\{b\}$ must form independent sets in $G$. In other words, we can write $P_1$ as:
    \[P_1=\{a,b\}\sqcup L, \quad\text{where } L \in \lk_{\inde(G)}(a) \cap \lk_{\inde(G)}(b).\]
\end{enumerate}

        To summarize we now have,
        \begin{align*}
        \int_{\overline{\mathcal{M}}_{g,G - e}} \psi_1^{k_1} \cdots \psi_n^{k_n} &= \sum_{\mathcal{P} \in \mathfrak{P}_{G}}(-1)^{l(\mathcal{P})+n}\int_{\overline{\mathcal{M}}_{g,l(\mathcal{P})}} \psi_1^{\alpha_1-|P_1|+1} \cdots \psi_{l(\mathcal{P})}^{\alpha_{l(\mathcal{P})}-|P_{l(\mathcal{P})}|+1} \\
        &  + \sum_{\substack{\mathcal{P} \vdash [n]; P_1 = \{a,b\} \sqcup L ;\\ \text{where $L \in \lk_{\inde(G)}(a) \cap \lk_{\inde(G)}(b)$}; \\ P_{i \neq 1} \in \mathrm{Ind}(G)}}(-1)^{l(\mathcal{P})+n}\int_{\overline{\mathcal{M}}_{g,l(\mathcal{P})}} \psi_1^{\alpha_1-|P_1|+1} \cdots \psi_{l(\mathcal{P})}^{\alpha_{l(\mathcal{P})}-|P_{l(\mathcal{P})}|+1}.
        \end{align*}
         Notice that by the wall-crossing formula, the first term above is simply the integral over $\overline{\mathcal{M}}_{g,G}$. For the second term, we can rewrite the power of $\psi_1$ and get:
        \begin{align*}
        \int_{\overline{\mathcal{M}}_{g,G - e}} \psi_1^{k_1} \cdots \psi_n^{k_n} & = \int_{\overline{\mathcal{M}}_{g,G}} \psi_1^{k_1} \cdots \psi_n^{k_n} \\
        & \hspace{-1.8cm} + \sum_{\substack{\mathcal{P} \vdash [n]; P_1 = \{a,b\} \sqcup L ;\\ \text{where $L \in \lk_{\inde(G)}(a) \cap \lk_{\inde(G)}(b)$}; \\ P_{i \neq 1} \in \mathrm{Ind}(G)}}(-1)^{l(\mathcal{P})+n}\int_{\overline{\mathcal{M}}_{g,l(\mathcal{P})}} \psi_1^{k_a+k_b + \sum_{l\in L}k_l-|L|-2+1} \cdots \psi_{l(\mathcal{P})}^{\alpha_{l(\mathcal{P})}-|P_{l(\mathcal{P})}|+1}. 
        \end{align*}
        For simplicity, we denote the second term in the above expression as $\mathcal{S}$ i.e.
        \begin{align}\label{second_term}
            \int_{\overline{\mathcal{M}}_{g,G - e}} \psi_1^{k_1} \cdots \psi_n^{k_n} = \int_{\overline{\mathcal{M}}_{g,G}} \psi_1^{k_1} \cdots \psi_n^{k_n}  + \mathcal{S}.
        \end{align}
        We will show $\mathcal{S}$ cancels out with the remaining integral over $\overline{\mathcal{M}}_{g,G/e}$, hence proving our theorem. Again using the wall-crossing formula, we now compute the integral over $\overline{\mathcal{M}}_{g,G/e}$, which has vertex set $V(G/e)=([n]\setminus\{a,b\})\sqcup\{ab\}$:
        \begin{align}\label{eq_proof_G/e}
            \int_{\overline{\mathcal{M}}_{g,G/e}} \psi_1^{k_1} \cdots \psi_{ab}^{k_a+k_b-1} \cdots \psi_n^{k_n} = 
             \sum_{\mathcal{Q} \in \mathfrak{P}_{G/e}}(-1)^{l(\mathcal{Q})+n-1}\int_{\overline{\mathcal{M}}_{g,l(\mathcal{Q})}} \psi_1^{\alpha_1-|Q_1|+1} \cdots \psi_{l(Q)}^{\alpha_{l(\mathcal{Q})}-|Q_{l(\mathcal{Q})}|+1}.
        \end{align}
        Without loss of generality, we assume the vertex $ab$ falls into the part $Q_1$ for all $\mathcal{Q} \in \mathfrak{P}_{G/e}$. By a similar argument as before, we may rewrite $Q_1$ as $\{ab\} \sqcup K$, for some $K \in \lk_{\inde(G/e)}(ab)$: 
        \begin{align*}
           \int_{\overline{\mathcal{M}}_{g,G/e}} \psi_1^{k_1} \cdots \psi_{ab}^{k_a+k_b-1} \cdots \psi_n^{k_n} & = \\ & \hspace{-3cm} \sum_{\substack{\mathcal{Q} \vdash V(G/e); Q_1 = \{ab\} \sqcup K ;\\ \text{where $K \in \lk_{\inde(G/e)}(ab)$}; \\ Q_{i \neq 1} \in \mathrm{Ind}(G)}}(-1)^{l(\mathcal{Q})+n-1}\int_{\overline{\mathcal{M}}_{g,l(\mathcal{Q})}} \psi_1^{k_a+k_b-1 + \sum_{l\in K}k_l-|K|-1+1} \cdots \psi_{l(Q)}^{\alpha_{l(\mathcal{Q})}-|Q_{l(\mathcal{Q})}|+1}.
        \end{align*}
        We now compare this sum with the sum $\mathcal{S}$ from (\ref{second_term}). Recall from Lemma \ref{tech} that there is a bijection between faces of $\lk_{\inde(G)}(a) \cap \lk_{\inde(G)}(b)$ and $\lk_{\inde(G/e)}(ab)$, so in both sums we are summing over the same partitions. The integrands in both sums are identical but the coefficients are different by $-1$. In other words, we have 
        \begin{align}\label{S_expression}
            \int_{\overline{\mathcal{M}}_{g,G/e}} \psi_1^{k_1} \cdots \psi_{ab}^{k_a+k_b-1} \cdots \psi_n^{k_n}  = -\mathcal{S}.
        \end{align}
        Finally, substitute (\ref{S_expression}) into (\ref{second_term}) and re-arrange, we obtained the first desired relation (\ref{del_con_psi}) in Theorem~\ref{thm_del_con_psi}.

        If $(k_a,k_b) = (0,0)$, we no longer have (\ref{eq_proof_G/e}). This is because the wall-crossing formula (\ref{wallcross}) no longer applies as the exponent of $\psi_{ab}$ is now negative. We prove a modified deletion--contraction relation $(\ref{del_con_00psi})$ in this case by replacing the integral:
        \begin{align*}
            \int_{\overline{\mathcal{M}}_{g,G}} \psi_1^{k_1} \cdots \psi^0_a\psi^0_b \cdots \cdots \psi_n^{k_n}
        \end{align*}
        with an integral to which the deletion–contraction relation \eqref{del_con_psi} from the previous case applies along the edge $e = ab$.

        Let $G^+$ be the graph obtained by attaching a leaf labelled as $\widetilde{ab}$ to $a$, and we assign a $\psi$-class to $a$. 
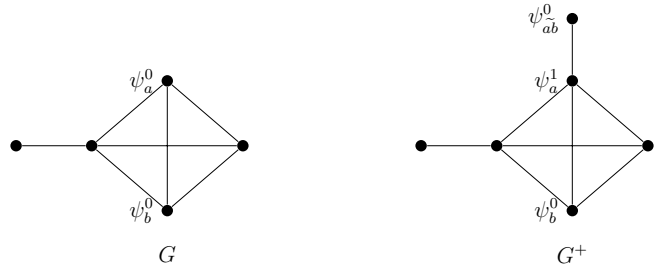
                \begin{figure}[ht]
\centering
\resizebox{0.5\textwidth}{!}{%
\begin{tikzpicture}[
  vertex/.style={circle,fill=black,inner sep=0pt,minimum size=6pt},
  edge/.style={line width=0.45pt},
  graphlabel/.style={font=\large}
]

\def\topY{1.2}
\def\midY{0}
\def\botY{-1.2}
\def\labelY{-2.0}

\begin{scope}[shift={(0,0)}]

\node[vertex,label=left:{\large$\psi_a^0$}] (g1) at (0,\topY) {};
\node[vertex] (g2) at (1.4,\midY) {};
\node[vertex,label=left:{\large$\psi_b^0$}] (g3) at (0,\botY) {};
\node[vertex] (g4) at (-1.4,\midY) {};
\node[vertex] (g5) at (-2.8,\midY) {};

\draw[edge]
(g5)--(g4)
(g4)--(g1)
(g1)--(g2)
(g2)--(g3)
(g3)--(g4)
(g4)--(g2)
(g1)--(g3);

\node[graphlabel] at (0,\labelY) {$G$};
\end{scope}

\begin{scope}[shift={(7.5,0)}]

\node[vertex,label=left:{\large$\psi^0_{\widetilde{ab}}$}] (h6) at (0,2.35) {};
\node[vertex,label=left:{\large$\psi_a^1$}] (h1) at (0,\topY) {};
\node[vertex] (h2) at (1.4,\midY) {};
\node[vertex,label=left:{\large$\psi_b^0$}] (h3) at (0,\botY) {};
\node[vertex] (h4) at (-1.4,\midY) {};
\node[vertex] (h5) at (-2.8,\midY) {};

\draw[edge]
(h6)--(h1)
(h5)--(h4)
(h4)--(h1)
(h1)--(h2)
(h2)--(h3)
(h3)--(h4)
(h4)--(h2)
(h1)--(h3);

\node[graphlabel] at (0,\labelY) {$G^+$};
\end{scope}

\end{tikzpicture}%
}

\label{fig:g-plus-example}
\end{figure}

        Now we can apply the previous deletion--contraction $(\ref{del_con_psi})$ to this leaf:
        \begin{align*}
            \int_{\overline{\mathcal{M}}_{g,G^+}} \psi_1^{k_1} \cdots \psi_{\widetilde{ab}}^{0}\psi^1_a \psi^0_b \cdots \psi_n^{k_n} &= \int_{\overline{\mathcal{M}}_{g,G \sqcup \{\widetilde{ab}\}}} \psi^{k_1}_1 \cdots \psi_{\widetilde{ab}}^{0}\psi^1_a \psi^0_b\cdots \psi^{k_n}_n + 
            \int_{\overline{\mathcal{M}}_{g,G / a  \widetilde{ab}}} \psi^{k_1}_1 \cdots \psi^0_{a\widetilde{ab}}\psi^0_b \cdots \psi^{k_n}_n
        \end{align*}
        The first integral vanishes by the Isolated string equation (\ref{Isolated string}) and we relabel the vertex $a\widetilde{ab}$ as $a$ in the second integral:
        \[          \int_{\overline{\mathcal{M}}_{g,G / a  \widetilde{ab}}} \psi^{k_1}_1 \cdots \psi^0_{a\widetilde{ab}}\psi^0_b \cdots \psi^{k_n}_n = \int_{\overline{\mathcal{M}}_{g,G}} \psi_1^{k_1} \cdots \psi^0_a \psi^0_b\cdots \cdots \psi_n^{k_n}. \]
         We see that this integral over $G^+$ is exactly the integral on the left-hand side of $(\ref{del_con_00psi})$. Also notice that in this integral over $G^+$, we have $(k_a,k_b) = (1,0)$ so we can apply the previous deletion--contraction relation (\ref{del_con_psi}) to the edge $e$ too, and get:
         \begin{align}\label{eq_proof00}
             \int_{\overline{\mathcal{M}}_{g,G^+}} \psi_1^{k_1} \cdots \psi_{\widetilde{ab}}^{0}\psi^1_a \psi^0_b \cdots \psi_n^{k_n} &= \int_{\overline{\mathcal{M}}_{g,G^+ - e}} \psi_1^{k_1} \cdots \psi_{\widetilde{ab}}^{0}\psi^1_a \psi^0_b \cdots \psi_n^{k_n} + \int_{\overline{\mathcal{M}}_{g,G^+ / e}} \psi_1^{k_1} \cdots \psi_{\widetilde{ab}}^{0} \psi^0_{ab} \cdots \psi_n^{k_n}.
         \end{align}
         Notice that $G^+ / e = G \odot e$, this second integral is exactly the integral over $G \odot e$ in $(\ref{del_con_00psi})$. By applying the deletion--contraction to the leaf and the Isolated string equation (\ref{Isolated string}) similarly as above, the first integral equals to the integral over $G-e$ in $(\ref{del_con_00psi})$:
         \begin{align*}
             \int_{\overline{\mathcal{M}}_{g,G^+ - e}} \psi_1^{k_1} \cdots \psi_{\widetilde{ab}}^{0}\psi^1_a \psi^0_b \cdots \psi_n^{k_n} = \int_{\overline{\mathcal{M}}_{g,G - e}} \psi_1^{k_1} \cdots \psi^0_a \psi^0_b \cdots \psi_n^{k_n}.
         \end{align*}
         Thus, (\ref{eq_proof00}) is exactly the desired statement $(\ref{del_con_00psi})$ and we are done. 
\end{proof}

\subsubsection{Deletion of $(0,0)$-edge} \label{subsec_del_00}
We show that the $\psi$-class integral sometimes does not change when we delete an edge $e = ab$ and $(k_a,k_b)= (0,0)$. 

\begin{definition}
    Let $NN_G(v)$ be the set of non-neighbours of the vertex $v$ in $G$. Alternatively, $NN_G(v)$ is also the $0$-skeleton of $\lk_{\inde(G)}(v)$.  
\end{definition}
\begin{proposition}\label{00_edge_stronger} Suppose $(k_a,k_b)=(0,0)$ and $k_{v} = 0$ for all $v \in NN_G(a) \cap NN_G(b)$, then we have
\begin{align*}
        \int_{\overline{\mathcal{M}}_{g,G}} \psi_1^{k_1} \cdots  \psi_n^{k_n} = \int_{\overline{\mathcal{M}}_{g,G - e}} \psi_1^{k_1} \cdots \psi_n^{k_n}.
    \end{align*}
     In particular, this holds if $NN_G(a) \cap NN_G(b) = \varnothing$.
\end{proposition}

\begin{proof} Recall from the proof of Theorem \ref{thm_del_con_psi},
    even when $(k_a,k_b)=(0,0)$, the relation (\ref{second_term}) still holds: 
        \begin{align}
            \int_{\overline{\mathcal{M}}_{g,G - e}} \psi_1^{k_1} \cdots \psi_n^{k_n} = \int_{\overline{\mathcal{M}}_{g,G}} \psi_1^{k_1} \cdots \psi_n^{k_n}  + \mathcal{S},
        \end{align}
    where 
    \begin{align*}
         \mathcal{S} = \sum_{\substack{\mathcal{P} \vdash [n]; P_1 = \{a,b\} \sqcup L ;\\ \text{where $L \in \lk_{\inde(G)}(a) \cap \lk_{\inde(G)}(b)$}; \\ P_{i \neq 1} \in \mathfrak{P}_{G}}}(-1)^{l(\mathcal{P})+n}\int_{\overline{\mathcal{M}}_{g,l(\mathcal{P})}} \psi_1^{0 + \sum_{l\in L}k_l-|L|-1} \cdots \psi_{l(\mathcal{P})}^{\alpha_{l(\mathcal{P})}-|P_{l(\mathcal{P})}|+1}. 
    \end{align*}

    Thus it is enough to show $\mathcal{S}$ vanishes under this assumption that $k_{v} = 0$ for all $v \in NN_G(a) \cap NN_G(b)$. Indeed, we must have $\sum_{l\in L} k_l = 0$ for all $L \in \lk_{\inde(G)}(a) \cap \lk_{\inde(G)}(b)$. This implies the exponent of $\psi_1$ in every integral in the sum is always negative, thus $\mathcal{S} = 0$.
\end{proof}

\subsection{Application: $G$-dilaton and $G$-string equations} Using the deletion--contraction property proved in the previous section, we derive a dilaton equation for graphically stable spaces in full generality, a string equation under some assumptions and a string equation in full generality in the case of complete multipartite graphs.

\begin{theorem}[{Theorem \hyperref[thm:B]{B}}]\label{dilaton} Label the neighbourhood of the vertex $n$ as 
\begin{align*}
    N_G(n) = \{u_1 , ..., u_d\} \subseteq [n-1],
\end{align*}
where $d = |N_G(n)|$. We also label the edge joining $n$ and $u_i$ as $e_i$. 
\begin{enumerate}
    \item ($G$-dilaton equation)
    \begin{align}\label{G_dilaton}
        \int_{\overline{\mathcal{M}}_{g,G}} \psi_1^{k_1} \cdots \psi_{n-1}^{k_{n-1}}\psi_n = (2g-2)\int_{\overline{\mathcal{M}}_{g,G - n}} \psi_1^{k_1} \cdots \psi_{n-1}^{k_{n-1}} + \sum_{i =1}^{d} \int_{\overline{\mathcal{M}}_{g,(G-e_1-\cdots - e_{i-1})/e_{i}}} \psi_1^{k_1} \cdots \psi_{nu_i}^{k_{u_i}} \cdots \psi_{n-1}^{k_{n-1}}.
    \end{align}
    Here $G-n$ denotes the graph obtained from $G$ after deleting the vertex $n$ and we set $e_{0} = \varnothing$.
    \item ($G$-string equation) Assume $k_{u_i} \neq 0$ for all $i=1,...,d$, we have 
            \begin{align}\label{G_string}
            \int_{\overline{\mathcal{M}}_{g,G}} \psi_1^{k_1} \cdots \psi_{n-1}^{k_{n-1}} = \sum_{i =1}^{d} \int_{\overline{\mathcal{M}}_{g,(G-e_1-\cdots - e_{i-1})/e_{i}}} \psi_{nu_i}^{k_{u_i}-1} \prod_{j \neq u_i} \psi_j^{k_j}.
        \end{align}
\end{enumerate}
These formulas hold provided that every graph appeared in the above formulas are $g$-stable.
\end{theorem}

\begin{remark}\label{classic_dilaton}
    When $G$ is the complete graph, the $G$-dilaton equation (\ref{G_dilaton}) specializes to the classical dilaton equation for $\mgnbar$ \cite{Witten1991}. Indeed, since every vertex of $K_n$ is dominant, we have $d = n-1$ and $(K_n-e_1-\cdots - e_{i-1})/e_{i} \simeq K_{n-1}$. Thus ($\ref{G_dilaton}$) becomes:
    \begin{align*}
        \int_{\overline{\mathcal{M}}_{g,K_n}} \psi_1^{k_1} \cdots \psi_{n-1}^{k_{n-1}}\psi_n &= (2g-2)\int_{\overline{\mathcal{M}}_{g,K_n - n \simeq K_{n-1}}} \psi_1^{k_1} \cdots \psi_{n-1}^{k_{n-1}} + (n-1) \int_{\overline{\mathcal{M}}_{g,K_{n-1}}} \psi_1^{k_1} \cdots \psi_{n-1}^{k_{n-1}} \\ 
        & = (2g-2+n-1) \int_{\overline{\mathcal{M}}_{g,K_{n-1}}} \psi_1^{k_1} \cdots \psi_{n-1}^{k_{n-1}}.
    \end{align*}
    For the $G$-string equation, the extra assumption on the exponents of $\psi$-classes prevents the specialization to the classical string equation. However, as we will see in the next section \ref{subsub_complete_string}, this assumption can be dropped in the case of complete multipartite graphs and we thus recover the classical string equation for $\mgnbar$.
\end{remark}

\begin{proof}[Proof of Theorem \ref{dilaton}]
    We first prove the $G$-dilaton equation (\ref{G_dilaton}). Notice that since $k_n =1$, the first deletion--contraction relation (Theorem \ref{thm_del_con_psi} (i)) applies to all the edges incident to $n$. We first apply it to $e_1$ and get: 
    \begin{align*}
         \int_{\overline{\mathcal{M}}_{g,G}} \psi_1^{k_1} \cdots \psi_{n-1}^{k_{n-1}}\psi_n = \int_{\overline{\mathcal{M}}_{g,G-e_1}}\psi_1^{k_1} \cdots \psi_{n-1}^{k_{n-1}}\psi_n + \int_{\overline{\mathcal{M}}_{g,G/e_1}} \psi_{nu_1}^{k_1} \cdots \psi_{n-1}^{k_{n-1}}. 
    \end{align*}
    Now we apply deletion--contraction again to $e_2$ for the integral over $\overline{\mathcal{M}}_{g,G-e_1}$ and get:
    \begin{align*}
        \int_{\overline{\mathcal{M}}_{g,G}} \psi_1^{k_1} \cdots \psi_{n-1}^{k_{n-1}}\psi_n = \int_{\overline{\mathcal{M}}_{g,G-e_1-e_2}}\psi_1^{k_1} \cdots \psi_{n-1}^{k_{n-1}}\psi_n + \int_{\overline{\mathcal{M}}_{g,(G-e_1)/e_2}}\psi_1^{k_1} \psi_{nu_2}^{k_2} \cdots \psi_{n-1}^{k_{n-1}} & \\
        & \hspace{-2.5cm} + \int_{\overline{\mathcal{M}}_{g,G/e_1}} \psi_{nu_1}^{k_1} \cdots \psi_{n-1}^{k_{n-1}}. 
    \end{align*}
    Then we apply deletion--contraction inductively to $e_3,...,e_d$ in a similar way. Eventually we get: 
    \begin{align*}
        \int_{\overline{\mathcal{M}}_{g,G}} \psi_1^{k_1} \cdots \psi_{n-1}^{k_{n-1}}\psi_n = \int_{\overline{\mathcal{M}}_{g,G-e_1-\cdots - e_d}}\psi_1^{k_1} \cdots \psi_{n-1}^{k_{n-1}}\psi_n + \sum_{i =1}^{d} \int_{\overline{\mathcal{M}}_{g,(G-e_1-\cdots - e_{i-1})/e_{i}}} \psi_1^{k_1} \cdots \psi_{nu_i}^{k_{u_i}} \cdots \psi_{n-1}^{k_{n-1}}. 
    \end{align*}
    Now notice that the vertex $n$ is an isolated vertex in the graph $G-e_1-\cdots - e_d$. So by the Isolated dilaton equation from Theorem \ref{Thm_cone_dilaton}, we have 
    \begin{align*}
        \int_{\overline{\mathcal{M}}_{g,G-e_1-\cdots - e_d}}\psi_1^{k_1} \cdots \psi_{n-1}^{k_{n-1}}\psi_n = (2g-2)\int_{\overline{\mathcal{M}}_{g,G-n}}\psi_1^{k_1} \cdots \psi_{n-1}^{k_{n-1}}.
    \end{align*}
    Substituting this back into the previous expression yields the $G$-dilaton equation (\ref{G_dilaton}). 

    The proof of $G$-string equation (\ref{G_string}) is similar. Under the assumption that $k_{u_i} \neq 0$ for all $i=1,...,d$, we apply deletion--contraction inductively to $e_1,...,e_d$ and get:
    \begin{align*}
        \int_{\overline{\mathcal{M}}_{g,G}} \psi_1^{k_1} \cdots \psi_{n-1}^{k_{n-1}} = \int_{\overline{\mathcal{M}}_{g,G-e_1-\cdots - e_d}}\psi_1^{k_1} \cdots \psi_{n-1}^{k_{n-1}} + 
        \sum_{i =1}^{d} \int_{\overline{\mathcal{M}}_{g,(G-e_1-\cdots - e_{i-1})/e_{i}}} \psi_{nu_i}^{k_{u_i}-1} \prod_{j \neq u_i} \psi_j^{k_j}.
    \end{align*}
    Similarly, by the Isolated string equation from Theorem \ref{Thm_cone_dilaton}, we have
    \begin{align*}
        \int_{\overline{\mathcal{M}}_{g,G-e_1-\cdots - e_d}}\psi_1^{k_1} \cdots \psi_{n-1}^{k_{n-1}} = 0
    \end{align*}
    and we are left with the desired $G$-string equation (\ref{G_string}). 
\end{proof}

\subsubsection{String equation for complete multipartite graphs}\label{subsub_complete_string}

One might hope to remove the hypothesis in the $G$-string equation \eqref{G_string} by setting $\psi^{-1}=0$. However, an adaptation of Example~\ref{example:counter00} shows that this naive convention fails in general. In the complete multipartite case, this convention does hold, allowing us to derive a fully general string equation for complete multipartite graphs.
\begin{theorem}[Complete multipartite string equation]\label{string_complete} Let $G$ be complete multipartite. Suppose we are also given non-negative integers $k_1,...,k_{n-1}$ such that $\sum k_i = 3g-3+n$. Order the neighbours of the vertex $n$ as 
\begin{align*}
    N_G(v) = \{u_1,...,u_k,u_{k+1},...,u_d\} \subseteq [n-1],
\end{align*}
such that $k_{u_i} = 0$ for $i=1,...,k$ and $k_{u_j} \neq 0 $ for $j = k+1,...,d$. We also label the edge joining $u_i$ and $n$ as $e_i$. Then: 
\begin{align}\label{formula_complete_string}
    \int_{\overline{\mathcal{M}}_{g,G}} \psi_1^{k_1} \cdots \psi_{n-1}^{k_{n-1}} = \sum_{i =1}^{d} \int_{\overline{\mathcal{M}}_{g,(G-e_1-\cdots - e_{i-1})/e_{i}}} \psi_{nu_i}^{k_{u_i}-1} \prod_{t \neq u_i} \psi_t^{k_t}.
\end{align}
These formulas hold provided that all graphs appearing above are $g$-stable.
\end{theorem}

\begin{remark}
    The complete graph is a complete multipartite graph. So in this case we recover the classic string equation for $\mgnbar$ \cite{Witten1991}, similarly to the discussion in Remark~\ref{classic_dilaton}. However, in general the graphs appear on the right hand side of (\ref{formula_complete_string}) need not all be complete multipartite. 
\end{remark}

The idea of the proof is similar to the proof of $G$-string equation. Except in this case, we may appeal to Proposition~ \ref{00_edge_stronger} to delete all the $(0,0)$-edges incident to $n$, and thus we no longer need the assumption that the exponents of $\psi$-classes corresponding to the neighbours of $n$ to be non-zero. 

\begin{proposition} Retaining the same notation and assumption as Theorem \ref{string_complete}, we have 
\begin{align*}
    \int_{\overline{\mathcal{M}}_{g,G}} \psi_1^{k_1} \cdots \psi_{n-1}^{k_{n-1}} = \int_{\overline{\mathcal{M}}_{g,G-e_1-\cdots - e_{k}}} \psi_1^{k_1} \cdots \psi_{n-1}^{k_{n-1}}.
\end{align*}
Moreover, since integrals with negative exponents of $\psi$-classes are zero, we can write:
\begin{align} \label{prop_complete_2}
    \int_{\overline{\mathcal{M}}_{g,G}} \psi_1^{k_1} \cdots \psi_{n-1}^{k_{n-1}} = \int_{\overline{\mathcal{M}}_{g,G-e_1-\cdots - e_{k}}} \psi_1^{k_1} \cdots \psi_{n-1}^{k_{n-1}} + \sum_{i =1}^{k} \int_{\overline{\mathcal{M}}_{g,(G-e_1-\cdots - e_{i-1})/e_{i}}} \psi_{nu_i}^{0-1} \prod_{t \neq u_i} \psi_t^{k_t}.
\end{align}
\begin{proof}\label{Prop_del_complete_multi} 
        Fix $i\in\{1,\ldots,k\}$, and set
        \[ H_i := G-e_1-\cdots-e_{i-1}.\]
        Define
        \[ N_i := NN_{H_i}(n)\cap NN_{H_i}(u_i),\]
        where $NN_H(v)$ denotes the set of non-neighbours of the vertex $v$ in a graph $H$. It suffices to show that
        \[N_i\subseteq \{u_1,\ldots,u_{i-1}\}.\]
        Indeed, this implies that $k_v=0$ for all $v\in N_i$, and the desired assertion then follows by applying Proposition~\ref{00_edge_stronger} inductively on $i$.
        
        We now prove the inclusion. Since $n$ and $u_i$ are joined by an edge in the complete multipartite graph $G$, they lie in different parts. Let $K_i$ be the part containing $u_i$, and let $v\in N_i$. Since $v$ is a non-neighbour of $u_i$ in $H_i$, and the only edges deleted in passing from $G$ to $H_i$ are $e_1,\ldots,e_{i-1}$, none of which is incident to $u_i$, so the neighbours of $u_i$ in $H_i$ are the same as its neighbours in $G$. Hence $v$ is also a non-neighbour of $u_i$ in $G$. Because $G$ is complete multipartite, this implies that $v$ lies in the same part as $u_i$, namely $v\in K_i$.
        
        On the other hand, $v\in NN_{H_i}(n)$. Since $n$ lies in a different part from $K_i$, it is adjacent in $G$ to every vertex of $K_i$. Therefore the only way for $v$ to be a non-neighbour of $n$ in $H_i$ is for the edge $nv$ to have been deleted among $e_1,\ldots,e_{i-1}$. Thus $v=u_j$ for some $j<i$, and hence $ v\in \{u_1,\ldots,u_{i-1} \}$.
        This proves the claimed inclusion $ N_i\subseteq \{u_1,\ldots,u_{i-1}\}$.
\end{proof}
\begin{remark}
In \cite[Corollary~2.14]{SRI}, the authors also consider the deletion of $(0,0)$-edges. As Example~\ref{example:counter00} shows, however, the statement does not hold in full generality. The proof uses the assertion that, when $i$ is not an endpoint of $e$, the independent sets of $G$ containing $i$ coincide with those of $G-e$ containing $i$. Although this assertion fails in general, it does hold when $G$ is complete multipartite, which is the only case needed in their application.
\end{remark}
    
\end{proposition}
\begin{proof}[Proof of Theorem \ref{string_complete}]
    First apply Proposition \ref{Prop_del_complete_multi} to delete $e_1,...,e_k$ and obtain (\ref{prop_complete_2}). Now we are left with edges $e_{j}$ for $j = k+1,...,d$, which we can apply deletion--contraction (Theorem~\ref{thm_del_con_psi}) to. We do so inductively to all $e_j$ and eventually $n$ becomes an isolated vertex, similar to the proof of $G$-string equation. Finally, we apply Isolated string equation (\ref{Isolated string}) to the isolated vertex and our theorem follows. 
\end{proof}

\begin{example}[Losev--Manin Space]\label{example_LM} We use the string equation for complete multipartite graphs (Theorem \ref{string_complete}) to prove Proposition \ref{LM_psi}. Suppose $n \geq 2$ and $k_1+k_2 = n-1 $ and recall we are interested in the integral:
\begin{align*}
        \int_{\overline{\mathcal{M}}_{0,K_2 * E_{n}}} \psi_1^{k_1}\psi_2^{k_2},
    \end{align*}
where the $\psi$-classes are concentrated on the two heavy points. Choose a light point, say vertex $l$ of the graph $E_n$. Notice that there are only two edges incident to $l$, namely the edges $e$ and $f$ that join the heavy points $1$ and $2$ respectively. We apply (\ref{formula_complete_string}) to this light point $l$ and obtain a recursion:
\begin{align*}
    \int_{\overline{\mathcal{M}}_{0,K_2 * E_{n}}} \psi_1^{k_1}\psi_2^{k_2} &=  \int_{\overline{\mathcal{M}}_{0,(K_2 * E_{n})/e}} \psi_1^{k_1-1}\psi_2^{k_2} + \int_{\overline{\mathcal{M}}_{0,(K_2 * E_{n}-e)/f}}\psi_1^{k_1}\psi_2^{k_2-1} \\
    & = \int_{\overline{\mathcal{M}}_{0,K_2 * E_{n-1}}} \psi_1^{k_1-1}\psi_2^{k_2} + \int_{\overline{\mathcal{M}}_{0,K_2 * E_{n-1}}}\psi_1^{k_1}\psi_2^{k_2-1}.
\end{align*}
We now compute the base-case which is when $n=2$. The Hassett weight of $\overline{M}_{0,K_2*E_2}$ is $(1,1,\varepsilon,\varepsilon)$, which belongs to the same coarse chamber (\!\cite{Hassett}) as the weight data $(1,1,1,1)$, thus we have $\overline{\mathcal{M}}_{0,K_2*E_2} \simeq \overline{\mathcal{M}}_{0,4} \simeq \Pone$. However, their universal curves are different. It is not so hard to see that the universal curve of our space is isomorphic to 
\begin{align*}
    \pi: \mathrm{Bl}_{(0,0),(\infty,\infty)}(\Pone \times \Pone) \longrightarrow \overline{\mathcal{M}}_{0,K_2*E_2} \simeq \Pone.
\end{align*}
The images of the sections $\sigma_1$, $\sigma_2$ are the strict transforms of constant sections $0$, $\infty$ respectively. Blowing-up a smooth point on a curve inside a smooth surface decreases the self-intersection number of that curve by one, hence $\sigma_1$, $\sigma_2$ both have self-intersection $-1$. Recall $\psi_i$ is also the first Chern class of the relative conormal bundle $N_{\sigma_i}^{\vee}$, so we deduce that $\psi_1 = \psi_{2} = [\mathrm{pt}] \in A^1(\overline{\mathcal{M}}_{0,K_2*E_2})$. Thus the base case is simply: 
\begin{align*}
    \int_{\overline{\mathcal{M}}_{0,K_2 * E_{2}}} \psi_1 = \int_{\overline{\mathcal{M}}_{0,K_2 * E_{2}}} \psi_2= 1.
\end{align*}
Solving the recursion with this base case yields the desired formula in Proposition \ref{LM_psi}: 
\begin{align*}
        \int_{\overline{\mathcal{M}}_{0,K_2 * E_{n}}} \psi_1^{k_1}\psi_2^{k_2} = \binom{n-1}{k_1}.
    \end{align*}

\end{example}

\subsection{Kappa classes}\label{sec:kappa}Given any $\psi$-class integral over $\mgraphbar$, we can repetitively apply the first deletion--contraction relation from Theorem \ref{thm_del_con_psi} to all the edges such that at least one of the endpoints supports a $\psi$-class. This reduces the original integral to a sum of integrals, such that each integral only has $\psi$-classes concentrated on isolated vertices. To be precise, we are left with integrals of the following form:
\begin{align}\label{eq:int_kappa}
    \int_{\overline{\mathcal{M}}_{g,H\sqcup \{u_1,...,u_t\}}} \psi_{u_1}^{k_{u_1}} \cdots  \psi_{u_t}^{k_{u_t}}.
\end{align}
Moreover, since all the edges in $H$ are $(0,0)$-edges, we can apply the second deletion-contraction relation from Theorem \ref{thm_del_con_psi} to all of them. If $g>0$, this will reduce the integral \eqref{eq:int_kappa} to a sum of integrals such that $H$ is a \textit{star} graph, see Figure \ref{fig:star_graph}. In genus $0$, the star graph is not $0$-stable, hence we cannot reduce to it.
\begin{figure}[htbp]
\centering
\begin{tikzpicture}[
    scale=1.0,
    every path/.style={line width=0.75pt, line cap=round, line join=round},
    vertex/.style={circle, fill=black, inner sep=1.8pt},
    lab/.style={font=\large}
]

\begin{scope}[shift={(0,0)}]
    \coordinate (a) at (1.2,0.75);
    \coordinate (b) at (0.25,-0.65);
    \coordinate (c) at (-0.85,0.25);
    \coordinate (d) at (2.15,0.15);

    \draw (a) -- (b);
    \draw (b) -- (c);
    \draw (c) -- (a);

    \draw (a) -- (d);

    \node[vertex] at (a) {};
    \node[vertex] at (b) {};
    \node[vertex] at (c) {};
    \node[vertex] at (d) {};

    \node[lab, above] at (a) {\small$\psi_a^0$};
    \node[lab, below] at (b) {\small$\psi_b^0$};
\end{scope}

\draw[->, line width=0.9pt] (2.9,0) -- (4.2,0);

\begin{scope}[shift={(5.0,0)}]
    \coordinate (ab) at (0,0);
    \coordinate (t)  at (0.55,1.05);
    \coordinate (l)  at (-0.90,-0.70);
    \coordinate (r)  at (1.05,-0.75);

    \draw (ab) -- (t);
    \draw (ab) -- (l);
    \draw (ab) -- (r);

    \node[vertex] at (ab) {};
    \node[vertex] at (t) {};
    \node[vertex] at (l) {};
    \node[vertex] at (r) {};

    \node[lab, right] at (ab) {\small$\psi_{ab}^0$};
    \node[lab, right] at (t) {\small$\psi_{\widetilde{ab}}^0$};
\end{scope}

\end{tikzpicture}
\caption{Illustration of turning $H$ into a star graph, by applying Theorem \ref{thm_del_con_psi} (ii) to the edge $ab$.}
\label{fig:star_graph}
\end{figure}
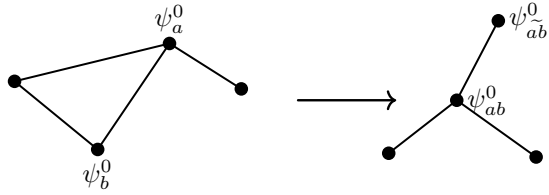

In Lemma \ref{lemma_psikappa}, we will show that an integral of the form \eqref{eq:int_kappa} equals to a integral of a monomial consisting of $\nkappa$-classes (Definition \ref{def:nkappa}). This leads to the following philosophical observation.
\begin{observation}\label{obs:kappa}
        The intersection number of $\psi$-classes over $\mgraphbar$ is determined by the deletion--contraction relations from Theorem \ref{thm_del_con_psi} and intersection numbers of $\nkappa$-classes over graphically stable spaces. Moreover, if $g>0$, we can further reduce the $\nkappa$-class integrals to graphically stable spaces defined by star graphs.
\end{observation}

\subsubsection{\texorpdfstring{$\tilde{\kappa}_i$-classes}{Tilde kappa\_i-classes}}
Let $\pi_{n+1} \colon \overline{\mathcal{M}}_{g,G \sqcup \{n+1\}} \rightarrow \mgraphbar$ be the morphism forgetting the marking $p_{n+1}$, which corresponds to an isolated vertex in $G$. Recall from Section \ref{universal} that this map is isomorphic to the universal curve of $\mgraphbar$.

\begin{definition}\label{def:nkappa}
    We define the $\nkappa_i$\textit{-class} to be:
    \begin{align*}
        \nkappa_i := (\pi_{n+1})_*(\psi_{n+1}^{i+1}) \in A^i(\mgraphbar; \mathbb{Q}).
    \end{align*}
\end{definition}\label{nkappa}
By Lemma \ref{cone_psi}, this is equivalent to
    \begin{align*}
        \nkappa_i = (\pi_{n+1})_* \bigl( c_1(\omega_{\pi_{n+1}})^{i+1} \bigl).
    \end{align*}
Immediately we have $\nkappa_0 = (2g-2)[\mgraphbar] \in A^0(\mgraphbar)$.
\begin{remark}
    When $G = K_n$, the $\nkappa$-classes are also known as the \textit{naive kappa classes} and they are related to the usual Miller--Morita--Mumford $\kappa$-classes (\!\cite[(1.5]{Arbarello_kappa}):
    \begin{align}\label{eq:kappa_wallcross}
        \kappa_i = \nkappa_i + \sum_{k=1}^{n} \psi_k^i \in A^i(\mgnbar).
    \end{align}
    In particular, when $n=0$ the two notions of kappa classes coincide.
\end{remark}
\begin{proposition}\label{lemma_kappa}Suppose $G$ also has an isolated vertex i.e. $G = H \sqcup \{n\}$, then in $A^i(\overline{\mathcal{M}}_{g,G})$ we have:
    \begin{align*}
        \nkappa_i = \pi_{n}^*\nkappa_i.
    \end{align*}   
\end{proposition}
\begin{proof}
    Recall the diagram (\ref{diag:fiber_square_blowup}) from earlier. By Lemma \ref{cone_psi} $(ii)$, we have $\tilde{\pi}_n^* \psi_m^{i+1} = \psi_m^{i+1}$. Pushing-forward to $\mgraphbar$:
    \begin{align*}
        (\tilde{\pi}_m)_* \tilde{\pi}_n^* \psi_m^{i+1} = (\tilde{\pi}_m)_* \psi_m^{i+1} = \nkappa_i.
    \end{align*}
    Now chasing the arrows in (\ref{diag:fiber_square_blowup}), we get
    \begin{align*}
        (\tilde{\pi}_m)_* \tilde{\pi}_n^* \psi_m^{i+1} = f_* \rho_* \rho^* h^* \psi_m^{i+1} = f_*h^* \psi_m^{i+1} = \pi_n^* (\pi_m)_*\psi_m^{i+1} =  \pi_n^* \nkappa_i.
    \end{align*}
    Here $\rho_* \rho^* = \mathrm{id}$ as $\rho$ is proper birational and we also used the fact that the square in (\ref{diag:fiber_square_blowup}) is Cartesian.
\end{proof}

\begin{lemma}\label{lemma_psikappa} Let $H$ be a $g$-stable graph, we have
    \begin{align*}
        \int_{\overline{\mathcal{M}}_{g,H\sqcup \{u_1,...,u_t\}}} \psi_{u_1}^{k_{u_1}} \cdots  \psi_{u_t}^{k_{u_t}} = \int_{\overline{\mathcal{M}}_{g,H}} \nkappa_{k_{u_1-1}} \cdots \nkappa_{k_{u_t-1}}.  
    \end{align*}
\end{lemma}
\begin{proof}
    This follows from inductively applying Proposition \ref{lemma_kappa}, Lemma \ref{cone_psi} $(ii)$ and the projection formula. 
\end{proof}
We invite the readers to compare Lemma \ref{lemma_psikappa} with Lemma \ref{psi_kappa} in the case when $H$ is a complete graph. The next proposition demonstrates that the deletion–contraction relation naturally clarifies the connection between $\psi$-class and $\kappa$-class integrals.

\begin{proposition}[\!\cite{Witten1991}]\label{prop:kappa_mg2} Suppose $g \geq 2$ and $k_1+k_2 = 3g-1$, then
    \begin{align*}
        \int_{\overline{\mathcal{M}}_{g,K_2}} \psi_1^{k_1}\psi_2^{k_2} = \int_{\mgbar} \kappa_{k_1-1}\kappa_{k_2-1} + \int_{\mgbar} \kappa_{k_1+k_2-2}.
    \end{align*}
\end{proposition}
\begin{proof}
    If $(k_1,k_2)=(0,0)$, then the statement follows easily as $\kappa_{-1} = 0$. So we assume $(k_1,k_2) \neq (0,0)$. We apply the first deletion--contraction relation from Theorem \ref{thm_del_con_psi} to the only edge of $K_2$ and get:
    \[\int_{\overline{\mathcal{M}}_{g,K_2}} \psi_1^{k_1}\psi_2^{k_2} = \int_{\overline{\mathcal{M}}_{g,E_2}} \psi_1^{k_1}\psi_2^{k_2} + \int_{\overline{\mathcal{M}}_{g,E_1}} \psi_{12}^{k_1+k_2-1}.\]
    Then by applying Lemma~\ref{lemma_psikappa} with $H = \varnothing$ and equation \eqref{eq:kappa_wallcross}, we have:
    \[\int_{\overline{\mathcal{M}}_{g,E_2}} \psi_1^{k_1}\psi_2^{k_2} = \int_{\mgbar} \nkappa_{k_1-1}\nkappa_{k_2-1} = \int_{\mgbar} \kappa_{k_1-1}\kappa_{k_2-1} .\]
    The same argument applies to the integral over $E_1$. Alternatively, the last steps also follow from Lemma~\ref{psi_kappa}.
\end{proof}

\subsection{Unweighted graph invariants}\label{subsec_unweight_graph} Recall from Section \ref{section:future_weighted_graph}, given a set of weights $\mathbf w \in\mathbb Z^n$ and  $g$, we define
\[
    \Psi_g(G,\mathbf w)
    :=
    \int_{\mgraphbar}
    \psi_1^{w_1+1}\cdots \psi_n^{w_n+1}.
\]
By Theorem \ref{thm_del_con_psi}, if we set all the weights on vertices to be $0$, then we obtain
\begin{align}\label{non_weight_del_con}
    \Psi_{g}(G, \mathbf{0}) = \Psi_{g}(G-e, \mathbf{0}) +\Psi_{g}(G/e, \mathbf{0}).
\end{align}
Here the weights are always the same, hence (\ref{non_weight_del_con}) is essentially a deletion--contraction relation for unweighted graphs (or, more simply, graphs). Also recall from Definition \ref{def_chrom} that the chromatic polynomial $\chi_G(x)$ satisfies 
\begin{align*}
    \chi_G(x) = \chi_{G-e}(x) - \chi_{G/e}(x);
\end{align*}
with base case $\chi_{E_n}(x)=x^n$. It turns out many $\psi$-class integrals over $\mgraphbar$ coincide with $\chi_G(x)$ evaluated at various numbers, and in fact these $\psi$-class integrals determines $\chi_G(x)$ completely. 
\begin{remark}\label{rmk_chromatic}
    We first make two observations regarding chromatic polynomials:
    \begin{itemize} 
        \item Notice that $(-1)^{|V(G)|}\chi_G(x)$ is the graph               invariant that satisfies
        \begin{align*}
            (-1)^{|V(G)|}\chi_G(x) = (-1)^{|V(G)|}\chi_{G-e}(x) + (-1)^{|V(G)|-1}\chi_{G/e}(x).
        \end{align*}
        \item The derivatives of $\chi_G(x)$ also satisfy the deletion--contraction relation. 
    \end{itemize}
\end{remark}
Thus computing the base case $\Psi_{g}(E_n, \mathbf{0})$ is the key. However, there are two small issues:
\begin{enumerate}
    \item In genus $0$, the graph $E_n$ is not $0$-stable.
    \item For dimension reasons $\Psi_{g}(G, \mathbf{0})= 0 $ if $g \neq 1$. 
\end{enumerate}
Thus in most cases we must modify our settings a little. 

\subsubsection{Example I: Balanced elliptic integral and acyclic orientations with unique sources}
Let us first consider the only meaningful integral in this unweighted setting. 
\begin{definition}[Balanced elliptic integral]For any non-empty graph $G$, we have
    \begin{align*}
    \Psi_{1}(G, \mathbf{0}) = \int_{\overline{\mathcal{M}}_{1,G}} \psi_1 \cdots \psi_n.
\end{align*}
\end{definition}

\begin{proposition}
\label{balanced_elliptic} For any non-empty graph $G$, we have
    \begin{align}\label{equation_balanced_elliptic}
        24 \cdot\Psi_{1}(G, \mathbf{0}) = (-1)^{|V(G)|-1}\left. \frac{d}{dx} \chi_G(x) \right|_{x = 0}.
    \end{align}
\end{proposition}
\begin{remark}\label{rmk_aco_source}
        The quantity $(-1)^{|V(G)|-1}\left. \frac{d}{dx} \chi_G(x) \right|_{x = 0}$ equals to the the number of acyclic orientations of $G$ with a unique, specified source (a vertex with all outgoing orientations)\cite[Theorem 7.3]{CurtisZaslavsky1983}.
    \end{remark}
\begin{proof}[Proof of Proposition \ref{balanced_elliptic}]
    Recall $\Psi_{1}(G, \mathbf{0})$ satisfies the deletion--contraction relation: 
    \begin{align*}
        \Psi_{1}(G, \mathbf{0}) = \Psi_{1}(G-e, \mathbf{0}) +\Psi_{1}(G/e, \mathbf{0}).
    \end{align*}
    This is the same deletion--contraction relation satisfied by the right hand side of (\ref{equation_balanced_elliptic}). Thus, in order to prove both sides are equal, it suffices to compare the base case which is when $G = E_n$. 

    By the Isolated dilaton equation (\ref{Isolated dilaton}), we see that
    \begin{align*}
        \Psi_{1}(E_n, \mathbf{0})  = \int_{\overline{\mathcal{M}}_{1,E_n}} \psi_1 \cdots \psi_n = 
    \begin{cases}
    \int_{\overline{\mathcal{M}}_{1,1}} \psi_1 = \frac{1}{24} & \text{if } n = 1, \\
    0 & \text{if } n > 1.
    \end{cases}
    \end{align*}
    On the other hand, we have
    \begin{align}\label{base_case_aco_s}
        (-1)^{|V(E_n)|-1}\left. \frac{d}{dx} \chi_{E_n}(x) \right|_{x = 0} = (-1)^{n-1}\left. \frac{d}{dx} x^n \right|_{x = 0} = \begin{cases}
    1 & \text{if } n = 1, \\
    0 & \text{if } n > 1.
    \end{cases}
    \end{align}
    Hence, both sides attain the same value in the base case and satisfy the same deletion--contraction relation, which implies the desired statement. 
\end{proof}

\subsubsection{Example II: Rational double-cone integrals and acyclic orientations with unique sources} To consider any other similar integrals, modifications must be made. For example in genus $0$, if a graph has more than three vertices and two of them are dominant, then this graph must contain a copy of $K_3$, hence it cannot be bipartite and it is $0$-stable. One way to make sure our graph always has two dominant vertices is to join a copy of $K_2$. 
\begin{definition}[Rational double-cone integrals] For any non-empty graph $G$, we define
    \begin{align*} 
        \delta_G : = \int_{\overline{\mathcal{M}}_{0,G *K_2 }} \psi_1 \cdots \psi_{n-1}.
    \end{align*}
\end{definition}
\begin{proposition}
\label{lem:double_cone}For any non-empty graph $G$, we have
    \begin{align}\label{eq_double_cone}
         \delta_G = (-1)^{|V(G)|-1}\left. \frac{d}{dx} \chi_G(x) \right|_{x = 0}.
    \end{align}
    In particular, recall the right hand side is the  the number of acyclic orientations of G with a unique, specified source (Remark \ref{rmk_aco_source}).
\end{proposition}
\begin{proof}
    It is clear that $\delta_G$ also satisfies the deletion--contraction relation: 
    \begin{align*}
        \delta_G = \delta_{G-e}+\delta_{G/e}.
    \end{align*}
    The right hand side of (\ref{eq_double_cone}) also satisfies this relation. So now we need to compare the base cases which is when $G = E_n$.

    On one hand, when $G =  E_n * K_2$, the moduli space $\overline{\mathcal{M}}_{0,G}$ is the Losev--Manin space. The $\psi$-classes in $\delta_G$ are all concentrated at the light points. Recall from Lemma \ref{LM}, we know $\psi_i = 0$ for all $i =1,...,n-1$. So we deduce that
    \begin{align*}
       \delta_{E_n}  = \int_{\overline{\mathcal{M}}_{0,E_n *K_2 }} \psi_1 \cdots \psi_{n-1} = \begin{cases}
    \int_{\overline{\mathcal{M}}_{0,K_3}} 1 = 1 & \text{if } n = 1, \\
    0 & \text{if } n > 1.
    \end{cases}
    \end{align*}
    This is exactly the base case of the right hand side of (\ref{eq_double_cone}), which was computed earlier in (\ref{base_case_aco_s}). Thus, both sides must be equal and we are done.
\end{proof}

\subsubsection{Example III: Rational $K_3$-augmented integrals and 2-colourings}
Similarly, we can consider $G \sqcup K_3$ which is $0$-stable for any $G$, including $E_n$. 
\begin{definition}[Rational $K_3$-augmented integrals] For any graph $G$, we define
    \begin{align*}
    \xi_G : = \int_{\overline{\mathcal{M}}_{0,G  \sqcup K_3 }} \psi_1 \cdots \psi_n.
\end{align*}
\end{definition}

\begin{proposition}
\label{lem_two_colour} For any graph $G$, we have 
    \begin{align} \label{equation_two_colour}
         \xi_G = (-1)^{|V(G)|} \chi_G(2).
    \end{align}
    In addition, $\chi_G(2) = 2^{k(G)}$ if $G$ is bipartite, $0$ otherwise. Here $k(G)$ denotes the number of connected components of $G$.    
\end{proposition}
\begin{proof}
    Again, notice that $\xi_G$ satisfies the deletion--contraction relation: 
    \begin{align*}
        \xi_G = \xi_{G-e}+\xi_{G/e}.
    \end{align*}
     This is the same deletion relation satisfied by the right hand side of (\ref{equation_two_colour}). So we need to compute the base cases of both sides and compare. On one side we have
    \begin{align*}
        \xi_{E_n} = \int_{\overline{\mathcal{M}}_{0,E_n  \sqcup K_3 }} \psi_1 \cdots \psi_n.
    \end{align*}
    Recall from Example \ref{Ex_Gra} that $\overline{\mathcal{M}}_{0,E_n  \sqcup K_3 }  \simeq \mathcal{M}_{0,E_n  \sqcup K_3 } \simeq (\Pone)^n$. Let $\pi \colon \mathcal{C} \rightarrow \overline{\mathcal{M}}_{0,E_n  \sqcup K_3 }$ be the universal curve. As explained in the proof of  Lemma $\ref{LM}$, $\pi$ is simply the trivial $\Pone$-bundle and we have a commutative diagram:
    \[
    \begin{tikzcd}[column sep=large, row sep=large]
    \mathcal{C} \simeq \Pone \times \overline{\mathcal{M}}_{0,E_n  \sqcup K_3 } \arrow[rr, "p"] \arrow[dr, "\pi"] &                                  & \Pone \\
                                  & \overline{\mathcal{M}}_{0,E_n  \sqcup K_3 } \simeq (\Pone)^n \arrow[ul, "\sigma_{1,...,n}", bend left=20] \arrow[ur, "\mathrm{pr}_{1,...,n}"'] & 
    \end{tikzcd}
    \]
    Here $p$ is the projection, $\sigma_{1,...,n}$ are the $n$ sections that correspond to the vertices of $E_n$ and $\mathrm{pr}_i$ is the projection to the $i$-th copy in the product. 
    
    In this case we have $\omega_{\pi} \simeq p^* \omega_{\Pone} \simeq p^*\mathcal{O}_{\Pone}(-2)$. Thus for $i=1,...,n$, we obtain
    \begin{align*}
        \psi_i := \sigma_i^* \omega_{\pi} \simeq \sigma_i^* p^*\mathcal{O}_{\Pone}(-2) \simeq \mathrm{pr}_i^* \mathcal{O}_{\Pone}(-2). 
    \end{align*}
    In other words, 
    \begin{align*}
        \psi_i = \mathcal{O}(0, \dots, \underset{i\text{-th entry}}{\underline{-2}}, \dots, 0) \in A^1((\Pone)^n).
    \end{align*}
    Hence, the base case is simply
    \begin{align*}
        \xi_{E_n} = \int_{\overline{\mathcal{M}}_{0,E_n  \sqcup K_3 }} \psi_1 \cdots \psi_n = (-2)^n.
    \end{align*}
    This is exactly the value of $(-1)^n\chi_{E_n}(2)$. Together with the deletion--contraction relation satisfied by $\xi_G$, we conclude $\xi_G = (-1)^{|V(G)|} \chi_G(2)$ as desired.
\end{proof}

\subsubsection{Example IV: Dummy vertex and acyclic orientations} We can also introduce a dummy vertex to host more $\psi$-classes so that our integrand is always a top form. 
\begin{definition}[Dummy vertex integrals] Let $g$, $m$ be non-negative integers such that $2g-2+m > 0$, we define
    \begin{align*}
    \lambda_{G,g,m} := \int_{\overline{\mathcal{M}}_{g,(G \sqcup \bullet) * K_m }}\psi_1 \cdots \psi_n \psi_{\bullet}^{3g-2+m}.
    \end{align*}
\end{definition}

\begin{figure}[htbp]
\centering

\begin{tikzpicture}[
    vertex/.style={circle,fill=black,inner sep=2.2pt},
    join edge/.style={gray!55,thin},
    box/.style={draw,dashed,rounded corners,inner sep=8pt}
]

\node[vertex] (g1) at (0,1.5) {};
\node[vertex] (g2) at (1,2.2) {};
\node[vertex] (g3) at (1,0.8) {};
\node[vertex] (g4) at (2,1.5) {};

\node[vertex] (b) at (1,-1.1) {};

\node[vertex] (k1) at (5,2) {};
\node[vertex] (k2) at (5,0) {};
\node[vertex] (k3) at (6.7,1) {};

\begin{scope}[on background layer]
    \foreach \u in {g1,g2,g3,g4,b}{
        \foreach \v in {k1,k2,k3}{
            \draw[join edge] (\u)--(\v);
        }
    }
\end{scope}

\draw[thick] (g1)--(g2);
\draw[thick] (g2)--(g4);
\draw[thick] (g4)--(g3);

\draw[thick] (k1)--(k2);
\draw[thick] (k2)--(k3);
\draw[thick] (k3)--(k1);

\node[
    box,
    fit=(g1)(g2)(g3)(g4),
    label=above:$G$
] {};

\node[
    box,
    fit=(k1)(k2)(k3),
    label=above:$K_3$
] {};

\end{tikzpicture}

\caption{Illustration of $\bigl(G\sqcup\{\bullet\}\bigr)*K_3$}
\label{fig:join-example}
\end{figure}
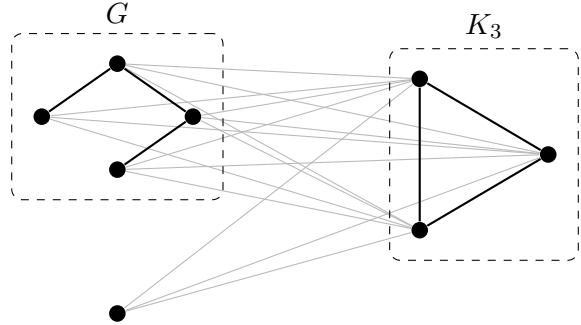

\begin{proposition}\label{lemma_acyc_psi} For any graph $G$, we have
\begin{align}
    24^g g!\cdot \lambda_{G,g,m} = (-1)^{|V(G)|} \chi_G(-(2g-2+m)).
\end{align}
Moreover, since $\chi_G$ is a polynomial of finite degree, it is completely determined by its values at finitely many points.
\end{proposition}
\begin{remark} \label{rmk_stanley_acyc}
    By a classic result of Stanley \cite{STANLEY_Acyc}, the value $(-1)^{|V(G)|}\chi_G(-1)$ equals to the number of acyclic orientations of $G$. More generally, $(-1)^{|V(G)|}\chi_G(-m)$ equals to the number of pairs $(\sigma, \mathcal{O})$, where $\sigma: V(G) \rightarrow \{1,...,m\}$ and $\mathcal{O}$ is an acyclic orientation such that if $u \rightarrow v$ in the orientation, then $\sigma(u) \geq \sigma(v)$. In particular, these values are all non-negative. 
\end{remark}
    
\begin{proof}[Proof of Proposition \ref{lemma_acyc_psi}]
    As before, first notice our integral satisfies the deletion--contraction relation: 
    \begin{align*}
        \lambda_{G,g,m} = \lambda_{G-e,g,m} + \lambda_{G/e,g,m}. 
    \end{align*}
    Next we compare the base cases of both sides. On the integral side we have
    \begin{align*}
        \lambda_{E_n,g,m} &= \int_{\overline{\mathcal{M}}_{g,(E_n \sqcup \bullet) * K_m }}\psi_1 \cdots \psi_n \psi_{\bullet}^{3g-2+m} \\
        & = \int_{\overline{\mathcal{M}}_{g,E_{n+1} * K_m }}\psi_1 \cdots \psi_n \psi_{n+1}^{3g-2+m}  \\
        & = \int_{\overline{\mathcal{M}}_{g,m}} \kappa^{n}_0 \kappa_{3g-3+m} \quad \text{by Lemma \ref{psi_kappa}} \\ 
        & = \frac{1}{24^g g!}(2g-2+m)^n.
    \end{align*}
    Here we also used the fact that $\kappa_0 = (2g-2+m) [\overline{\mathcal{M}}_{g,m}]$ and the degree of the top $\kappa$-class is $1/(24^g g!)$.

    On the other hand, we have 
    \begin{align*}
        (-1)^{|V(E_n)|} \chi_{E_n}(-(2g-2+m)) = (-1)^{n} (-(2g-2+m))^n = (2g-2+m)^n. 
    \end{align*}
    Therefore, the base cases are equal, so both sides of the statement are also equal to each other. 
\end{proof}
\begin{corollary} For any graph $G$, we have
    \begin{align*}
        \int_{\overline{\mathcal{M}}_{0,G * K_3 }}\psi_1 \cdots \psi_n = (-1)^{|V(G)|} \chi_G(-1).
    \end{align*}
\end{corollary}
\begin{proof}
    By Proposition \ref{lemma_acyc_psi}, we have 
    \begin{align*}
        (-1)^{|V(G)|} \chi_G(-1) = \lambda_{G,0,3} = \int_{\overline{\mathcal{M}}_{0,(G \sqcup \bullet) * K_3 }}\psi_1 \cdots \psi_n \psi_{\bullet}.
    \end{align*}
    As shown in Figure~\ref{fig:join-example}, the vertex $\bullet$ has three incident edges all connected to the $K_3$. Apply the $G$-dilaton equation (\ref{G_dilaton}) to $\psi_\bullet$ yields the desired formula
    \begin{align*}
        \int_{\overline{\mathcal{M}}_{0,(G \sqcup \bullet) * K_3 }}\psi_1 \cdots \psi_n \psi_{\bullet} &= -2 \cdot \int_{\overline{\mathcal{M}}_{0, G  * K_3 }}\psi_1 \cdots \psi_n + 3 \cdot \int_{\overline{\mathcal{M}}_{0, G  * K_3 }}\psi_1 \cdots \psi_n \\
        & = \int_{\overline{\mathcal{M}}_{0, G  * K_3 }}\psi_1 \cdots \psi_n \qedhere
    \end{align*}
\end{proof}

\subsubsection{Comparison with integrals over Deligne--Mumford spaces}\label{subsec_comparison_psi} Following \cite{RobReinke}, we now recall the construction of certain graph-associated integrals over $\mgnbar$. Let $G$ be a graph. For each $i =1,...,n$, let $\pi_{G,i}$ be the morphism that forgets all but the closed neighbourhood of $i$ in $G$: 
\begin{align*}
    \pi_{G,i}\colon \mgnbar \longrightarrow \overline{\mathcal{M}}_{g,|N_G[i]|}.
\end{align*}
Here $N_G[i] :=  N_G(i) \cup \{i\} \subseteq [n]$ denotes the closed neighbourhood of $i$. Importantly, the ordering of the markings of $\overline{\mathcal{M}}_{g,|N_G[i]|}$ are inherited from $\mgnbar$. 

\begin{definition}[{\!\cite{RobReinke}}]\label{def_rob} For non-negative integers $g,m$ such that $2g-2+m > 0$, we define
    \begin{align*}
        \omega_{G,g,m} := \int_{\overline{\mathcal{M}}_{g, n + m}} \pi^*_{G * K_m,1}(\psi_1) \cdots  \pi^*_{G * K_m,n}(\psi_n) \cdot \pi_M^*([\mathrm{pt}]).
    \end{align*}
    where $\pi_M$ is the map forgetting the first $n$-markings.
\end{definition}
The integrals $\omega_{G,g,m}$ are defined independently of $\mgraphbar$: the graph $G$ is used only to determine which markings are forgotten.
\begin{theorem}[\!\cite{RobReinke}, Theorem 1.4] We have the formula
\begin{align*}
    \omega_{G,g,m} = &  \;(-1)^{|V(G)|} \chi_G(-(2g-2+m))\\
     24 \cdot \int_{\overline{\mathcal{M}}_{1,n}} \pi^*_{G ,1}(\psi_1) \cdots  \pi^*_{G,n}(\psi_n) &=  (-1)^{|V(G)|-1}\left. \frac{d}{dx} \chi_G(x) \right|_{x = 0}.
\end{align*}
\end{theorem}
This theorem was proved with a similar strategy to our computations in this section---compute the base cases and prove a deletion--contraction relation \cite[(9)]{RobReinke}:
\begin{align}\label{eq:rob_del}
    \omega_{G,g,m} = \omega_{G-e,g,m} + \omega_{G/e,g,m}.
\end{align}
Comparing their results with Proposition \ref{balanced_elliptic} and Proposition \ref{lemma_acyc_psi}, we immediately see that
\[24^g g! \cdot \lambda_{G,g,m} = \omega_{G,g,m},\]
\begin{align*}
    \int_{\overline{\mathcal{M}}_{1,G}} \psi_1 \cdots \psi_n = \int_{\overline{\mathcal{M}}_{1,n}} \pi^*_{G ,1}(\psi_1) \cdots  \pi^*_{G,n}(\psi_n).
\end{align*}
It is clear that the definitions of Reinke--Silversmith integrals over $\overline{\mathcal{M}}_{g,n}$ are completely analogous to those over the graphically stable spaces $\overline{\mathcal{M}}_{g,G}$, and that their resulting values coincide. This leads us to conjecture the following.
\begin{conjecture}[Conjecture~{\hyperref[conj:I]{I}}]\label{conjecture_psi}
Let $g \geq 1$, and let $k_1, \dots, k_n$ be positive integers—with at most one $k_i$ allowed to be $0$—satisfying $\sum_{i=1}^n k_i = 3g - 3 + n$. Then for any $g$-stable graph $G$, we have
    \begin{align*}
        \int_{\mgraphbar} \psi_1^{k_1} \cdots \psi_n^{k_n} = \int_{\mgnbar} \pi^*_{G,1}(\psi_1^{k_1}) \cdots \pi^*_{G,n}(\psi_n^{k_n}).
    \end{align*}
\end{conjecture}

\begin{remark}\label{rmk:conj} 
    An immediate naive guess would be 
    \begin{align*}
        \rho_G^* (\psi_i) = \pi^*_{G,i}(\psi_i)
    \end{align*}
    for all $i$ where $\rho_G$ is the reduction morphism (\ref{reduction_morphism}), since then this conjecture follows after applying the projection formula to $\rho_G$. 
    However, this is not true in general as it fails due to boundary corrections, but this conjecture asserts that the boundary corrections always vanish after taking products. The claim does hold 
    in special cases such as when $G = E_n$ or when $G$ is the complement of a matching (i.e., the graph obtained by deleting a set of pairwise disjoint edges from a complete graph).
\end{remark}
\begin{lemma}\label{lem:complement_matching} Let $g\geq 0$, let $G$ be a $g$-stable graph, and let $k_1,\ldots,k_n$ be nonnegative integers. Suppose that either $G=E_n$ or $G$ is the complement of a matching. Then
    \begin{align*}
        \int_{\mgraphbar} \psi_1^{k_1} \cdots \psi_n^{k_n} = \int_{\mgnbar} \pi^*_{G,1}(\psi_1^{k_1}) \cdots \pi^*_{G,n}(\psi_n^{k_n}).
    \end{align*}
\end{lemma}
We have verified Conjecture \ref{conjecture_psi} in many other cases with the help of the SageMath package \textsf{admcycles}~\cite{admcycles}, and a genus-$0$ version of this conjecture can also be formulated with suitable modifications. Moreover,
this conjecture is equivalent to the following conjecture. 
\begin{conjecture}\label{conj:psi_2} Let $e$ be an edge joining $a$ and $b$. Suppose $(k_a,k_b) \neq (0,0) $, then
    \begin{align*}
        \int_{\mgnbar} \pi^*_{G,1}(\psi_1^{k_1}) \cdots \pi^*_{G,n}(\psi_n^{k_n}) &=  \int_{\mgnbar} \pi^*_{G-e,1}(\psi_1^{k_1}) \cdots \pi^*_{G-e,n}(\psi_n^{k_n}) \\  & \hspace{0.8cm}+ \int_{\overline{\mathcal{M}}_{g,n-1}} \pi^*_{G/e,1}(\psi_1^{k_1}) \cdots  \pi^*_{G/e,ab}(\psi_{ab}^{k_a+k_b-1})  \cdots \pi^*_{G/e,n}(\psi_n^{k_n}). 
    \end{align*}
\end{conjecture}
In other words, the Reinke--Silversmith type integrals satisfy the same deletion--contraction relation as the integrals over the graphically stable spaces (Theorem \ref{thm_del_con_psi}). A special case (\ref{eq:rob_del}) of Conjecture \ref{conj:psi_2} has been proved in \cite{RobReinke}.
\begin{proposition}
    Conjecture \ref{conjecture_psi}
     is equivalent to Conjecture \ref{conj:psi_2}.
\end{proposition}
\vspace{-2em}
\begin{proof}
    The forward implication is clear, since we have shown the deletion-contraction property for the $\psi$-class integrals over graphically stable spaces in Theorem \ref{thm_del_con_psi}. For the backward implication, we show both sides of Conjecture \ref{conjecture_psi} integrate to the same value when $G = E_n$. Then Conjecture \ref{conjecture_psi} follows from a straightforward induction on the number of edges, using Conjecture \ref{conj:psi_2}. 

    When $G = E_n$, we are considering forgetful morphisms:
    \begin{align*}
        \pi_{E_n,i}\colon \mgnbar \longrightarrow \overline{\mathcal{M}}_{g,\{i\}}.
    \end{align*}
    The classes $\pi^*_{E_n,i}(\psi_i)$ are also known as $\omega$-classes. Now the base case of Conjecture \ref{conjecture_psi} follows from either noticing that $\rho_G^* (\psi_i) = \pi^*_{E_n,i}(\psi_i)$ as in Lemma \ref{lem:complement_matching}, or recalling \cite[Theorem 3.1]{blankers_Omega} and the wall-crossing formula (\ref{wallcross}). 
\end{proof}

\section{Grothendieck classes}\label{Sec_Gro}
We now turn our attention to another deletion--contraction property. In this section, we study the topology of the coarse moduli space $M_{g,G}$. Our main result, Theorem~\ref{Gro_Class}, provides a deletion--contraction relation for their classes within the \textit{Grothendieck ring of varieties}. In genus zero, we explicitly determine the Grothendieck class of $M_{0,G}$ in Theorem~\ref{Gro_thm_zero}. As an application, we derive various topological invariants of $M_{0,G}$ (Corollary~\ref{lhop}) and investigate several links to combinatorics in the subsequent subsections. Finally, in the last section, we investigate the topological invariants of $M_{g,G}$ when $g > 0$.

Throughout this section, we denote by $k$ the base field, by $\mathbb{L}$ the Grothendieck class of the affine line $\mathbb{A}_{k}^1$, and by $1$ the class of a point.

\subsection{Main result: Deletion--contraction for Grothendieck classes}
\hypertarget{GroClassTarget}{}
\begin{theorem}[{Theorem \hyperref[thm:D]{D}}]\label{Gro_Class}

    The following relation holds within the Grothendieck ring of varieties $K_0(\mathrm{Var}_{k})$:
    \begin{align} \label{del_con_gro}
        [M_{g,G}] = [M_{g, G-e}] - [M_{g, G/e}],
    \end{align}
    providing all three spaces are $g$-stable.
\end{theorem}

\begin{proof}
Let $e$ be an edge that connects vertices $a$ and $b$. It is clear from the definition of $(G-e)$-stable curves that 
    \begin{align*}
        M_{g, G-e} \setminus M_{g, G} = 
        \{(C,p_1,...,p_n) \in M_{g, G-e} \;| \; p_a = p_b\}.
    \end{align*}
    By combining $p_a$ and $p_b$ into a new vertex $p_{ab}$, we obtain a smooth $(G/e)$-stable curve. This is because, by Lemma~\ref{tech}, the sets of markings allowed to collide with ${p_a,p_b}$ in $G-e$ is the same as the sets of markings allowed to collide with ${p_{ab}}$ in $G/e$. Conversely, given a $(G/e)$-stable curve $(C,p_1....p_{ab},...,p_n)$, we obtain a $(G-e)$-stable curve $(C,p_1,...,p_n)$ such that $p_a = p_{ab}, p_b = p_{ab}$.
    Thus we deduce
        \begin{align*}
        M_{g, G-e} \setminus M_{g, G} = 
        M_{g, G/e}.
    \end{align*}
     This becomes the desired additive relation in $K_0(\mathrm{Var}_{k})$:
    \[
    [M_{g,G-e}] - [M_{g,G}]
    = [M_{g,G/e}]. \qedhere
    \]
\end{proof}
\begin{remark}
We expect that the deletion--contraction relation of Theorem~\ref{Gro_Class} also holds for the class of the fine moduli stack $\mathcal{M}_{g,G}$ in the Grothendieck ring of Deligne--Mumford stacks over $k$. This would allow one to study stack-theoretic invariants of graphically stable spaces, such as their orbifold Euler characteristics. 
\end{remark}

\begin{remark}
    It is tempting to look for a similar deletion--contraction relation for the compactification $\overline{M}_{g,G}$. However, the coarse chamber decomposition of Hassett spaces \cite{Hassett} tells us sometimes even when we delete an edge, the coarse moduli space does not change at all. For example,
    \begin{align*}
        \overline{M}_{g,K_n} \simeq  \overline{M}_{g,K_n-e}  \simeq \overline{M}_{g,K_n-e-f},
    \end{align*}
    where $e,f$ are any two edges of the complete graph.
\end{remark}

\subsubsection{A wall-crossing formula for Grothendieck classes}\label{sec:alternative_GroClass}
In \cite{Sidd_TopHass}, the authors studied the Grothendieck class of the moduli space of smooth weighted  curves. A straightforward adaption of their result yields the following: 
\begin{proposition}[{\!\cite[Proposition 4.1]{Sidd_TopHass}}]\label{sidd}
In $K_0(\mathrm{Var}_{k})$, for a $g$-stable graph $G$ we have 
\begin{align}\label{gro_sum}
    [M_{g,G}] = \sum_{r= 0}^n N_{G,r} [M_{g,r}],
\end{align}
where $N_{G,r}$ is the number of $G$-stable partitions (Definition \ref{def:g_stable_partition}) of length $r$. Moreover, if $2g-2+r \leq 0$, we set $[M_{g,r}] = 0$. 
\end{proposition}
This proposition is obtained by stratifying the boundary of the partial compactification $M_{g,n}\subseteq M_{g,G}$ and can be viewed as a \enquote{wall-crossing} formula for Grothendieck classes, analogous to the wall-crossing formula for $\psi$-class integrals \eqref{wallcross}.
It turns out one can also express chromatic polynomials in a basis such that the coefficients are exactly $N_{G,r}$: 
\begin{align}\label{chrom}
    \chi_G(x) = \sum_{r = 0}^n N_{G,r} (x)_r,
\end{align}
where $(x)_r$ is the falling factorial:  
\begin{align*}
    (x)_r := \begin{cases}
    1 & \text{if } r = 0, \\
     x(x-1)\cdots (x-r+1) & \text{if } r \geq 1.
    \end{cases}
\end{align*}
This expression can be deduced as follows: a proper $x$-colouring using exactly $r$ colours first partitions $[n]$ into $r$ non-empty independent sets, counted by $N_{G,r}$, and then assigns $r$ distinct colours to these independent sets, giving $(x)_r$ choices. 

\begin{proof}[Second proof of Theorem \ref{Gro_Class}]
    By Definition \ref{def_chrom}, chromatic polynomial satisfies the deletion--contraction property. So by formally replacing $(x)_r$ by $[M_{g,r}]$,  we obtain another proof of Theorem~\ref{Gro_Class}. \qedhere
\end{proof}

\subsection{Graph invariants from motivic invariants}
A \textit{generalised Euler characteristic} with value in a ring $R$ is a ring homomorphism:
\begin{align*}
    \mathsf{e}\colon \grogrp \longrightarrow R.
\end{align*}
By the definition of Grothendieck ring of varieties, these morphisms satisfy the cut-and-paste property and are multiplicative with respect to fibre products over $k$. That is for a closed subvariety $Z$ of $X$, we have
\begin{align*}
    \mathsf{e}([X]) = \mathsf{e}([Z]) + \mathsf{e}([X \setminus Z]); \quad 
    \mathsf{e}([X \times_k Y]) = \mathsf{e}([X])\mathsf{e}([Y]).
\end{align*}

Given a non-negative integer $g$ and a generalised Euler characteristic $\mathsf{e}$, we can define a graph invariant:
\begin{align}\label{Gamma_graph_invarient}
    \Gamma_{g,\mathsf{e}}\colon \{\text{$g$-stable graphs} \} &\longrightarrow R \\
    G &\longmapsto \mathsf{e}([M_{g,G}]).\nonumber
\end{align}
As an immediate consequence of Theorem \ref{Gro_Class}, we deduce:
\begin{proposition}\label{prop_invarient} Let $f$ be an edge, then $\Gamma_{g,\mathsf{e}}$ satisfies the deletion--contraction relation:
    \begin{align*}
        \Gamma_{g,\mathsf{e}}(G)= \Gamma_{g,\mathsf{e}}(G-f)-\Gamma_{g,\mathsf{e}}(G/f).
    \end{align*}
\end{proposition}

In the rest of this section, we will see how some of these graph invariants $\Gamma_{g,\mathsf{e}}$ are related to existing graph invariants. 

\begin{example}\label{example_euler} Here are some examples of the generalised Euler characteristics that will be recalled later.
\begin{enumerate}
    \item Over $k = \mathbb{C}$, the usual Euler characteristic is one such example: 
        \begin{align*}
        \chi \colon K_0(\mathrm{Var}_{\mathbb{C}}) \longrightarrow \mathbb{Z}; \quad
        [X] \longmapsto \chi(X(\mathbb{C}));
    \end{align*}
    with $\chi(\lef) = 1$. A more general example is the \textit{virtual Poincaré polynomial} from Deligne's weight filtration $W_{\bullet}H_c^*(X)$:
    \begin{align*}
        P_{\bullet}^{\mathrm{vir}}(t)\colon K_0(\mathrm{Var}_{\mathbb{C}}) &\longrightarrow \mathbb{Z}[t]; \quad [X] \longmapsto P_{X}^{\mathrm{vir}}(t) :=\sum_{m=0}^{2d}\chi_c^{m}(X)t^m ,
    \end{align*}
    where \begin{align*}
        \chi_c^{m}(X) := \sum_{i = 0}^{2d}(-1)^i \mathrm{dim \, Gr}_{m}^W H_c^{i}(X(\mathbb{C}); \mathbb{Q}).
    \end{align*}
    Here $d$ denotes the complex dimension of $X$ and we have $P_{\lef}^{\mathrm{vir}}(t) = t^2$. The evaluations of the virtual Poincaré polynomial are also examples of generalized Euler characteristics: 
    \begin{itemize}
    \item The ordinary Euler characteristic is obtained by the specialization
    \[
        P_X^{\mathrm{vir}}(1)=\chi_c(X)=\chi(X),
    \]
    where the second equality follows, for example, from \cite[p.~141]{Fulton_toric}.

    \item The \textit{weight-$0$ compactly supported Euler characteristic} is obtained by the specialization
    \[
        P_X^{\mathrm{vir}}(0) = \chi_c^0(X)\colon K_0(\mathrm{Var}_{\mathbb C}) \longrightarrow \mathbb Z.
    \]
    If $X$ is smooth, then Poincaré duality for mixed Hodge structures (e.g. \cite[§6.3]{MHS_book}) gives
    \[
        \dim \operatorname{Gr}_m^W H_c^j(X)
        =
        \dim \operatorname{Gr}_{2d-m}^W H^{2d-j}(X).
    \]
    Therefore in this case $\chi_c^0(X)$ agrees with the \textit{top-weight Euler characteristic} $\tw(X)$, defined by
    \[
        \tw(X)
        :=
        \sum_{i=0}^{2d}
        (-1)^i
        \dim \operatorname{Gr}_{2d}^W H^i(X(\mathbb C);\mathbb Q).
    \]
    Indeed, setting $m=0$ in Poincaré duality identifies the weight-$0$ compactly supported cohomology of $X$ with the weight-$2d$ ordinary cohomology of $X$, and the signs agree because $2d$ is even. Moreover, we have $\chi_c^0(\lef) = \tw(\lef) = 0$.
\end{itemize}
     \item Over $k = \mathbb{R}$, the usual Euler characteristic does not satisfy the cut-and-paste property but compactly supported  Euler characteristic does, and this defines a ring homomorphism.
    \begin{align*}
        \chi_c \colon K_0(\mathrm{Var}_{\mathbb{R}}) \longrightarrow \mathbb{Z}; \quad [X] \longmapsto \chi_c(X(\mathbb{R})).
    \end{align*}
    In particular, $\chi_c(\lef) = -1$.

    \item Over $k = \mathbb{F}_q$, every variety has finitely many $\mathbb{F}_q$-points. Point-counting also defines a generalised Euler characteristic:
    \begin{align*}
        \# \colon K_0(\mathrm{Var}_{\mathbb{F}_q}) \longrightarrow \mathbb{Z}; \quad
        [X] \longmapsto |X(\mathbb{F}_q)|,
    \end{align*}
    with $\#(\lef) = q$.
    
\end{enumerate}
    
\end{example}

\subsection{Genus 0} The main result of this subsection is an explicit formula of the Grothendieck class of $\mzerog$ over any field. Let $G$ be a $0$-stable graph. Since $G$ is non-empty, non-edgeless, and non-bipartite, we have
\[
    \chi_G(0)=\chi_G(1)=\chi_G(2)=0.
\]
Indeed, $\chi_G(0)=0$ for every non-empty graph, $\chi_G(1)=0$ for every non-edgeless graph, and $\chi_G(2)=0$ for every non-bipartite graph. It follows that
\begin{align}\label{eq:cancel}
    x(x-1)(x-2)\mid \chi_G(x).
\end{align}
Thus the following definition is well-defined.

\begin{definition}\label{def:Q_G}
For any $0$-stable graph $G$, we define the polynomial
\[
    Q_G(x):=\frac{\chi_G(x+1)}{x^3-x}\in\mathbb Z[x].
\]
\end{definition}
\hypertarget{GroThmZeroTarget}{}
\begin{theorem}[{Theorem \hyperref[thm:E]{E}}]\label{Gro_thm_zero}
Let $G$ be a $0$-stable graph. In $K_0(\mathrm{Var}_{k})$, we have
    \begin{align}\label{Gro_MzeroG}
        [\mzerog] = Q_G(\lef) \in\mathbb Z[\lef].
    \end{align}
\end{theorem}
\begin{example}\label{ex_groclass_DM}
    Let $G = K_n$. The chromatic polynomial of the complete graph is given by:
    \[\chi_G(x+1) = (x+1)x(x-1) \cdots (x-n+2).\]
    By Theorem \ref{Gro_thm_zero}, we expect the Grothendieck class of $\mzerog$ to be:
    \begin{align}
        [\mzerog] = [M_{0,n}] = (\lef -2)(\lef - 3)\cdots (\lef - n +2) \in K_0(\mathrm{Var}_{k}).
    \end{align}
    Indeed, this identity can also be derived inductively by considering the forgetful morphism     $M_{0,n} \rightarrow M_{0,n-1}$, which is a Zariski local fibration with fibres isomorphic to $\mathbb{P}^1 \setminus \{n-1 \, \mathrm{points}\}$.
\end{example}

\begin{example} Let $G = K_2 * E_n$. The Losev--Manin space $\overline{M}_{0,G}$ is in fact toric \cite{LosevManin}, and the loci of smooth curves is the dense torus i.e. $\mzerog \simeq \mathbb{G}_m^{n-1}$. Now $G$ has chromatic polynomial:
\[\chi_G(x+1) = x(x+1) \chi_{E_n}(x-1) = x(x+1)(x-1)^n.\]
Thus, Theorem~\ref{Gro_thm_zero} implies that the Grothendieck class of $\mzerog$ is $(\lef-1)^{n-1}$,
which coincides with the class of the algebraic torus $\mathbb{G}_m^{n-1}$.
    
\end{example}
\begin{example}
    Let $G = E_n \sqcup K_3$, and recall that $M_{0,G} \simeq (\Pone)^n$ (Example \ref{Ex_Gra}). The chromatic polynomial is:
    \begin{align*}
        \chi_G(x+1) = \chi_{E_n}(x+1) \chi_{K_3}(x+1) = (x+1)^n (x+1)x(x-1).
    \end{align*}
    Thus, Theorem~\ref{Gro_thm_zero} implies that the Grothendieck class of $\mzerog$ is $(\lef+1)^n$,
which coincides with the class of $(\Pone)^n$.
\end{example}
This following relation will be useful when studying the compactification $\overline{M}_{0,G}$. 
\begin{corollary} Let $G_1,...,G_t$ be $0$-stable graphs, then in $K_0(\mathrm{Var}_{k})$ we have
\begin{align*}
    (\lef^3-\lef)^{t-1}[M_{0,G_1}] [M_{0,G_2}]  \cdots [M_{0,G_t}] =[M_{0,G_1 \sqcup G_2 \sqcup \cdots \sqcup G_t}].
\end{align*}
\end{corollary}
\begin{proof}
    This relation follows from Theorem \ref{Gro_thm_zero} and a graph-theoretical observation:
    \begin{align*}
        \chi_{G_1}(x)\chi_{G_2}(x) \cdots \chi_{G_t}(x) = \chi_{G_1 \sqcup G_2 \sqcup \cdots \sqcup G_t}(x),
    \end{align*}
    which can be proven easily using a graph colouring argument.
\end{proof}

In what follows, we explain the proof of Theorem \ref{Gro_thm_zero} and discussed the combinatorics behind various generalized Euler characteristics of $\mzerog$ in details.

\subsubsection{Graphical configuration spaces} Let $G$ be any graph. We define the \textit{graphical configuration space} of a $k$-variety $M$ to be
\begin{align*}
    \mathrm{Conf}_G(M) := M^n - \bigcup_{ ij \in E(G)}\{(p_1,...,p_n) \in M^n: \, p_i = p_j\}.
\end{align*}
Such configuration spaces have been studied  in \cite{EulChrom,Khoroshkin_Lyskov,GraphicalConf}. 

By studying the Thom--Gysin exact sequence associated to the pair $\mathrm{Conf}_G(M) \subseteq \mathrm{Conf}_{G-e}(M)$, Eastwood and Huggett \cite{EulChrom} showed that the complex Euler characteristic of $\text{Conf}_G(M)$ satisfies a deletion--contraction relation similar to the one in Theorem \ref{Gro_Class}:
\begin{align}\label{del_con_euler}
    \chi(\mathrm{Conf}_G(M)) =  \chi(\mathrm{Conf}_{G-e}(M)) -  \chi(\mathrm{Conf}_{G/e}(M)).
\end{align}
We generalize their result as follows.
\begin{lemma}In $K_0(\mathrm{Var}_{k})$, we have
    \begin{align}\label{del_con_conf}
        [\mathrm{Conf}_G(M)] = [\mathrm{Conf}_{G-e}(M)] - [\mathrm{Conf}_{G/e}(M)].
    \end{align}
\end{lemma}
\begin{proof}
    Analogous to the proof of Theorem \ref{Gro_Class}.
\end{proof}

\begin{proposition} \label{Gro_Conf}
In $K_0(\mathrm{Var}_{k})$, we have $[\mathrm{Conf}_G(M)] = \chi_G([M])$.
\end{proposition}
\begin{proof}
    Consider the base case which is when $G = E_n$, the edgeless graph. It is clear from the definition that $\mathrm{Conf}_{E_n}(M) = M^n$, thus $[\mathrm{Conf}_{E_n}(M)] = [M]^n$. Together with the deletion--contraction relation (\ref{del_con_conf}), we obtain the desired result.
\end{proof}

\begin{corollary}\label{Gro_Conf_Pone}
    In $K_0(\mathrm{Var}_{k})$, we have $[\mathrm{Conf}_G(\Pone)] = \chi_G(\lef+1) \in \mathbb{Z}[\lef]$.
\end{corollary}
\begin{proof}
Substitute $[\Pone] = \lef+1$ into the identity in Proposition~\ref{Gro_Conf}.   
\end{proof}

\subsubsection{The Grothendieck class of $M_{0,G}$}
Let $G$ be a $0$-stable graph; recall that this means that $G$ has at least three vertices and is neither edgeless nor bipartite. By definition, we have
\[
    M_{0,G}
    \simeq
    \mathrm{Conf}_G(\Pone)/\mathrm{PGL}_2.
\]
The $0$-stability assumption ensures that every configuration in $\mathrm{Conf}_G(\Pone)$ contains at least three distinct points. Consequently, the action of $\mathrm{PGL}_2$ is free, and the quotient above is well-defined. We also get an étale local $\mathrm{PGL}_2$-fibration:
\begin{align} \label{quotient_map_pgl2}
        \pi \colon  \mathrm{Conf}_{G}(\mathbb{P}^1) \longrightarrow \mathrm{Conf}_{G}(\mathbb{P}^1)/ \mathrm{PGL}_2 = \mzerog.
    \end{align}

\begin{remark} In \cite{EulChrom},
    using the identity (\ref{del_con_euler}) the authors showed that:
    \begin{align} \label{euler_conf}
    \chi(\mathrm{Conf}_G(\Pone({\mathbb{C}})) = \chi_G(2).
    \end{align}
This result holds for any graph and can also be deduced by substituting $\lef = 1$ into the formula in Corollary \ref{Gro_Conf_Pone}. It is therefore tempting to study $\chi(M_{0,G}(\mathbb{C}))$ via the fibration $\pi$ in (\ref{quotient_map_pgl2}), which leads to
\begin{align*}
    \chi(\mathrm{Conf}_G(\Pone(\mathbb{C}))) = \chi(\mathrm{PGL}_2(\mathbb{C})) \chi(M_{0,G}(\mathbb{C})).
\end{align*}
But since we are only considering $0$-stable graphs that cannot be bipartite, so $\chi_G(2) = 0$. We also have $\chi(\mathrm{PGL}_{2}(\mathbb{C})) = 0$, which can be deduced from Proposition \ref{Gro_PGL} by substituting $\lef = 1$. Therefore, we cannot conclude anything about the complex Euler characteristic of $M_{0,G}$ directly this way. However, it turns out by working with the Grothendieck classes, which are finer invariants than the Euler characteristics, we can actually say a lot more about $\mzerog$. 
\end{remark}

\begin{lemma}\label{Zariski}
    The quotient map $\pi$ (\ref{quotient_map_pgl2}) is also an Zariski local fibration. 
\end{lemma}
\begin{proof}
    Since $G$ is not bipartite nor edgeless, it must contain an odd cycle $C \subset G$.
    For any ordered triple $(i,j,k)$ of vertices of $C$, we define 
    \begin{align*}
        U_{ijk} := \{[\Pone,x_1,...,x_n] \in \mzerog: \, x_i,x_j,x_k \text{ are pairwise distinct} \}. 
    \end{align*}
    It is clear that these sets are preserved by the action of $\mathrm{PGL}_2$ and they are Zariski open sets. We claim that $\{U_{ijk}\}_{\{i,j,k\}\in\mathcal I_C}$ is an open covering of $\mzerog$, here $\mathcal{I}_C$ denotes the set of ordered three-element subsets of $V(C)$. Indeed, suppose $[\Pone,x_1,...,x_n] \in \mzerog$ that is not contained in any of the $U_{ijk}$, then for all $i \in C$ we must have $x_i \in \{a,b\}$, for some $a,b \in \Pone$. This gives a $2$-colouring of the odd cycle $C$ which is not possible. 
    
    Now we show that $\pi$ is a trivial principal $\mathrm{PGL}_2$-bundle over $U_{ijk}$ and hence our lemma follows. Indeed, since the $\mathrm{PGL}_2$ action on $\Pone$ is simply $3$-transitive, we can rigidify $[\Pone,x_1,...,x_n] \in U_{ijk}$ by sending $(x_i,x_j,x_k)$ to $(0,1,\infty)$ and the rest points becomes the cross-ratio of itself with $x_i,x_j,x_k$. This defines a local section of $\pi$ over $U_{ijk}$ and hence trivialises it.
\end{proof}
Next, we prove the polynomiality part of Theorem \ref{Gro_thm_zero}. 
\begin{lemma}\label{lem:poly_0}
    The Grothendieck class of $\mzerog$ belongs to the subring $\mathbb{Z}[\lef] \subset K_0(\mathrm{Var}_{k})$ i.e. $[\mzerog] \in \mathbb{Z}[\lef]$.
\end{lemma}
\begin{proof} Let $C$ be an odd cycle and
 recall the open cover $\{U_{ijk}\}_{\{i,j,k\}\in\mathcal I_C}$ of $\mzerog$ from the proof of Lemma \ref{Zariski}. Notice that for each triple $\{i,j,k\}$ we have
 \[U_{ijk} \simeq M_{0,G_{ijk}},\]
 where $G_{ijk}$ is the graph obtained by adding in the edges $\{ij,ik,jk\}$. Observe that $G_{ijk}$ contains a copy of $K_3$ spanned by $i,j,k$. More generally, the intersection of an arbitrary intersection of these open sets $U_{ijk}$ can always be identified with $M_{0,H}$ such that $H$ contains a copy of $K_3$. By the inclusion–exclusion principle, we have:
\[
[M_{0,G}]=
\sum_{\varnothing \neq A \subseteq \mathcal{I}_C}
(-1)^{|A|+1}
\left[
\bigcap_{(i,j,k)\in A} M_{0,G_{ijk}}
\right] \in K_0(\mathrm{Var}_{k}).
\]
Therefore, it suffices to show $[M_{0,H}] \in \mathbb{Z}[\lef] $ for every graph $H$ that contains a copy of $K_3$. We argue by induction on the number of edges of $H$ lying outside the distinguished $K_3$. If there are no such edges i.e. $H = K_3 \sqcup E_{n-3}$, then in this case we have $M_{0,K_3 \sqcup E_{n-3}} \simeq (\Pone)^{n-3}$ (Example \ref{Ex_Gra}). Consequently, we have $[M_{0,K_3 \sqcup E_{n-3}}] = (\lef+1)^{n-3} \in \mathbb{Z}[\lef] $. The inductive step follows straightforwardly from the deletion--contraction relation satisfied by $[M_{0,H}]$ as shown in Theorem \ref{Gro_Class}.
\end{proof}

The following facts about $K_0(\mathrm{Var}_{k})$ are well-known to the experts and we include it for completeness. 

\begin{proposition}\label{Gro_PGL} In $K_0(\mathrm{Var}_{k})$, we have $[\mathrm{PGL}_2] = \lef^3 - \lef$. 
\end{proposition}
\begin{proof}
    We know $\mathrm{PGL}_2$ is the open locus inside $\mathbb{P}(\mathrm{Mat}_{2\times 2}) \cong \mathbb{P}^3$ where the determinant is non-zero. The complement locus of which the determinant is zero can be identified with the quadric $\{xw-yz = 0\} \subset \mathbb{P}^3$, and it is isomorphic to $\mathbb{P}^1 \times \mathbb{P}^1$. Therefore, we get the desired equality:
    \[[\mathrm{PGL}_2] = [\mathbb{P}^3] - [\mathbb{P}^1 \times \mathbb{P}^1] 
        = \lef^3+\lef^2+\lef +1- (\lef+1)(\lef+1) 
         = \lef^3 - \lef. \qedhere\]
\end{proof}

\begin{proposition} \label{Lef_relation}
    For any field $k$, the following ring homomorphism is injective:
    \begin{align*}
        \varphi \colon \mathbb{Z}[x] \longrightarrow K_0(\mathrm{Var}_{k}); \quad x \longmapsto \mathbb{L}.
    \end{align*}
    In other words, we have a ring isomorphism $\mathbb{Z}[x] \simeq \mathbb{Z}[\lef]$.
\end{proposition}
\begin{proof}
Recall the virtual Poincaré polynomial introduced in Example~\ref{example_euler}, which is defined for varieties over $\mathbb{C}$. More generally, for varieties over an arbitrary field $k$, there is a Poincaré polynomial
\[
    P_{\bullet}(t)\colon K_0(\mathrm{Var}_k)\longrightarrow \mathbb{Z}[t]
\]
satisfying $P_{\lef}(t)=t^2$,
see \cite[Theorem~2.10]{Nicaise_poincare_poly}. We therefore obtain a composition
\[
    \mathbb{Z}[x]
    \xlongrightarrow{\varphi}
    K_0(\mathrm{Var}_k)
    \xrightarrow{P_{\bullet}(t)}
    \mathbb{Z}[t],
\]
which sends $x$ to $t^2$, and it is clearly injective. Hence, $\varphi$ must also be injective. 
\end{proof}

We are now ready to prove Theorem \ref{Gro_thm_zero}.

\begin{proof}[First proof of Theorem \ref{Gro_thm_zero}]
    Since multiplication in Grothendieck ring respects Zariski local fibration, 
    by Lemma \ref{Zariski}, we obtain the following relation:
    \begin{align}
        [\mathrm{PGL}_2][\mzerog] &= [\mathrm{Conf}_G(\mathbb{P}^1)] \in K_0(\mathrm{Var}_{k})]
        \intertext{By Corollary \ref{Gro_Conf_Pone}, Lemma \ref{lem:poly_0} and Proposition \ref{Gro_PGL}, we obtain a relation in the subring $\mathbb{Z}[\lef]$:}
        (\lef^3 - \lef)[\mzerog] &= \chi_G(\lef + 1) \in \mathbb{Z}[\lef].
    \end{align}
    Recall from (\ref{eq:cancel}) that 
    \[ (\lef^3 - \lef) \mid  \chi_G(\lef + 1),\]
    and the polynomial $Q_G$ was defined as their quotient (Definition \ref{def:Q_G}):
    \[Q_G(\lef):= \frac{\chi_G(\lef + 1)}{\lef^3 - \lef}.\]
    Therefore, we deduce that
    \[(\lef^3 - \lef)[\mzerog] = (\lef^3 - \lef) Q_G(\lef) \in \mathbb{Z}[\lef].\]
    By Proposition~\ref{Lef_relation}, we have $\lef^3-\lef\neq 0$ in $\mathbb{Z}[\lef]$. Since $\mathbb{Z}[\lef]$ is an integral domain, we may therefore cancel the factor $\lef^3-\lef$ from both sides, yielding the desired result.
\end{proof}
\begin{remark}
We remark that it is important to work in the subring $\mathbb{Z}[\lef]$, as $K_0(\mathrm{Var}_{k})$ is not an integral domain \cite{Gro_ring_domain}.
\end{remark}
We present a second proof utilizing Proposition \ref{sidd}. 
\begin{proof}[Second proof of Theorem \ref{Gro_thm_zero}]
    Recall from Example \ref{ex_groclass_DM}, via the forgetful map $M_{0,r+1} \rightarrow M_{0,r}$, we can easily deduce 
    \begin{align}
        [M_{0,r}] = (\lef -2)(\lef - 3)\cdots (\lef - r +2) \in K_0(\mathrm{Var}_{k}).
    \end{align}
    Substitute this identity into (\ref{gro_sum}), we get
    \[ [\mzerog] = \sum_{r = 3}^n N_{G,r} (\lef -2)(\lef - 3)\cdots (\lef - r +2).\]
    Since $G$ is $0$-stable, we have $N_{G,0}= N_{G,1}=N_{G,2} =0$. Thus in $\mathbb{Z}[\lef]$, we deduce once again that:
    \begin{align*}
        (\lef+1)\lef(\lef-1) [\mzerog] = \sum_{r = 0}^n N_{G,r} (\lef+1)\lef(\lef-1)(\lef -2)\cdots (\lef - r +2) 
         = \chi_G(\lef+1),
    \end{align*}
    where the latter equality follows from (\ref{chrom}).
\end{proof}

\subsubsection{Application: Generalised Euler characteristics of $\mzerog$} Everything we have discussed so far work over any field $k$. As an application of Theorem \ref{Gro_thm_zero}, we will now study the topology of $\mzerog$ over specific fields.

\begin{corollary}[{Corollary \hyperref[cor:F]{F}}] \label{lhop}The chromatic polynomial of $G$ determines the following information about $\mzerog$.
\begin{enumerate}
    \item Over $k = \mathbb{C}$, 
    \begin{align}
        \chi(\mzerog(\mathbb{C})) &= \frac{1}{2}\left. \frac{d}{dx} \chi_G(x) \right|_{x = 2}. \label{eul}
        \end{align}
        This quantity coincides with a conjectural count in \cite{GraphScattering} of critical values of certain CHY scattering potential associated to $G$, see Section \ref{subsec_euler_char} for a detailed discussion.
        \item Over $k = \mathbb{C}$,
        \begin{align}
        \tw(\mzerog(\mathbb{C})) &= -\left.\frac{d}{dx} \chi_G(x) \right|_{x = 1}. \label{tw}
        \end{align}
        This quantity is, up to sign, Crapo's $\beta$-invariant of $G$, which can also be interpreted as the number of acyclic orientations of $G$ with a specified unique source and a specified unique sink. See Section \ref{dual_complex} for a detailed discussion.
        \item Over $k = \mathbb{C}$,
        \begin{align}
        \pvir(t) &= \frac{\chi_G(t^2 + 1)}{t^6 - t^2} \in \mathbb{Z}[t^2]. \label{vir}
        \end{align}

    \item Over $k = \mathbb{R}$,
    \begin{align}\label{real}
        \chi_c(\mzerog(\mathbb{R})) = \frac{1}{2} \left. \frac{d}{dx} \chi_G(x) \right|_{x = 0}.
    \end{align}
     This quantity can be interpreted as, up to sign, the number of acyclic orientations of $G$ with a specified unique source. See Section \ref{real_locus} for a detailed discussion. 

    \item Over $k = \mathbb{F}_q$,
    \begin{align} \label{Fq}
        \# (\mzerog(\mathbb{F}_q)) = \frac{\chi_G(q + 1)}{q^3 - q} \in \mathbb{Z}[q].
    \end{align}
\end{enumerate}
\end{corollary}

\begin{proof}[Proof of Corollary \ref{lhop}] By Proposition~\ref{Lef_relation}, over $\mathbb{R}$ or $\mathbb{C}$, an element from  $\mathbb{Z}[\lef]$ can genuinely be viewed as a differentiable or holomorphic function of one variable $\lef$. Recall from Theorem \ref{Gro_thm_zero} we have
\begin{align*}
        [\mzerog] = Q_G(\lef) = \frac{\chi_G(\lef + 1)}{\lef^3 - \lef} \in \mathbb{Z}[\lef].
    \end{align*}
Therefore, we obtain (\ref{eul}), (\ref{tw}) and (\ref{real}) by evaluating the above formula using L'Hôpital's rule, at $\lef = 1, 0, -1$ respectively. 

Moreover, (\ref{vir}) \& (\ref{Fq}) are obtained by simply evaluating the above formula at $\lef = t^2,q$ respectively. The choices of these evaluations were explained in Example \ref{example_euler}.
\end{proof}

\begin{remark}
     We justify why $\frac{1}{2} \! \left. \frac{d}{dx} \chi_G(x) \right|_{x = 0}$ and $\frac{1}{2} \! \left. \frac{d}{dx} \chi_G(x) \right|_{x = 2}$ are integers. We know $\chi_G$ has no constant term, write 
    \begin{align*}
        \chi_G(x) = a_1x+a_2x^2+...+a_nx^n \implies \frac{d}{dx} \chi_G(x) = a_1+2a_2x+...+na_nx^{n-1}.
    \end{align*}
    Therefore both numbers are integral if and only if $a_1$ is even. By the assumption on stability, $\chi_G(2) = 0$, hence $0 \equiv 2a_1 \,\mathrm{ mod }\, 4$ and $a_1$ is even.
\end{remark}
\begin{remark} We give a combinatorial interpretation of  (\ref{Fq}). From the moduli functor point of view, we know the $\mathbb{F}_q$-points of $\mzerog$ are precisely the isomorphism classes of smooth $G$-stable rational curves  over $\mathrm{Spec } \, \mathbb{F}_q$ :
    \begin{align*} 
        \mzerog(\mathbb{F}_q) &=\{(x_1,...x_{|V(G)|} )\in (\mathbb{P}^1(\mathbb{F}_q))^{|V(G)|};  \; x_i \neq x_j \text{ if } ij\in E(G)\}/ \sim.
         \end{align*}
        Since $\mathbb{P}^1(\mathbb{F}_q)$ can be identified with $\{1,...,q+1\}$, we are essentially picking a number from $\{1,...,q+1\}$ for each vertex, such that the numbers must be distinct if two vertices connected by an edge. This is exactly a $(q+1)$-proper colouring of $G$, up to relabelling the colours $\{1,...,q+1\}$ using an element from $\mathrm{PGL}_2(\mathbb{F}_q)$. 
\end{remark}


\subsubsection{Boundary complex}\label{dual_complex} We work over $k = \mathbb{C}$ and we aim to discuss the combinatorics behind (\ref{tw}): 
    \begin{align*}
        \tw(\mzerog(\mathbb{C})) &= -\left.\frac{d}{dx} \chi_G(x) \right|_{x = 1}.
    \end{align*}
\begin{definition}[\!\cite{CRAPO_beta}]
        When $G$ has at least one edge, the $\beta$-invariant of $G$ is defined as 
\begin{align*}
    \beta(G) := (-1)^{|V(G)|-k(G)+1} \left.\frac{d}{dx} \chi_G(x) \right|_{x = 1}.
\end{align*}
Where $k(G)$ is the number of connected components of $G$.
\end{definition}
The $\beta$-invariants also have an interpretation as counts of certain orientations due to the work of Greene and Zaslavsky.
\begin{theorem}[{\!\cite[Theorem 7.2]{CurtisZaslavsky1983}}] \label{Las}  Pick two adjacent vertices $P$ and $Q$. Then $\beta(G)$ equals to      the number of acyclic orientations of $G$ that have a unique source at $P$ and a unique sink at $Q$. Moreover, it is independent of the choice of $P$ and $Q$, additionally we regard an isolated vertex as neither a source nor a sink. We denote the set of such orientations as $\mathrm{ACO}_{P,Q}(G)$. 
\end{theorem}
In summary, we may rewrite (\ref{tw}) as:
\begin{proposition}\label{prop_tw} For any $P,Q$ that are joined by an edge,  we have
    \begin{align*}
        \tw(\mzerog(\mathbb{C})) = (-1)^{n-k(G)}\beta(G) =  (-1)^{n-k(G)}|\mathrm{ACO}_{P,Q}(G)|.
    \end{align*}
\end{proposition}

The top-weight Euler characteristic of a complex variety is closely related to the combinatorics of the boundary of a simple normal crossing (snc) compactification of that variety, and this will give a further combinatorial simplification of Proposition $\ref{prop_tw}$.

Recall $\mzerog \subset \overline{M}_{0,G}$ is a snc compactification and let $\Delta_{0,G}$ be the \textit{boundary complex} of this compactification. The complex $\Delta_{0,G}$ is the cell complex whose $k$-faces corresponds to codimension $k+1$ strata. Such complex has been studied in the following cases:
\begin{enumerate}
    \item When $G = K_n$, Vogtmann in \cite{Vogtmann_1990} showed that $\Delta_{0,G}$ is homotopic to a wedge of $(n-2)!$ spheres of dimension $n-4$. 
    \item When $G = K_n * E_m$ with $n \geq 2$, Cavalieri, Hampe, Markwig and Ranganathan in \cite{CAVALIERI_HAMPE_MARKWIG_RANGANATHAN_2016} showed that $\Delta_{0,G}$ can be identified with the link of the Bergman fan of a graphical matroid, which is homotopic to a wedge of $(n-2)!(n-1)^m$ spheres of dimension $n+m-4$ by \cite{Ardila06}. 
\end{enumerate}

It is well-known that Deligne’s weight spectral sequence gives an isomorphism between the top-weight cohomology of a smooth variety $X$ and the reduced homology of the boundary complex of a snc compactification of $X$ (e.g. \cite{CGP21}, Theorem 5.8). In our case this gives:
\begin{align} \label{eq_tw_isom}
    \mathrm{Gr}_{2n-6}^WH^{2n-6-k}(\mzerog; \mathbb{Q}) \cong \widetilde{H}_{k-1}(\bound; \mathbb{Q}).
\end{align}
An immediate consequence is that
\begin{align}\label{CGP}
    \tw(\mzerog) = - \widetilde{\chi}(\Delta_{0,G}).
\end{align}
The boundary $\overline{M}_{0,G} \setminus \mzerog$ is stratified, and each stratum can be combinatorially represented by a tree --- a genus $0$ stable $G$-graph (\!\cite[Definition 2.6]{SRI}). We denote the set of such trees as $\Gamma_{0,G}$. We have the following simple relation: 
\begin{lemma}\label{prop_euler_complex}We have
    \begin{align}
        \sum_{T \in \Gamma_{0,G}} (-1)^{|E(T)|} = - \widetilde{\chi}(\Delta_{0,G}),
    \end{align}
    where $|E(T)|$ is the number of bounded edges of $T$. 
\end{lemma}
\begin{proof} 
     Each $T \in \Gamma_{0,G}$ represents a stratum of $\overline{M}_{0,G}$ of codimension $|E(T)|$. So if $|E(T)| \geq 1$, the corresponding stratum is an intersection of $|E(T)|$ many boundary divisors, which in turn contributes an $|E(T)|-1$ dimensional face of $\Delta_{0,G}$. Hence, we have
     \begin{align*}
         \sum_{T \in \Gamma_{0,G}; \; |E(T)| \geq 1} (-1)^{|E(T)|-1} = \chi(\Delta_{0,G}).
     \end{align*}
     Now there is one more tree that has $|E(T)|= 0$ corresponding to the stratum $\mzerog$, so we get the desired statement:
     \[-\sum_{T \in \Gamma_{0,G}} (-1)^{|E(T)|-1} = -(-1 + \chi(\Delta_{0,G})) = - \widetilde{\chi}(\Delta_{0,G}). \qedhere\]
\end{proof}

Combining Lemma \ref{prop_euler_complex} with ($\ref{CGP}$) and Proposition \ref{prop_tw}, we obtain the following combinatorial formula:
\begin{corollary}[Corollary~{\hyperref[cor:H]{H}}]\label{beta}
\label{beta}For any $0$-stable graph $G$ and any edge $e$ with end points labelled as $P$, $Q$, we have
    \begin{align}
        \sum_{T \in \Gamma_{0,G}} (-1)^{|E(T)|}  & = (-1)^{n-k(G)}\beta(G) = (-1)^{n-k(G)}|\mathrm{ACO}_{P,Q}(G)|.
    \end{align}
\end{corollary}

When $G$ has two dominant vertices, Blankers, Gillespie and Levinson in \cite{SRI} constructed certain \textit{Sign Reversing Involution (SRI)} to evaluate the sum in Corollary \ref{beta}, and they arrived at the following theorem.
\begin{theorem}[\!\cite{SRI}, Theorem 3.1]\label{Thm_SRI} Let $P,Q$ be two dominant vertices of $G$ and $H = G-\{P,Q\}$, then we have
    \begin{align*}
        \sum_{T \in \Gamma_{0,G}} (-1)^{|E(T)|} = (-1)^{n-1}|\mathrm{ACO}(H)|.
    \end{align*}
    Here $\mathrm{ACO}(H)$ denotes the set of acyclic orientations of $H$.
\end{theorem}
Corollary \ref{beta} can thus be viewed as a generalization of the above theorem to arbitrary $0$-stable graphs, and we now give a different proof of this theorem.  
\begin{proof}[Proof of Theorem \ref{Thm_SRI}]
    By Corollary $\ref{beta}$, this theorem is equivalent to showing that
    \begin{align*}
        (-1)^{n-1}|\mathrm{ACO}(H)| = (-1)^{n-k(G)}|\mathrm{ACO}_{P,Q}(G)|.
    \end{align*}
    Since $G$ has dominant vertices so it must be connected i.e. $k(G) = 1$. The two dominants vertices ensure $G$  must contain a $3$-cycle, hence $G$ is not bipartite and is $0$-stable. 
    
    Consider an acyclic orientation of $G$ with unique sink $P$ and unique source $Q$. Since both vertices are dominant, for all $v \in V(H)$, we must have orientations 
    \begin{align*}
        v \rightarrow P; \quad  v \leftarrow Q \quad \mathrm{and} \quad Q \rightarrow P.
    \end{align*}
    Therefore, the only edges haven't been oriented are the ones in $H$ and we get an acyclic orientation of $H$. Conversely, given any acyclic orientation of $H$, there is only one way to orient incident edges to $P$ and $Q$, such that $P$ is the unique sink and $Q$ is the unique source. So there is a bijection between $\mathrm{ACO}(H)$ and $\mathrm{ACO}_{P,Q}(G)$, and the theorem follows.
\end{proof}


\subsubsection{Euler characteristics and scattering potentials}\label{subsec_euler_char} We discuss the connection between the Euler characteristic of $\mzerog$ (\ref{eul})
\begin{align*}
    \chi(\mzerog(\mathbb{C})) = \frac{1}{2}\left. \frac{d}{dx} \chi_G(x) \right|_{x = 2},
\end{align*}
and certain CHY scattering potentials in particle physics \cite{CHY_scattering} associated to $G$. 

In \cite{GraphScattering}, the authors associated to each $G$ a potential on $M_{0,n}$. When $G$ is \textit{copious} (see \cite[Theorem 1.3]{GraphScattering}), the number of critical points of this potential is well-defined and this number is denoted as $\mu(G)$. On the other hand, as explained in \cite[§5]{GraphScattering}, this number $\mu(G)$ can also be viewed as the \textit{maximum likelihood degree} of a smooth very affine variety $M_G$ of dimension $n-3$. Then by the work of Huh \cite{Huh_likelihood}, we have: 
\[\mu(G) = (-1)^{n-3}\chi(M_G).\]
\vspace{-1em}
\begin{conjecture}[{\!\cite[Conjecture 7.1]{GraphScattering}}]\label{conj_scattering} For any copious graph $G$, we have
    \begin{align*}
        \mu(G) = \left| \sum_{r=3}^n N_{G,r}(-1)^{r-3} (r-3)!  \right|.
    \end{align*}
    Recall from Proposition \ref{sidd},  $N_{G,r}$ is the number of $G$-stable partitions of $[n]$ into $r$ non-empty parts.
\end{conjecture}
The authors showed this conjecture is true when $G$ has a dominant vertex \cite[Theorem 7.7]{GraphScattering}. This conjectural quantity turns out to agree with $\chi(\mzerog)$:
\begin{proposition}\label{prop:conj_scattering} Assume $G$ is $0$-stable, we have
    \begin{align*}
        \frac{1}{2}\left. \frac{d}{dx} \chi_G(x) \right|_{x = 2} = \sum_{r=3}^n N_{G,r} (-1)^{r-3} (r-3)!. 
    \end{align*}
\end{proposition}
\begin{proof}
    Recall from (\ref{chrom}) we can express $\chi_G(x)$ as
    \begin{align*}
        \chi_G(x) = \sum_{r = 0}^n N_{G,r} (x)_r,
    \end{align*}
    where $(x)_r = x(x-1)\cdots (x-r+1)$. Since $G$ is $0$-stable i.e it is not edgeless nor bipartite and has at least three vertices, we have $N_{G,0} = N_{G,1} = N_{G,2} = 0$. So we assume $r \geq 3$. Differentiate $(x)_r$ we get:
    \begin{align*}
        \frac{d}{dx} (x)_r &= x(x-1)(x-2)\cdots (x-r+1)\Bigl(\frac{1}{x}+\frac{1}{x-1}+\frac{1}{x-2}+ \cdots +\frac{1}{x-r+1}\Bigl) \\ 
        & = x(x-1)(x-3)\cdots (x-r+1) + (x)_r \Bigl(\frac{1}{x}+\frac{1}{x-1}+\frac{1}{x-3}+ \cdots +\frac{1}{x-r+1}\Bigl)
    \end{align*}
    Evaluate at $x = 2$ and notice that $(2)_r = 0$ for $r \geq 3$, we have: 
    \begin{align*}
        \left. \frac{d}{dx} (x)_r \right|_{x = 2} = 2(2-1)(2-3)\cdots (2-r+1)    
    \end{align*}
    There are $r-3$ negative factors, hence
    \begin{align*}
        \left. \frac{d}{dx} (x)_r \right|_{x = 2} = 2(-1)^{r-3}(r-3)!.
    \end{align*}
    Summing over $r$, we obtain the desired expression:
    \begin{align*}
        \frac{1}{2}\left. \frac{d}{dx} \chi_G(x) \right|_{x = 2} &= \frac{1}{2}\sum_{r=3}^n N_{G,r}  \left. \frac{d}{dx} (x)_r \right|_{x = 2} \\
        &=\sum_{r=3}^n N_{G,r} (-1)^{r-3} (r-3)!. \qedhere
    \end{align*}
\end{proof}
Therefore, when $G$ is copious and $0$-stable, Conjecture \ref{conj_scattering} is equivalent to the following:
\begin{align*}
    (-1)^{n-3}\chi(M_G) = |\chi(\mzerog)|.
\end{align*}
This statement is true if $M_G \simeq \mzerog$. Fry \cite{Fry} showed this is indeed the case when $G$ has a dominant vertex, see also \cite[Remark 7.6]{GraphScattering}. Hence, this gives another proof of Conjecture \ref{conj_scattering} in the case of dominant vertex.

\subsubsection{Poincaré polynomial}\label{sec:poincare_poly}Suppose $G$ has a dominant vertex, recall from \ref{section_hyperplane}, in this case $M_{0,G}$ is essentially the complement of a complex hyperplane arrangement:
\begin{align}
    M_{0,G} \simeq \mathrm{Conf}_{G-v}(\mathbb{C})/\mathrm{Aff}_1(\mathbb{C}).
\end{align}
By a classic result of Arnol’d--Orlik--Solomon, the mixed Hodge structure of  $H^k(\mathrm{Conf}_{G-v}(\mathbb{C}))$ is pure of weight $2k$. This remains true after quotienting by $\mathrm{Aff}_1(\mathbb{C})$ i.e. $H^k(\mzerog)$ is pure of weight $2k$ when $G$ has a dominant vertex, see \cite[Lemma 4.3.3]{Khoroshkin_Lyskov}. So in this case, we can deduce the usual Poincaré polynomial from the virtual Poincaré polynomial (Example \ref{example_euler}). Define the Poincaré polynomial of $\mzerog$ to be
\begin{align*}
    P_{\mzerog}(q) := \sum_{k=0}^{2d}  b_k (\mzerog)q^k.
\end{align*}
where $d$ is the complex dimension of $\mzerog$ and $b_k(\mzerog) := \mathrm{dim \,} H^k(\mzerog; \mathbb{Q})$. 
\begin{proposition}
    If $G$ has a dominant vertex, then we have 
    \begin{align*}
        P_{\mzerog}(q) = (-1)^{n}\frac{q^n}{q^2+1} \chi_G(1-\frac{1}{q}).
    \end{align*}
\end{proposition}
\begin{proof}
    Recall the definition of virtual Poincaré polynomial:
    \begin{align*}
        P_{\mzerog}^{\mathrm{vir}}(t) :=\sum_{m=0}^{2d} \Bigl(\sum_{j = 0}^{2d}(-1)^j \mathrm{dim \, Gr}_{m}^W H_c^{j}(\mzerog) \Bigl)t^m.
    \end{align*}
    Since $\mzerog$ is smooth, the Poincaré duality of mixed Hodge structures (e.g. \cite[§6.3]{MHS_book}) applies and gives:
    \begin{align*}
        \mathrm{dim \, Gr}_{m}^W H_c^{j}(\mzerog) = \mathrm{dim \, Gr}_{2d-m}^W H^{2d-j}(\mzerog). 
    \end{align*}
    We know $H^{2d-j}(\mzerog)$ is pure of weight $2(2d-j)$, so the only graded piece that does not vanish is when $2d-m = 2(2d-j)$: 
    \begin{align*}
         \mathrm{dim \, Gr}_{m}^W H_c^{j}(\mzerog) = \mathrm{dim \, Gr}_{2d-m}^W H^{2d-j}(\mzerog) = \begin{cases}
     \mathrm{dim \,} H^{2d-j}(\mzerog)& \text{if } m = 2j-2d, \\
     0 & \text{otherwise}.
     \end{cases}
    \end{align*}
    Therefore, the double sum collapses down to: 
    \begin{align*}
        P_{\mzerog}^{\mathrm{vir}}(t) = \sum_{j = 0}^{2d}(-1)^j b_{2d-j}(\mzerog)t^{2j-2d}
    \end{align*}
    Substituting $k = 2d-j$ ,we get
        \begin{align*}
        P_{\mzerog}^{\mathrm{vir}}(t) &= \sum_{k = 0}^{2d}(-1)^k b_{k}(\mzerog)(t^2)^{d-k} \\
         &=t^{2d} \sum_{k = 0}^{2d} b_{k}(\mzerog)(-t^{-2})^k \\
         & = t^{2d} P_{\mzerog}(-t^{-2})
    \end{align*}
    Recall (\ref{vir}): 
    \begin{align*}
         \pvir(t) = \frac{\chi_G(t^2 + 1)}{t^6 - t^2} .
    \end{align*}
    Substituting $q = -t^{-2}$ we get:
    \begin{align*}
        \frac{1}{(-1/q)^3-1/q} \chi_G(-\frac{1}{q}+1) = (-1)^{n-3} \frac{1}{q^{n-3}} P_{\mzerog}(q).
    \end{align*}
    Simplify the fractions we get
    \[-\frac{q^3}{1+q^2}
\chi_G\left(1-\frac{1}{q}\right)
=
(-1)^{n-3}\frac{1}{q^{n-3}}P_{M_{0,G}}(q).\]
Rearranging then yields the desired expression.
\end{proof}

\subsubsection{The real locus} \label{real_locus}
We now work over $\mathbb{R}$ and suppose $G$ has a distinguished vertex labelled as $s$. Recall (\ref{real}) from Corollary \ref{lhop} which states that: 
\begin{align*}
    \chi_c(\mzerog(\mathbb{R})) = \frac{1}{2} \left. \frac{d}{dx} \chi_G(x) \right|_{x = 0}.
\end{align*}
Recall from Remark \ref{rmk_aco_source} that we have the following combinatorial interpretation of this quantity.
\begin{theorem}[\!{\cite[Theorem 7.3]{CurtisZaslavsky1983}}]\label{thm_ACO_source} Denote $\mathrm{ACO}_{s}(G)$ as the set of acyclic orientations of $G$ with a unique source $s$, then we have
\begin{align*}
    (-1)^{|V(G)|-1}\left. \frac{d}{dx} \chi_G(x) \right|_{x = 0} = |\mathrm{ACO}_{s}(G)|.
\end{align*} 
\end{theorem}
  Thus we can rewrite (\ref{real}) as the following. 
 \begin{proposition} For any $0$-stable graph $G$, we have
    \begin{align} \label{prop_real_2}
    \chi_c(\mzerog(\mathbb{R}))  &=  \frac{1}{2} (-1)^{|V(G)|-1} |\mathrm{ACO}_{s}(G)|.\nonumber
    \intertext{Moreover, Poincaré duality implies that} \chi(\mzerog(\mathbb{R}))&=\frac{1}{2}|\mathrm{ACO}_{s}(G)|. 
\end{align}

 \end{proposition}

We will justify this statement in the case of $s$ being a dominant vertex, from a combinatorial point of view. Recall from Section \ref{section_hyperplane}, in the case of dominant vertex we have
\[ M_{0,G}(\mathbb{R}) \simeq \mathrm{Conf}_{G-s}(\mathbb{R})/\mathrm{Aff}_1(\mathbb{R}).\]

\begin{theorem}[Zaslavsky \cite{Zaslavsky1975FacingUT}]\label{thm_zaslavsky}
    There is a one-to-one correspondence between connected components of $\mathrm{Conf}_{H}(\mathbb{R})$ and acyclic orientations of $H$.
\end{theorem}
\begin{corollary}
    There is a one-to-one correspondence between connected components of $\mathrm{Conf}_{G-s}(\mathbb{R})$ and $\mathrm{ACO}_{s}(G)$.
\end{corollary}
\begin{proof}
    Given a connected component of $\mathrm{Conf}_{G-s}(\mathbb{R})$, this corresponds to a unique acyclic orientation of $G-s$ by Theorem~ \ref{thm_zaslavsky}. We can extend this acyclic orientation to $G$ by letting $s$ be a source. Since $s$ is dominant, $s$ must be unique. 
\end{proof}
Now we state two more facts:
\begin{itemize}
    \item Every connected component of $\mathrm{Conf}_{G-s}(\mathbb{R})$ is open convex in $\mathbb{R}^{|V(G)|-1}$, as they are defined by a set of inequalities. Hence, each connected component is homeomorphic to $\mathbb{R}^{|V(G)|-1}$.
    \item $\mathrm{Aff}_1(\mathbb{R}) = \{ax+b; a \in \mathbb{R}^{\times}, b \in \mathbb{R}\}$ has two connected components corresponds to positive $a$ and negative $a$. 
\end{itemize}

Therefore, combining these results we can re-derive (\ref{prop_real_2}) in the case of $s$ being a dominant vertex as follows: 
\begin{align*}
    \chi (\mathrm{Aff}_1(\mathbb{R}))\chi(\mzerog(\mathbb{R})) &= \chi(\mathrm{Conf}_{G-s}(\mathbb{R})) \\
    2 \cdot \chi(\mzerog(\mathbb{R}))  & = \sum_{C \in \pi_0(\mathrm{Conf}_{G-s}(\mathbb{R}))} \chi(C) \\ 
    & = \sum_{C \in \pi_0(\mathrm{Conf}_{G-s}(\mathbb{R}))} \chi(\mathbb{R}^{|V(G)|-1})\\
    & =  |\pi_0(\mathrm{Conf}_{G-s}(\mathbb{R}))|\\
    &= |\mathrm{ACO}_s(G)|.
\end{align*}

\subsection{Higher genus}\label{sec_higher_g} In this subsection we study the topology of $M_{g,G}$ when $g > 0$. Recall from Theorem \ref{Gro_Class}, we have
\begin{align*}
    [M_{g,G}] = [M_{g, G-e}] - [M_{g, G/e}]\in K_0(\mathrm{Var}_k).
\end{align*} 

In the case of positive genus, the Grothendieck class $[M_{g,G}]$ is not generally a polynomial in $\lef$ as it is in the genus $0$ case (Theorem \ref{Gro_thm_zero}). For example,  \cite[Theorem 1.4]{MgnPoly} implies that the polynomiality of $[M_{g,n}]$ fails whenever $n \geq \mathrm{max}\{\left\lceil {(25 - 3g)}/2 \right\rceil, 0\}$. 

Nonetheless, we can still compute topological invariants of $M_{g,G}$ without knowing its Grothendieck class. Throughout Section \ref{sec_higher_g}, we will work over $\mathbb{C}$ and we will focus on the usual Euler characteristics and the top-weight Euler characteristics. The top-weight Euler characteristic is again determined by the topology of its boundary complex:
\begin{align}\label{def_tw_g}
\tw (M_{g,G}) = - \Tilde{\chi}(\Delta_{g,G}).
\end{align}
Here $\Delta_{g,G}$ is the boundary complex of $\mathcal{M}_{g,G} \subset \overline{\mathcal{M}}_{g,G}$ and 
the equality follows from a generalization of (\ref{eq_tw_isom}) to the setting of normal crossings compactification $\mathcal{M}_{g,G} \subset \overline{\mathcal{M}}_{g,G}$ of smooth Deligne--Mumford stacks, see \cite[Theorem~5.8]{CGP21}.

\subsubsection{Genus $1$} We compute the Euler characteristic and the top-weight Euler characteristic in the genus $1$ case.
\begin{proposition}[{Corollary \hyperref[cor:G]{G}}]\label{lem_chi_one} Suppose $G$ has at least one vertex, we have
    \begin{align*}
        \chi(M_{1,G}) &= - \frac{1}{12}\left. \frac{d}{dx} \chi_G(x) \right|_{x = 0} + \frac{\chi_G(1)}{3} +  \frac{\chi_G(2)}{4}+\frac{\chi_G(3)}{9}-\frac{\chi_G(4)}{48}. \\
        \tw(M_{1,G})& = -\frac{1}{2} \left. \frac{d}{dx} \chi_G(x) \right|_{x = 0} + \chi_G(1) -  \frac{\chi_G(2)}{4}.
    \end{align*}
\end{proposition}
\begin{remark}
    When $G = K_n$, we have $\chi_{K_n}(x) = x(x-1)(x-2) \cdots (x-n+1)$.
    Proposition~\ref{lem_chi_one} specializes to some known results about $M_{1,n}$:
    \begin{itemize}
        \item $\chi(M_{1,n}) = 1,1,0,0 \text{ for $n=1,2,3,4$ and }  \chi(M_{1,n}) = (-1)^n(n-1)!/12 \text{ for $n \geq 5$ }$, see \cite{Getzler_M1n}.
        \item $\tw(M_{1,n}) = 0, 0 \text{ for $n=1,2$ and }  \tw(M_{1,n}) = (-1)^n(n-1)!/2 \text{ for $n \geq 3$ }$, see \cite{Chan_M1n}.
    \end{itemize}
    In particular, Proposition~\ref{lem_chi_one} provides a uniform treatment for all possible numbers of markings.
\end{remark}
Both assertions in Proposition \ref{lem_chi_one} follow from direct computations of the base cases, namely when $G = E_n$.
\begin{proposition} \label{prop_genus_1_base}
 We have
\begin{align*}
    \chi(M_{1,E_n}) &= 
    \begin{dcases} 
        1 & \text{if } n = 1, \\
        2^{n-2} + 3^{n-2} - \frac{4^{n-2}}{3} + \frac{1}{3} \quad & \text{if } n > 1.
    \end{dcases} \\[2ex]
    \tw(M_{1,E_n}) &= 
    \begin{dcases} 
        0 & \text{if } n = 1, \\
        1-2^{n-2} & \text{if } n > 1.
    \end{dcases}
\end{align*}
\end{proposition}
The proof of Proposition \ref{prop_genus_1_base} will be given in Appendix \ref{proof_genus_1}.
\begin{proof}[Proof of Proposition \ref{lem_chi_one}]
     A direct computation with $\chi_{E_n}(x) = x^n$ shows that:
     \begin{align*}
    -\frac{1}{12}\left.\frac{d}{dx}\chi_{E_n}(x)\right|_{x=0}
    + \frac{\chi_{E_n}(1)}{3}
    + \frac{\chi_{E_n}(2)}{4} 
    + \frac{\chi_{E_n}(3)}{9}
    - \frac{\chi_{E_n}(4)}{48} & \\
& \hspace{-5.2cm} =
\begin{dcases}
    -\frac{1}{12} + \frac{1}{3} + \frac{2}{4}
    + \frac{3}{9} - \frac{4}{48} = 1,
    & \text{if } n=1, \\
    0 + \frac{1}{3} + \frac{2^n}{4}
    + \frac{3^n}{9} - \frac{4^n}{48},
    & \text{if } n>1.
\end{dcases}
\end{align*}
\begin{align*}
-\frac{1}{2}\left.\frac{d}{dx}\chi_{E_n}(x)\right|_{x=0}
+ \chi_G(1)
- \frac{1}{4}\chi_{E_n}(2)
&=
\begin{dcases}
    -\frac{1}{2} + 1 - \frac{2}{4} = 0,
    & \text{if } n=1, \\
    1 - \frac{2^n}{4} = 1 - 2^{n-2},
    & \text{if } n>1.
\end{dcases}  
\end{align*}

    This computations match with Proposition \ref{prop_genus_1_base}. In addition, we know both the Euler characteristics and the top-weight Euler characteristics satisfy the same deletion--contraction relation (Proposition \ref{prop_invarient}) as $\chi_G(x)$, thus completes the proof.
\end{proof}

\subsubsection{Genus $g \geq 2$}\label{section:arb_genus}We compute the Euler characteristic and the top-weight Euler characteristic when $g \geq 2$ under certain chromatic assumptions on the graph.
\begin{definition}
    The \textit{chromatic number} of a graph $G$ is the smallest number of colours needed to colour the vertices of $G$ so that no two adjacent vertices share the same colour. We denote such number as $\chi(G)$.  
\end{definition}
\begin{remark}The following are some properties of $\chi(G)$ which can be easily deduced:
\begin{enumerate}
    \item $\chi(G) = \mathrm{min}\{r \geq 1 : N_{G,r} > 0\}$.
    \item The celebrated Four Colour Theorem is equivalent to saying that $\chi(G) \leq 4$ when $G$ is planar. 
\end{enumerate}
\end{remark}
The main result of this section is the following.
\begin{proposition}\label{lem_tw_g}
    Suppose $g \geq2$ and denote $B_k$ as the $k$-th Bernoulli number. 
    \begin{enumerate}
        \item If $\chi(G) > 2g+2$, then 
        \begin{align}\label{chi_Mg_G}
        \chi(M_{g,G}) = \frac{B_{2g}}{2g(2g-2)}\chi_G(2-2g).
        \end{align}
    \end{enumerate}
    \begin{enumerate}
        \item If $\chi(G) > g+1$, then 
        \begin{align}\label{chi_g_G}
        \tw(M_{g,G}) = \frac{B_g}{g(1-g)}\chi_G(1-g).
    \end{align}
    \end{enumerate} 
\end{proposition}
The strategy that will be used to prove Proposition \ref{lem_tw_g} is similar to \cite{Sidd_TopHass}. Recall Proposition~\ref{sidd} and deduce that:
\begin{align}
    \chi(M_{g,G}) &= \sum_{r=0}^n N_{G,r} \chi(M_{g,r})  \label{eq_sum_euler}\\ 
    \tw(M_{g,G})    &= \sum_{r=0}^n N_{G,r} \tw(M_{g,r}). \label{eq_sum_tw}
\end{align}
In general, $\chi(M_{g,r})$ and $\tw(M_{g,G})$ are known only through recursive formulas; see \cite{Euler_Mgn,Sn_Euler}. However, when the number of markings is sufficiently large, closed formulas can be obtained:
\begin{enumerate}
    \item It is well-known that a smooth curve with more then $2g + 2$ marked
    points is automorphism free (e.g. \cite[p. ~45]{ACGH_vol_1}). Thus, when $r > 2g+2$, the orbifold Euler characteristic of $\mathcal{M}_{g,r}$ equals to the ordinary Euler characteristic. The former is given by the celebrated Harer--Zagier formula \cite{HarerZagier1986}: 
    \begin{align}\label{closed_orb}
        \chi(M_{g,r}) = \chi_{\mathrm{orb}}(\mathcal{M}_{g,r}) =  (-1)^r (2g-1)\frac{(2g+r-3)!}{(2g)!}B_{2g}.
    \end{align}
    
    \item By \cite[Corollary 8.1]{Sn_Euler}, when $r > g+1$, we have
    \begin{align}\label{closed_tw}
    \tw(M_{g,r}) = (-1)^{r+1}\frac{(g+r-2)!}{g!}B_g.
\end{align}
\end{enumerate}
So we can try to simply substitute (\ref{closed_orb}), (\ref{closed_tw}) into (\ref{eq_sum_euler}), (\ref{eq_sum_tw}) respectively.

\begin{proof}[Proof of Proposition \ref{lem_tw_g}]
We provide the proof of \eqref{chi_Mg_G}; the proof of \eqref{chi_g_G} is completely analogous and is therefore omitted. First of all, since $\chi(G) > 2g+2$, we have $N_{G,r} = 0$ for $r \leq 2g+2$.
    Now substitute (\ref{closed_orb}) into (\ref{eq_sum_euler}) and obtain:
    \begin{align*}
        \chi(M_{g,G})  = 
        \sum_{r=0}^n N_{G,r} (-1)^r (2g-1)\frac{(2g+r-3)!}{(2g)!}B_{2g} .
    \end{align*}
    Recall the falling factorial:
    \begin{align*}
        (x)_r = x(x-1)\cdots (x-r+1).
    \end{align*}
    Substitute $x = 2-2g$ we get:
    \begin{align*}
        (2-2g)_r &= (2-2g)(1-2g) \cdots (2-2g-r+1) \\
            &= (-1)^r \frac{(2g+r-3)!}{(2g-3)!}.
    \end{align*}
    Substitute this back to $\chi(M_{g,G})$, we get
    \begin{align*}
        \chi(M_{g,G})  &= (2g-1)\frac{(2g-3)!}{(2g)!}B_{2g}
        \sum_{r=0}^n N_{G,r} (2-2g)_r \\
        & = \frac{B_{2g}}{2g(2g-2)}\chi_G(2-2g)
    \end{align*}
    where the last equality follows from the expression of the chromatic polynomial (\ref{chrom}) discussed earlier.
      \end{proof}

\appendix
\section{Some calculations on genus $1$ edgeless space} \label{proof_genus_1}
In the appendix we prove Proposition \ref{prop_genus_1_base}, which states that:
\begin{align}
    \chi(M_{1,E_n}) &= 
    \begin{dcases} \label{eq_chi_1_n}
        1 & \text{if } n = 1, \\
        2^{n-2} + 3^{n-2} - \frac{4^{n-2}}{3} + \frac{1}{3} \quad & \text{if } n > 1.
    \end{dcases} \\[2ex]
    \tw(M_{1,E_n}) &= 
    \begin{dcases} \label{eq_chi_tw_1_n}
        0 & \text{if } n = 1, \\
        1-2^{n-2} & \text{if } n > 1.
    \end{dcases}
\end{align}

\subsection{Euler characteristics}
\begin{proof}[Proof of (\ref{eq_chi_1_n})] It is clear that $\chi(M_{1,E_1}) = \chi(M_{1,1} \simeq \mathbb{C}) = 1 $. So we assume $n >1$. Notice that the fine moduli space $\mathcal{M}_{1,E_n}$ can be naturally identified with $\mathcal{E}^{n-1}$, the $(n-1)$-th fibre power of the universal curve  $\mathcal{E} \rightarrow \mathcal{M}_{1,1}$. Thus induces a map on the coarse spaces:
\begin{align*}
    \varphi \colon M_{1,E_n} \longrightarrow M_{1,1} \simeq \mathbb{C}_j;
\end{align*}
such that the fibre over $[E,\mathcal{O}]$ is $E^{n-1}/\mathrm{Aut}(E,\mathcal{O})$. It is well-known that:
\begin{align*}
\mathrm{Aut}(E,\mathcal{O}) =
    \begin{cases}
        \mu_6, \quad\text{if } j =0 \\
        \mu_4, \quad\text{if } j =1728 \\
        \mu_2, \quad\text{else.}
    \end{cases}
\end{align*}
Therefore, we can stratify $\varphi$ into fibrations and deduce
\begin{align}\label{eq_chi_quotient}
    \chi(M_{1,E_n}) &= \chi(\mathbb{C}_j \setminus \{0,1728\}) \chi(E^{n-1}/\mu_2) + \chi(E^{n-1}/\mu_4) + \chi(E^{n-1}/\mu_6) \\
    &= -\chi(E^{n-1}/\mu_2) + \chi(E^{n-1}/\mu_4) + \chi(E^{n-1}/\mu_6).\nonumber
\end{align}
To compute the Euler characteristics of these quotient spaces, we use the following standard fixed-point formula for a finite group $G$ acting smoothly on a compact manifold $X$; see, for example, \cite{Hirzebruch_EulerChar}:
\begin{align*}
    \chi(X/G)
    =
    \frac{1}{|G|}
    \sum_{g \in G}
    \chi(X^g),
\end{align*}
where $X^g$ denotes the fixed-point locus of $g$. In our setting, the action of the automorphism group on $E^{n-1}$ is diagonal, so we have $(E^{n-1})^{g} = (E^{g})^{n-1}$.
\begin{enumerate}
    \item In the generic case, we can identify $\mu_2$ with $\{1,-1\}$, where $-1$ acts via involution: $P \mapsto -P$. We deduce that:
    \begin{align*}
        E^{1} &= \{P \in E \; | \; P = P\} = E;\\
        E^{-1} &= \{P \in E \; | \; P = -P\} = E[2] \simeq \mathbb{Z}/{2\mathbb{Z}} \times \mathbb{Z}/{2\mathbb{Z}}.
    \end{align*} 
    Thus we have, 
    \begin{align*}
        \chi(E^{n-1}/\mu_2) = \frac{\chi(E^1)^{n-1}+ \chi(E^{-1})^{n-1}}{2} = \frac{0+4^{n-1}}{2}.
    \end{align*}
    \item When $j = 1728$, we can identify $E \simeq \mathbb{C}/\Lambda$, where $\Lambda = \mathbb{Z} \oplus \mathbb{Z}i $ and $\mu_4 = \{\pm 1, \pm i\}$ acts by multiplication. So we have:
    \[ E^1 = E; \quad E^{-1} = E[2];\]
    \[E^{i} = \{z \in \mathbb{C}/\Lambda \; | \; iz = z\} = \{\mathcal{O}, \frac{1+i}{2}\}; \quad  E^{-i} = \{z \in \mathbb{C}/\Lambda \; | \; -iz = z\} = \{\mathcal{O}, \frac{1+i}{2}\}.\]
    Thus we have,
    \begin{align*}
        \chi(E^{n-1}/\mu_4) =\frac{\chi(E^1)^{n-1}+ \chi(E^{-1})^{n-1} + \chi(E^{i})^{n-1}+\chi(E^{-i})^{n-1}}{4} = \frac{0+4^{n-1}+2\cdot 2^{n-1} }{4}.
    \end{align*}
    \item  When $j = 0$, we can identify $E \simeq \mathbb{C}/\Lambda$, where $\Lambda = \mathbb{Z} \oplus \mathbb{Z}\omega $ and $\omega = e^{2\pi i/3}$. In this case $\mu_6 = \langle \zeta \rangle$, where $\zeta = e^{\pi i/3}$. One can easily show that:
    \begin{align*}
        E^1 = E; \quad E^\zeta = E^{\zeta^5} =\{\mathcal{O}\}; \quad E^{\zeta^2} = E^{\zeta^4} =\{\mathcal{O}, \frac{1-\omega}{3},\frac{2(1-\omega)}{3}\}; \quad E^{\zeta^3} = E[2].
    \end{align*}
    Thus we have,
    \begin{align*}
        \chi(E^{n-1}/\mu_6) = \frac{0+2\cdot 1+2 \cdot 3^{n-1}+4^{n-1}}{6}.
    \end{align*}
\end{enumerate}
Combining these pieces into (\ref{eq_chi_quotient}), we obtain the desired expression:
\begin{align*}
    \chi(M_{1,E_n}) &= - \frac{4^{n-1}}{2} + \frac{4^{n-1}+2\cdot 2^{n-1} }{4} + \frac{2\cdot 1+2 \cdot 3^{n-1}+4^{n-1}}{6} \\
    & = -2 \cdot 4^{n-2} + 4^{n-2}+2^{n-2}+ \frac{1}{3}+ 3^{n-2}+ \frac{2}{3} \cdot 4^{n-2} \\
    & = 2^{n-2} + 3^{n-2} - \frac{4^{n-2}}{3} + \frac{1}{3}. \qedhere
\end{align*}
\end{proof}

\subsection{Top-weight Euler characteristics} We first prove a combinatorial identity.
\begin{claim}\label{claim_An} Let $S(n,k)$ be the Stirling numbers of the second kind, we have
    \begin{align*}
        A_n : = \sum_{k = 1}^{n} (-1)^{k-1} (k-1)! S(n,k) = \begin{cases}
    1 & \text{if } n = 1, \\
    0 & \text{if } n > 1.
    \end{cases}
    \end{align*}
\end{claim}
\begin{proof}
    Recall the exponential generating series of $S(n,k)$:
    \begin{align}\label{generate_stirling}
         \sum_{n \geq k} S(n,k)  \frac{x^n}{n!}  = \frac{(e^x-1)^k}{k!}.
    \end{align}
    Now consider the generating series of $A_n$: 
    \begin{align*}
        \sum_{n \geq 1} A_n \frac{x^n}{n!} &=  \sum_{n \geq 1} \sum_{k = 1}^{n} (-1)^{k-1} (k-1)! S(n,k)  \frac{x^n}{n!} 
        \intertext{Swap the sums:}
        & = \sum_{k \geq 1}(-1)^{k-1} (k-1)!  \sum_{n \geq k} S(n,k)  \frac{x^n}{n!} 
        \intertext{Substitute (\ref{generate_stirling}) into the equation we get:}
         & = \sum_{k \geq 1} (-1)^{k-1} (k-1)! \frac{(e^x-1)^k}{k!}\\
          & = \sum_{k \geq  1}\frac{(-1)^{k-1}}{k}(e^x-1)^k.
    \end{align*}
    Now recall that 
    \begin{align*}
        \sum_{k \geq 1} \frac{(-1)^{k-1}}{k} y^k =  \mathrm{log}(1+y).
    \end{align*}
    After a change of variable $y = e^x-1$, we have:
    \begin{align*}
        \sum_{n \geq 1} A_n \frac{x^n}{n!} =  \mathrm{log}(1+e^x-1) = x.
    \end{align*}
    Therefore, $A_1 = 1$ and $A_n = 0$ for $n > 1$.
\end{proof}

\begin{proof}[Proof of (\ref{eq_chi_tw_1_n})] By the stability condition, $E_n$-stable curves cannot possess any rational tails. So we can enumerate all the combinatorial types of genus-$1$ $E_n$-stable graphs easily:
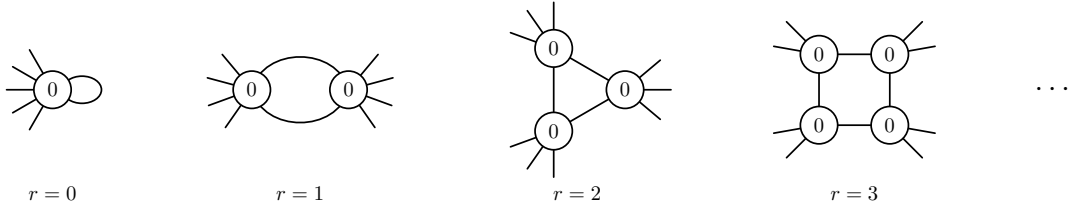
\begin{center}

\def\leglen{0.78}   
\def\graphrot{90}   

\begin{tikzpicture}[
    scale=0.78,
    transform shape,
    every path/.style={line width=0.65pt, line cap=round, line join=round},
    v/.style={
        circle,
        draw,
        fill=white,
        inner sep=0pt,
        minimum size=18pt,
        font=\small,
        rotate=-\graphrot
    }
]

\newcommand{\leg}[2]{\draw (#1) -- ++(#2:\leglen);}

\begin{scope}[shift={(0,0)}]
\begin{scope}[rotate=\graphrot]
    \coordinate (a) at (0,0);

    \draw (0,-0.50) ellipse [x radius=0.23, y radius=0.33];

    \foreach \ang in {150,120,90,60,30}
        \leg{a}{\ang};

    \node[v] at (a) {$0$};
\end{scope}

\node[font=\small] at (0,-1.75) {$r=0$};
\end{scope}

\begin{scope}[shift={(4.2,0)}]
\begin{scope}[rotate=\graphrot]
    \coordinate (b1) at (0,0.82);
    \coordinate (b2) at (0,-0.82);

    \draw (b1) to[out=-165, in=165, looseness=1.20] (b2);
    \draw (b1) to[out=-15,  in=15,  looseness=1.20] (b2);

    \foreach \ang in {145,110,75,40}
        \leg{b1}{\ang};
    \foreach \ang in {215,250,285,320}
        \leg{b2}{\ang};

    \node[v] at (b1) {$0$};
    \node[v] at (b2) {$0$};
\end{scope}

\node[font=\small] at (0,-1.75) {$r=1$};
\end{scope}

\begin{scope}[shift={(8.9,0)}]
\begin{scope}[rotate=\graphrot]
    \coordinate (c1) at (0.70,0.40);
    \coordinate (c2) at (-0.70,0.40);
    \coordinate (c3) at (0,-0.81);

    \draw (c1) -- (c2);
    \draw (c2) -- (c3);
    \draw (c3) -- (c1);

    \foreach \ang in {0,35,70}
        \leg{c1}{\ang};
    \foreach \ang in {110,145,180}
        \leg{c2}{\ang};
    \foreach \ang in {230,270,310}
        \leg{c3}{\ang};

    \node[v] at (c1) {$0$};
    \node[v] at (c2) {$0$};
    \node[v] at (c3) {$0$};
\end{scope}

\node[font=\small] at (0,-1.75) {$r=2$};
\end{scope}

\begin{scope}[shift={(13.6,0)}]
\begin{scope}[rotate=\graphrot]
    \coordinate (d1) at (-0.60,0.60);
    \coordinate (d2) at (0.60,0.60);
    \coordinate (d3) at (-0.60,-0.60);
    \coordinate (d4) at (0.60,-0.60);

    \draw (d1) -- (d2);
    \draw (d2) -- (d4);
    \draw (d4) -- (d3);
    \draw (d3) -- (d1);

    \foreach \ang in {135,100}
        \leg{d1}{\ang};
    \foreach \ang in {45,80}
        \leg{d2}{\ang};
    \foreach \ang in {225,260}
        \leg{d3}{\ang};
    \foreach \ang in {-45,-80}
        \leg{d4}{\ang};

    \node[v] at (d1) {$0$};
    \node[v] at (d2) {$0$};
    \node[v] at (d3) {$0$};
    \node[v] at (d4) {$0$};
\end{scope}

\node[font=\small] at (0,-1.75) {$r=3$};
\end{scope}

\node[font=\Large] at (17,0) {$\cdots$};

\end{tikzpicture}
\end{center}

For $0 \leq r \leq n-1$, let $\Gamma_r$ denote the combinatorial type with
$r+1$ vertices and $r+1$ edges. Since $\Gamma_r$ has $r+1$ internal
edges, its associated simplex in $\Delta_{1,E_n}$ has dimension $r$. A marking distribution on $\Gamma_r$ is a map
\[
m \colon [n] \longrightarrow V(\Gamma_r),
\]
which assigns to each labelled leg the vertex to which it is attached.
The stability condition requires every vertex to carry at least one
labelled leg, so $m$ is necessarily surjective. In particular, the number
of vertices is at most $n$, which explains the bound $r \leq n-1$. Define the automorphism group of the marked graph $(\Gamma_r,m)$ by
\[\mathrm{Aut}(\Gamma_r,m):= \{\phi \in \mathrm{Aut}(\Gamma_r):\phi(m(i)) = m(i) \}.\]
Since $m$ is surjective, every automorphism
$\phi \in \operatorname{Aut}(\Gamma_r,m)$ fixes every vertex of
$\Gamma_r$.

For $r \geq 2$, every edge of $\Gamma_r$ is uniquely determined by its
two endpoints. Therefore, once all the vertices are fixed, all the edges
are fixed as well, and consequently $\mathrm{Aut}(\Gamma_r,m)=\{\mathrm{id}\}$. For $r=0$, there is an
involution exchanging the two half-edges of the loop. Thus $\mathrm{Aut}(\Gamma_0,m)
\simeq
\mathbb{Z}/2\mathbb{Z}.$
However, this involution fixes the unique internal edge and therefore acts
trivially on the associated $0$-simplex. Hence $\Gamma_0$ contributes a
single $0$-cell to the boundary complex.

For $r=1$, the two vertices are fixed by the markings, but the two
parallel edges may still be exchanged. Therefore,
\[
\mathrm{Aut}(\Gamma_1,m)
\simeq
\mathbb{Z}/2\mathbb{Z}.
\]
Write the normalized edge lengths as $(\ell_1,\ell_2)=(t,1-t)$, where $0 \leq t \leq 1$. 
Then the non-trivial automorphism $\tau$ acts by
\[\tau \colon t \longmapsto 1-t.\]
Moreover, the two endpoints $t=0$ and $t=1$ correspond to contracting one
of the two parallel edges, and both contractions yield the graph
$\Gamma_0$. Thus, before quotienting by the edge-exchanging involution,
the two endpoints are identified and the resulting space is $S^1$.
The involution acts on this $S^1$ by reflection, and therefore
\[
S^1/\langle \tau\rangle \simeq [0,1].
\]
 It follows that the $1$-skeleton
$\Delta_{1,E_n}^{(1)}$ is a star hence
contractible, and the quotient map
\[\Delta_{1,E_n} \longrightarrow \Delta_{1,E_n}/\Delta_{1,E_n}^{(1)}\]
is a homotopy equivalence hence preserves the Euler characteristics. 

It remains to count the $r$-cells for $r \geq 2$. Temporarily label the
$r+1$ vertices of $\Gamma_r$ by the elements of $[r+1]$. A stable marking
distribution is then a surjection
\[
m \colon [n] \twoheadrightarrow [r+1].
\]
The number of such surjections is $(r+1)!S(n,r+1)$, as the Stirling number of the second kind $S(n,r+1)$ counts the number of ways of distributing $n$ balls into $r+1$ unordered, non-empty baskets. 
The dihedral group $\mathrm{Dih}_{r+1} \simeq \mathrm{Aut}(\Gamma_r)$
acts on these marking distributions.
The stabilizer of a marking distribution $m$ under this action is exactly
$\operatorname{Aut}(\Gamma_r,m)$ which is
trivial for $r \geq 2$, so the action of $\operatorname{Dih}_{r+1}$ is free. Therefore,
the number $c_r$ of $r$-cells of $\Delta_{1,E_n}/\Delta_{1,E_n}^{(1)}$ is
\[
c_0 = 1, \, c_1 = 0 \text{ and }
c_r =
\frac{(r+1)!}{\lvert \operatorname{Dih}_{r+1} \rvert}
S(n,r+1), \quad r \geq 2 .
\]
Therefore, taking the alternating sum over $c_i$, we obtain:
\begin{align*}
    \chi(\Delta_{1,E_n}) =\chi(\Delta_{1,E_n}/\Delta_{1,E_n}^{(1)}) &= 1 - 0 +  \sum_{r = 2}^{n-1} (-1)^r \frac{(r+1)!}{|\mathrm{Dih}_{r+1}|}S(n,r+1).\\ 
    \intertext{Recall $|\mathrm{Dih}_{r+1}| = 2(r+1)$ and shift the summation index by $1$, we have: }
    \chi(\Delta_{1,E_n})  & = 1 + \frac{1}{2}\sum_{k = 3}^{n} (-1)^{k-1} (k-1)! S(n,k) \\ 
    & = 1 + \frac{1}{2}A_n - \frac{1}{2} + \frac{1}{2} S(n,2)
\end{align*}
where $A_n$ was defined and computed in Claim \ref{claim_An}. Hence, we conclude:
\begin{align*}
    \chi(\Delta_{1,E_n})  = \begin{cases}
1 + \frac{1}{2} - \frac{1}{2} + \frac{1}{2} S(1,2) = 1 & \text{if } n = 1, \\[1ex]
1 + 0 - \frac{1}{2} + \frac{1}{2} S(n,2) = 2^{n-2}   & \text{if } n > 1.
\end{cases}
\end{align*}
Recall from (\ref{def_tw_g}) that $\tw(M_{1,E_n})= - \Tilde{\chi}(\Delta_{1,E_n}) = 1 - \chi(\Delta_{1,E_n})$, so we obtain the desired expression:
\begin{align*}
    \tw(M_{1,E_n})
    &=
    \begin{cases}
        0, & \text{if } n=1, \\
        1-2^{n-2}, & \text{if } n>1.
    \end{cases}
    \qedhere
\end{align*}
\end{proof}

\bibliographystyle{alpha}
\bibliography{delcon}\medskip

Andy Xiaoan Yang \> \textit{Queen Mary University of London} \> \href{mailto:xiaoan.yang@qmul.ac.uk}{xiaoan.yang@qmul.ac.uk}

\end{document}